\documentclass[a4paper,oneside, reqno]{amsart} 
\usepackage{tikz} 
\usepackage{bm} 
\usepackage{listings} 
\usepackage{enumitem} 
\usepackage{wrapfig} 
\usepackage{subcaption} 
\usepackage{float} 
\usepackage{amssymb, amsthm} 
\usepackage{amsmath, mathtools}

\usepackage{todonotes} 

\usetikzlibrary{intersections}
\usetikzlibrary{decorations.markings}

\let\oldchi\chi 
\renewcommand*{\chi}{\protect\raisebox{2pt}{$\oldchi$}} 

\newcommand{\cone}{C}

\DeclareMathOperator{\tor}{Tor}
\DeclareMathOperator{\st}{st} 
\DeclareMathOperator{\stellar}{ss} 
\DeclareMathOperator{\sd}{sd} 

\DeclareMathOperator{\link}{lk}
\DeclareMathOperator{\Hom}{Hom} 
\DeclareMathOperator{\myempty}{\emptyset} 
\DeclareMathOperator{\suchthat}{\ | \ }
\DeclareMathOperator{\mydeg}{\overline{\deg}} 
\newcommand{\doubleoverline}[1]{{\overline{\overline{#1}}}} 

\newtheorem{prop}{Proposition}[section] 
\newtheorem{lemma}[prop]{Lemma} 
 
\newtheorem{theorem}[prop]{Theorem} 

\newtheorem{cor}[prop]{Corollary}

\theoremstyle{remark} 
\newtheorem{remark}[prop]{Remark} 
\theoremstyle{definition} 
\newtheorem{defn}[prop]{Definition} 

\theoremstyle{definition} 
\newtheorem{exmp}[prop]{Example} 
\theoremstyle{definition} 
\newtheorem{cons}[prop]{Construction} 

\numberwithin{equation}{section}

\definecolor{colour1}{RGB}{34,167,240} 
\definecolor{colour2}{RGB}{240,206,34} 
\definecolor{colour3}{RGB}{240,34,99}  

\title{Non-trivial higher Massey products in moment-angle complexes} 

\author{Jelena Grbi\'{c}}
\address{School of Mathematical Sciences, University of Southampton, UK}
\email{J.Grbic@soton.ac.uk}

\author{Abigail Linton}
\address{School of Mathematical Sciences, University of Southampton, UK}
\email{A.Linton@soton.ac.uk}

\begin{document}

	\begin{abstract}
		As part of various obstruction theories, non-trivial Massey products have been studied in symplectic and complex geometry, commutative algebra and topology for a long time. 
		We introduce a general approach to constructing non-trivial Massey products in the cohomology of moment-angle complexes, using homotopy theoretical and combinatorial methods.
		Our approach sets a unifying way of constructing higher Massey products of arbitrary cohomological classes and generalises all existing examples of non-trivial Massey products in moment-angle complexes.
		As a result, we obtain explicit constructions of infinitely many non-formal manifolds that appear in topology, complex geometry and algebraic geometry.
	\end{abstract}
	\maketitle 
	
	\section{Introduction}
	
	A moment-angle complex $\mathcal{Z}_{\mathcal{K}}$ over a simplicial complex $\mathcal{K}$ on $m$ vertices is built from ordered products of discs and circles in $\mathbb{C}^m$ that are glued together along the face category of $\mathcal{K}$. 
	The coordinate $T^m$-action on $\mathbb{C}^m$ descends to a natural $T^m$ action on moment-angle complexes.
	If $\mathcal{K}$ is a triangulation of a sphere, the moment-angle complex $\mathcal{Z}_\mathcal{K}$ is a manifold that admits a smooth complex structure.
	These manifolds also generalise many well-known smooth complex manifolds such as Hopf and Calabi-Eckmann manifolds. 
	
	Massey products are higher operations in the homology of differential graded algebras.
	In the context of commutative algebra, supposing $\mathbf{k}$ is a field or $\mathbb{Z}$, the cohomology algebra of $\mathcal{Z}_\mathcal{K}$ is isomorphic to the Tor-algebra $\tor_{\mathbf{k}[m]}(\mathbf{k}[\mathcal{K}], \mathbf{k})$ of the face ring $\mathbf{k}[\mathcal{K}]$, due to \cite{TorusActionsEMProof} and \cite[Theorem~1]{BaskakovProduct}.
	The face ring $\mathbf{k}[\mathcal{K}]$ is Golod if all Massey products in $\tor_{\mathbf{k}[m]}(\mathbf{k}[\mathcal{K}], \mathbf{k})$ vanish.
	Hence, Massey products in $\mathcal{Z}_\mathcal{K}$ are obstructions to Golodness of $\mathbf{k}[\mathcal{K}]$.
	From the perspective of complex geometry, by identifying $\mathcal{Z}_\mathcal{K}$ with
	the complement $U(\mathcal{K})$ of a coordinate subspace arrangement 
	corresponding to $\mathcal{K}$, moment-angle complexes are LVM manifolds~\cite{BosioMeersseman,PanovUstinovsky} 
	when $\mathcal{K}$ is the boundary of the dual of a simple polytope.
	Massey products are obstructions to the formality of these manifolds.
	The combinatorial approach to Massey products in moment-angle complexes has been used to prove cohomological rigidity of L\"obell manifolds~\cite{rigidity}, which are built from 3-dimensional polytopes in the Pogorelov class.
	However, currently, most known examples of Massey products in moment-angle complexes are sporadic due to how difficult they are to calculate explicitly.	
	
	The first non-trivial Massey products in moment-angle complexes were discovered by  Baskakov~\cite{Baskakov}, who constructed an infinite family of triple Massey products. 
	Limonchenko~\cite{Limonchenko} constructed the first family of non-trivial $n$-Massey products for $n\geqslant 2$ on lowest-degree classes in moment-angle complexes.
	Families of non-trivial Massey products in moment-angle complexes associated to special geometric direct families of $2$-truncated cubes (flag nestohedra) are due to Buchstaber and Limonchenko \cite{directfamilies_BL}, who also applied these families to the differentials in Eilenberg-Moore and Milnor spectral sequences. 
	In \cite{Limonchenko_multiwedge}, Limonchenko constructs non-trivial higher Massey products in highly-connected moment-angle complexes by using the simplicial multiwedge operation (or J-construction), which takes a simplicial complex and builds a new one that has the same combinatorial structure as the original.

	Using combinatorics and homotopy theory, we give the first systematic and unifying approach for constructing non-trivial Massey products in the cohomology of moment-angle complexes.
	We show that the combinatorics of $\mathcal{K}$ encodes Massey products.
	By doing this, we expose some of the structural behaviour of Massey products with respect to combinatorial operations, and spark the ability to construct concrete examples of non-trivial Massey products in commutative algebra, complex geometry and combinatorics, as well as toric topology.
	
	Our starting point
	is the cup product, which is a $2$-Massey product.
	The categorical product of simplicial complexes is the join, which is mirrored by the product of moment-angle complexes $\mathcal{Z}_{\mathcal{K}_1*\mathcal{K}_2}=\mathcal{Z}_{\mathcal{K}_1} \times \mathcal{Z}_{\mathcal{K}_2}$ and  the existence of a non-trivial cup product in the cohomology of  $\mathcal{Z}_{\mathcal{K}_1*\mathcal{K}_2}$.
	Unlike cup products, Massey products are higher operations so certain $(n-1)$-Massey products must be trivial in order to define $n$-Massey products.
	
	There is a classification result for $3$-Massey products of cohomological classes in lowest degree in moment-angle complexes~\cite[Theorem 6.1.1]{DenhamSuciu},~\cite{LowestDegreeClassification}, but it vitally relies on the fact that the lowest degree classes are represented combinatorially by cycles in the  $1$-skeleton of $\mathcal{K}$.
	This technique does not generalise to higher dimensions since it is unknown how to combinatorially realise an arbitrary $n$-cycle. So far there has not been a systematic way to construct triple Massey, or any $n$-Massey, products of higher dimensional classes.
	We give two constructions that address these drawbacks.
	
	In Construction~\ref{cons: simplicial complex for n-Massey}, to construct non-trivial $n$-Massey products in moment-angle complexes, we start with the join of $n$ simplicial complexes $\mathcal{K}_i$. To trivialise the lower Massey products, we systematically remove  certain simplices from the join by an operation called star deletion and call the constructed simplicial complex $\mathcal{K}$.
	We show that $\mathcal{Z}_\mathcal{K}$ has a non-trivial $n$-Massey product in Theorem~\ref{thm: joins}. 
	It is important to emphasise that we do not impose any restrictions on $n$-arity of these Massey products, on the choice of simplicial complexes $\mathcal{K}_i$ for any $i$, nor on the dimension of classes in the Massey product.
	This construction generalises Baskakov's~\cite{Baskakov} family of non-trivial triple Massey products in the cohomology of moment-angle complexes, taking triangulations of spheres for $\mathcal{K}_1, \mathcal{K}_2$ and $\mathcal{K}_3$.
	Also it generalises Limonchenko's~\cite[Theorem 2]{Limonchenko} family of $n$-Massey products, which are built by removing simplices from the join of $n$ $0$-spheres.
	
	Notably, our construction produces the first examples of non-trivial Massey products on torsion classes, as well as examples with non-trivial indeterminacy. Such an example is constructed by star deleting simplices in the join of the projective plane $\mathbb{R}P^2$ and two copies of the $0$-sphere, as illustrated in Example~\ref{ex: joins torsion example}.
	We also create the first infinite families of higher Massey products with non-trivial indeterminacy in moment-angle complexes, on arbitrary cohomological classes, by extending our construction in Section~\ref{sec: infinite families non-trivial indeterminacy}.
	
	The topological properties and homotopy type of $\mathcal{K}$ do not determine the topology of the moment-angle complex $\mathcal{Z}_\mathcal{K}$. 
	However, unexpectedly, in Construction~\ref{cons: edge contraction} we deform $\mathcal{K}$ up to homotopy to create a new simplicial complex $\mathcal{L}$
	such that $\mathcal Z_{\mathcal{L}}$ has an explicitly constructed $n$-Massey product if $\mathcal Z_{\mathcal{K}}$ has an $n$-Massey product. Crucially, the Massey product in $\mathcal{Z}_\mathcal{L}$ can be of different dimensional cohomological classes to those in $\mathcal{Z}_\mathcal{K}$. In this construction, the simplicial complex $\mathcal L$  has the same homotopy type as $\mathcal K$ and is obtained by systematically ``stretching'' certain simplices of $\mathcal K$. 
	In Theorem~\ref{thm: edge contractions Massey}, we show that these Massey products are non-trivial, even if they have non-trivial indeterminacy. 
	
	Consequently,  we can construct infinite families of non-trivial Massey products from known examples by ``stretching" simplices  in a controlled way.
	For example, from each of the obstruction graphs  in the classification of lowest-degree triple Massey products in moment-angle complexes \cite{DenhamSuciu, LowestDegreeClassification}, we obtain infinite families of non-trivial triple Massey products of higher dimensional classes.
	We give an alternative proof of known examples of non-trivial triple Massey products in moment-angle manifolds, such as those associated with Pogorelov polytopes~\cite{Liz} and permutahedra or stellohedra~\cite{LimonchenkoFlagNestohedra, Limonchenko_multiwedge} using ``stretched" obstruction graphs.
	Also, the two constructions, Constructions~\ref{cons: simplicial complex for n-Massey} and~\ref{cons: edge contraction}, can be combined to create new higher Massey products. We use this to create $k$-Massey products in moment-angle manifolds associated with $n$-dimensional permutahedra and stellohedra for every $k<n$, including Massey products with non-trivial indeterminacy.
	
	Even though it has been known for decades that Massey products are important obstructions in many fields, we have the first general methods to calculate and construct $n$-Massey products of classes in any degree, for any $n$, including Massey products with non-trivial indeterminacy.
	The first infinite family of examples of non-formal spaces or non-Golod face rings were constructed by Limonchenko~\cite[Theorem~4.10]{Limonchenko_multiwedge} using moment-angle complexes associated to graph associahedra. There are other explicit families constructed in \cite{Limonchenko}, \cite{Limonchenko_multiwedge} and \cite{directfamilies_BL}.
	More generally, our framework constructs infinitely many families of such examples, confirming that non-trivial higher Massey products are much more common in moment-angle complexes and moment-angle manifolds than previously thought.
	
	Furthermore, these techniques do not just apply to moment-angle complexes. We study Massey products in moment-angle complexes via combinatorics; 
	one key fact to do this is that the cohomology of $\mathcal{Z}_\mathcal{K}$ decomposes into a direct sum of cohomology groups of full subcomplexes of $\mathcal{K}$~\cite[Theorem 1]{BaskakovProduct}.
	For a topological pair $(X,A)$, a polyhedral product $(X,A)^\mathcal{K}$ is a generalisation of a moment-angle complex since   $\mathcal{Z}_\mathcal{K}=(D^2, S^1)^\mathcal{K}$. 
	In the case of a topological space $A$ and its cone $\cone A$, Bahri, Bendersky, Cohen and Gitler \cite[Theorem 1.12]{BBCG_CupProducts} showed that the cohomology of $(\cone A, A)^\mathcal{K}$ also decomposes in terms of $H^*(A)$ and the cohomology of full subcomplexes of $\mathcal{K}$ when $H^*(A)$ satisfies the strong K\"unneth formula. 
	Using this decomposition and our constructions, it is possible to produce non-trivial Massey products in $(\cone A, A)^\mathcal{K}$ by incorporating cohomological classes of $A$ to the classes we construct in the cohomology of full subcomplexes of $\mathcal{K}$ in order to create Massey products in $\mathcal{Z}_\mathcal{K}$.

	\section{Preliminaries}
	\subsection{Moment-angle complexes}
	Let $\mathcal{K}$ be a simplicial complex on the vertex set $[m]=\{1, \ldots, m\}$. The \textit{moment-angle complex} $\mathcal{Z}_\mathcal{K}$ \cite[Definition 3.2.1]{TorusActionsCombinatorics} is 
	\[
	\mathcal{Z}_\mathcal{K}=\bigcup\limits_{\sigma \in \mathcal{K}} \left( D^2, S^1 \right)^\sigma \subset (D^2)^m
	\]
	where 
	$
	(D^2, S^1)^\sigma = 
	\prod_{i=1}^{m} Y_i
	$
	for $Y_i=D^2$ if $i\in \sigma$, and $Y_i=S^1$ if~$i\notin \sigma$.
	A moment-angle complex $\mathcal{Z}_\mathcal{K}$ is a manifold if $\mathcal{K}$ is a triangulation of a sphere.
	
	In this paper, all coefficients are in $\mathbf{k}$, which is a field or $\mathbb{Z}$. 
	As a subspace of the polydisc, $\mathcal{Z}_\mathcal{K}$ has a cellular decomposition that induces a multigrading on the cellular cochain groups $C^*(\mathcal{Z}_\mathcal{K})$. 
	For $J\subset [m]$, the \textit{full subcomplex} $\mathcal{K}_J$ is $\{\sigma\in \mathcal{K} \ | \ \sigma\subset J\}$.
	Let $\widetilde{C}^{*}(\mathcal{K}_J)$ be the augmented simplicial cochain complex.
	The cohomology ring of $\mathcal{Z}_\mathcal{K}$ can be expressed in combinatorial terms.
	
	\begin{theorem}\cite{BaskakovProduct} \label{thm: full Hochster's}
		There is an isomorphism of cochains \[\widetilde{C}^{*-1}(\mathcal{K}_J)\to C^{*-|J|, 2J}(\mathcal{Z}_\mathcal{K})\subset C^{*+|J|}(\mathcal{Z}_\mathcal{K})\]
		that induces an isomorphism of algebras 
		\begin{equation}\label{eq: hochster}
			H^*(\mathcal{Z}_\mathcal{K}) \cong \bigoplus_{J \subset [m]} \widetilde{H}^*(\mathcal{K}_J)
		\end{equation}
		where $\widetilde{H}^{-1}(\mathcal{K}_{\varnothing})=\mathbf{k}$.
	\end{theorem}
	
	We refer to the cohomology decomposition \eqref{eq: hochster} as Hochster's formula~\cite{Hochster}.
	Let $C_p(K_J)$ be simplicial chain complex for $\mathcal{K}_J$.
	The cochain group $C^{p}(\mathcal{K}_J)=\Hom (C_p(\mathcal{K}_J), \mathbf{k})$ has a basis of $\chi_L$ for a $p$-simplex $L \in \mathcal{K}_J$,  where $\chi_L$ takes the value $1$ on $L$ and $0$ otherwise. 
	A subset $J\subset [m]$ has an order inherited from $[m]$.
	If $j$ is the $r$th element of $J$, define
	\begin{equation}\label{eq: varepsilon}
		\varepsilon (j, J)=(-1)^{r-1}
	\end{equation}
	and for $L\subset J$, define
	$
	\varepsilon(L, J)=\prod_{j\in L} \varepsilon (j, J)$.
	For simplices $L=\{l_1, \ldots, l_p\}$, $M=\{m_1, \ldots, m_q\}$, we denote $\{l_1, \ldots, l_p, m_1, \ldots, m_q\}$ by $L\cup M$.
	The product on $\bigoplus_{J\subset [m]} \widetilde{H}^*(\mathcal{K}_J)$ is induced by $C^{p-1} (\mathcal{K}_I) \otimes C^{q-1} (\mathcal{K}_J) \to C^{p+q-1} (\mathcal{K}_{I\cup J})$,
	\begin{equation} \label{eq: cochain product}
		\chi_L \otimes \chi_M \mapsto 
		\begin{cases}
			c_{L\cup M}\; \chi_{L\cup M} & \text{if } I\cap J = \varnothing,  \\
			0 & \text{otherwise} 
		\end{cases}
	\end{equation}
	where
	$
	c_{L\cup M} = \varepsilon(L,I) \;\varepsilon(M,J) \;\zeta\; \varepsilon(L\cup M, I\cup J)
	$
	and $
	\zeta= \prod_{k\in I\setminus L} \varepsilon(k,k\cup J\setminus M)$.
	
	For a cochain $a\in C^p(\mathcal{K}_J)$, let the \textit{support of $a$} be the set $S_a$ of $p$-simplices $\sigma\in \mathcal{K}_J$ such that
	$
	a= \sum_{\sigma \in S_a} a_{\sigma} \chi_{\sigma}
	$
	for a nontrivial coefficient $a_{\sigma}\in \mathbf{k}$. For a cohomology class $\alpha \in \widetilde{H}^p(\mathcal{K}_J)$, we say that $\alpha$ is \textit{supported on} $\mathcal{K}_J$.
	
	\begin{lemma} \label{lem: multiplication of general cochains}
		For a simplicial complex $\mathcal{K}$, let $a \in C^{p}(\mathcal{K}_I)$ and $b\in C^{q}(\mathcal{K}_J)$. Let the order of vertices in $\mathcal{K}$ be such that $i<j$ for every $i\in I$ and $j\in J$. Suppose that
		$
		a= \sum_{\sigma\in S_a} a_{\sigma} \chi_{\sigma}$ and $b= \sum_{\tau\in S_b} b_{\tau} \chi_{\tau}
		$
		for $p$-simplices
		$\sigma \in S_a\subset \mathcal{K}_I$, $q$-simplices $\tau \in S_b\subset \mathcal{K}_J$ and coefficients $a_{\sigma}, b_{\tau}\in \mathbf{k}$.  Then the product $ab\in C^{p+q+1}(\mathcal{K}_{I\cup J})$ is given by 
		\[
		ab= (-1)^{|I|(q+1)} \sum_{\sigma\in S_a} \sum_{\tau\in S_b} a_{\sigma} b_{\tau} \chi_{\sigma \cup \tau}.
		\]
	\end{lemma}
	
	\begin{proof}
		The product $ab$ is given by
		\begin{align*}
			ab &=\left( \sum_{\sigma\in S_a} a_{\sigma} \chi_{\sigma} \right) \left( \sum_{\tau\in S_b} b_{\tau} \chi_{\tau} \right) \\
			&= \sum_{\sigma\in S_a} \sum_{\tau\in S_b} a_{\sigma}\; b_{\tau} \;\varepsilon(\sigma, I) \;\varepsilon(\tau, J)\; \zeta \;\varepsilon(\sigma \cup \tau, I \cup J)\; \chi_{\sigma \cup \tau}
		\end{align*}
		where $\zeta=1$ since all vertices of $I$ are ordered before vertices of $J$ in $\mathcal{K}$.
		
		By the definition of $\varepsilon$, and since all elements $I$ are ordered before $J$, $\varepsilon(\sigma \cup \tau, I\cup J)=\varepsilon(\sigma, I)\varepsilon(\tau, I\cup J) $. Furthermore, for each $q$-simplex $\tau = \{i_1, \ldots, i_{q+1}\} \subset J$,
		\begin{align*}
			\varepsilon(\tau, I \cup J) &= \prod_{j\in \{1, \ldots, q+1\}} \varepsilon(i_j, I \cup J)= \prod_{j\in \{1, \ldots, q+1\}} (-1)^{|I|} \varepsilon(i_j, J) \\
			&= (-1)^{|I|(q+1)}  \varepsilon(\tau, J).
		\end{align*}
		Therefore, since $\varepsilon(I,J)^2=1$ for any sets $I, J$, the statement follows.
	\end{proof}

	\subsection{Massey products}\label{sect: combinatorial Massey products}
	Massey products are higher cohomology operations that were introduced in a short note by Massey \cite{Massey} and were thereafter first used by Massey and Uehara in \cite{MasseyJacobi} to prove that Whitehead products satisfy the Jacobi identity. 
	They have many applications for example as topological invariants, obstructions to formality and for calculating differentials in spectral sequences.
	
	
	\begin{defn}\label{def: defining system}
		Let $(A,d)$ be a differential graded algebra with classes $\alpha_i$ in $H^{p_i}(A,d)$ for $1\leqslant i \leqslant n$. Let $a_{i,i}\in A^{p_i}$ be a representative for $\alpha_i$. A \textit{defining system} associated to $\langle \alpha_1, \ldots, \alpha_n \rangle$ is a set of elements $(a_{i,k})$ for $1\leqslant i \leqslant k \leqslant n$ and $(i,k) \neq(1,n)$  such that $a_{i,k}\in A^{p_i + \cdots + p_k-k+i}$ and
		\[
		d(a_{i,k})=\sum_{r=i}^{k-1} \overline{a_{i,r}}a_{r+1,k}
		\]
		where $\overline{a_{i,r}}=(-1)^{1+\deg a_{i,r}}a_{i,r}$.
		To each defining system of $\langle \alpha_1, \ldots, \alpha_n \rangle$, the \textit{associated cocycle} is defined as
		\[
		\sum_{r=1}^{n-1} \overline{a_{1,r}}a_{r+1,n} \in A^{p_1 + \cdots + p_n-n+2}.
		\]
		The \textit{$n$-Massey product }$\langle \alpha_1, \ldots, \alpha_n \rangle$ is the set of cohomology classes of associated cocycles for all possible defining systems. 
		The \textit{indeterminacy} of a Massey product is the set of differences between elements in $\langle \alpha_1, \ldots, \alpha_n \rangle$.
		The Massey product is called \textit{trivial} if $0\in \langle \alpha_1, \ldots, \alpha_n \rangle$.
	\end{defn}

	We use Theorem~\ref{thm: full Hochster's} to give a correspondence between defining systems in $C^*(\mathcal{Z}_\mathcal{K})$ and in $\bigoplus \limits_{J \subset [m]} C^*(\mathcal{K}_J)$.
	For any $a\in C^{p+|J|+1}(\mathcal{Z}_\mathcal{K})$ with $p \geqslant 0$ and $J\subset[m]$, there is a corresponding $a\in C^p(\mathcal{K}_J)$.
	
	\begin{defn} \label{def: degree of elements in C(KJ)}
		For $a\in C^p(\mathcal{K}_J)$, let $\mydeg(a)=p+|J|+1$ and let $\doubleoverline{a}=(-1)^{1+\mydeg a}a=(-1)^{p+|J|}a$. 
	\end{defn}
	
	Let $\langle \alpha_1, \ldots, \alpha_n \rangle \subset H^*(\mathcal{Z}_\mathcal{K})$, where each class $\alpha_i\in H^{p_i+|J_i|+1}(\mathcal{Z}_\mathcal{K})$ corresponds to $\alpha_i\in H^{p_i}(\mathcal{K}_{J_i})$. 
	Let $(a_{i,k})\subset C^{*}(\mathcal{Z}_\mathcal{K})$ be a defining system for $\langle \alpha_1, \ldots, \alpha_n \rangle$, where $a_{i,i}=a_i$ is a cocycle representative for $\alpha_i$.
	Then $a_{i,k}\in C^{p_i+\cdots + p_k + |J_i \cup \cdots \cup J_k|+1}(\mathcal{Z}_\mathcal{K})$ and $d(a_{i,k})=\sum_{r=i}^{k-1} \overline{a_{i,r}}a_{r+1,k}$.
	By Theorem~\ref{thm: full Hochster's}, there are corresponding cochains $a_{i,k}\in C^{p_i+\cdots+p_k}(\mathcal{K}_{J_i\cup \cdots \cup J_k})$ and 
	\begin{align*}
		\mydeg(a_{i,k})&=p_i+\cdots + p_k + |J_i \cup \cdots \cup J_k|+1 \\
		&= (p_i+|J_i|+1)+\cdots + (p_k+|J_k|+1) -k+i \\
		&= \mydeg(a_{i})+\cdots + \mydeg(a_{k})-k+i.
	\end{align*}
	By the product in \eqref{eq: cochain product}, $d(a_{i,k})=\sum_{r=i}^{k-1} \doubleoverline{a_{i,r}}a_{r+1,k}$.
	Hence $(a_{i,k})\subset \bigoplus \limits_{J \subset [m]} C^*(\mathcal{K}_J)$ is a defining system that corresponds to the defining system $(a_{i,k})\subset C^*(\mathcal{Z}_\mathcal{K})$ and the associated cocycle $\omega\in C^{p_1+\cdots+p_n+|J_1\cup \cdots \cup J_n|+2}(\mathcal{Z}_\mathcal{K})$ corresponds to the associated cocycle $\omega \in C^{p_1+\cdots +p_n+1}(\mathcal{K}_{J_1\cup \cdots \cup J_n})$.
	
	Let $\langle \alpha_1, \alpha_2, \alpha_3 \rangle$  be a triple Massey product on $\alpha_i\in H^{p_i+|J_i|+1}(\mathcal{Z}_\mathcal{K})$ for $i=1,2,3$. 
	The indeterminacy of a triple Massey product is \[\alpha_1 \cdot H^{p_2+p_3+|J_2\cup J_3|+1}(\mathcal{Z}_\mathcal{K})+\alpha_3 \cdot H^{p_1+p_2+|J_1\cup J_2|+1}(\mathcal{Z}_\mathcal{K}).\] 
	By Theorem~\ref{thm: full Hochster's},  $\alpha_i$ corresponds to $\alpha_i \in \widetilde{H}^{p_i}(\mathcal{K}_{J_i})$ and the indeterminacy of $\langle \alpha_1, \alpha_2, \alpha_3 \rangle$ is
	\begin{equation} \label{eq: indeterminacy}
		\alpha_1 \cdot \widetilde{H}^{p_2+p_3}(\mathcal{K}_{J_2\cup J_3})+
		\alpha_3 \cdot \widetilde{H}^{p_1+p_2}(\mathcal{K}_{J_1\cup J_2}).
	\end{equation}
	In general, the indeterminacy of an $n$-Massey product can be expressed in terms of matric Massey products \cite[Proposition 2.3]{MatricMasseyProducts}, but this is not a helpful expression for calculations.

	\begin{exmp}
		\label{ex: Non-trivial indeterminacy triple Massey example} 
		Let $\mathcal{K}$ be the simplicial complex in Figure~\ref{fig:Massey example}.
		Let $\alpha_1, \alpha_2, \alpha_3 \in H^3(\mathcal{Z}_\mathcal{K})$ correspond to $\alpha_1=[\chi_1]\in \widetilde{H}^0(\mathcal{K}_{12})$, $\alpha_2=[\chi_3]\in \widetilde{H}^0(\mathcal{K}_{34})$, $\alpha_3=[\chi_5] \in \widetilde{H}^0(\mathcal{K}_{56})$. 
		Since $\widetilde{H}^1(\mathcal{K}_{1234})$ and $\widetilde{H}^1(\mathcal{K}_{3456})=0$, the products $\alpha_1\alpha_2\in\widetilde{H}^1(\mathcal{K}_{1234})$ and $\alpha_2\alpha_3\in$ $\widetilde{H}^1(\mathcal{K}_{3456})$ are zero. 
		
		A cochain $a_{12}\in C^0(\mathcal{K}_{1234})$ such that $d(a_{12})=\doubleoverline{\chi_1} \chi_3=0$ is of the form $a_{12}=c_1 \chi_3+c_2(\chi_1+\chi_4+\chi_2)$, for any $c_1, c_2\in \mathbf{k}$. 
		A cochain $a_{23}\in C^0(\mathcal{K}_{3456})$ such that $d(a_{23})=\doubleoverline{\chi_3}\cdot \chi_5=\chi_{35}$ is of the form $a_{23}=c_3(\chi_4+\chi_6+\chi_3+\chi_5) +\chi_5$ for any $c_3\in \mathbf{k}$.
		Then the associated cocycle $\omega\in C^1(\mathcal{K})$ is 
		$\doubleoverline{a}_1a_{23}+\doubleoverline{a}_{12}a_3= 
		c_3(\chi_{14} +\chi_{16}+\chi_{15})+\chi_{15}+c_1 \chi_{35} + c_2(\chi_{15}+\chi_{25}) 
		$. 
		For $\chi_1,\chi_5\in C^0(\mathcal{K})$, $\omega=c_3d(\chi_1)+\chi_{15}+(c_1-c_2)\chi_{35}+c_2d(\chi_5)$. 
		Also, $[\omega]=[\chi_{15}+(c_1-c_2)\chi_{35}]\neq 0$ for any $c_1, c_2, c_3 \in \mathbf{k}$. 
		Therefore $\langle \alpha_1, \alpha_2, \alpha_3 \rangle\subset H^8(\mathcal{Z}_\mathcal{K})$ is non-trivial with non-trivial indeterminacy,
		$ 
		\alpha_1 \cdot \widetilde{H}^0(\mathcal{K}_{3456}) + \alpha_3 \cdot \widetilde{H}^0(\mathcal{K}_{1234})=\alpha_3 \cdot \widetilde{H}^0(\mathcal{K}_{1234}). 
		$

		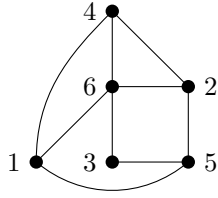
\begin{figure}[htbp]
			\centering
			\begin{tikzpicture}[inner sep=2mm]
			\coordinate (a) at (0,0);
			\coordinate (b) at (1,0);
			\coordinate (c) at (0,1);
			\coordinate (d) at (1,1);
			\coordinate (e) at (0,2);
			\coordinate (f) at (-1,0);
			
			\draw (c)--(a) node[left]{3}-- (b) node[right]{5} -- (d) node[right]{2} --(e) node[left]{4};
			\draw (d) -- (c) node[left]{6} -- (f) node[left] {1};
			\draw (f)  to [out=90,in=-135] (e);
			\draw (f) to [out=-40,in=-140] (b);
			\draw (c)--(e);
			
			\fill (a) circle (2.5pt);\fill (b) circle (2.5pt);\fill (c) circle (2.5pt);
			\fill (d) circle (2.5pt);\fill (e) circle (2.5pt);\fill (f) circle (2.5pt);
			\end{tikzpicture} 
			\caption{A simplicial complex $\mathcal{K}$ for which $\mathcal{Z}_\mathcal{K}$ has a non-trivial $3$-Massey product with non-trivial indeterminacy.}
			\label{fig:Massey example}
		\end{figure}
	\end{exmp}

	\section{Massey products via join and star deletion} \label{section:combinatorial operations}
	The categorical product of simplicial complexes $\mathcal{K}_1$ and $\mathcal{K}_2$ is the join $\mathcal{K}_1 *\mathcal{K}_2$.
	This induces a product in moment-angle complexes, $\mathcal{Z}_{\mathcal{K}_1 * \mathcal{K}_2}=\mathcal{Z}_{\mathcal{K}_1}\times \mathcal{Z}_{\mathcal{K}_2}$.
	In this way cup products in $H^*(\mathcal{Z}_{\mathcal{K}_1 * \mathcal{K}_2})$ can be seen combinatorially. 
	Since Massey products are higher operations, we require lower Massey products to be trivial.
	The idea is to start with the join of simplicial complexes and remove certain simplices in order to trivialise lower Massey products.	
	To remove simplices, we use \textit{star deletion}.
	
	For a simplicial complex $\mathcal{K}$, the \textit{star} and \textit{link} of a simplex $I\in \mathcal{K}$ are \[\st_{\mathcal{K}}I=\{J\in \mathcal{K} \suchthat I \cup J \in \mathcal{K} \} \text{  and  } \link_{\mathcal{K}}I=\{J\in \mathcal{K} \suchthat I \cup J \in \mathcal{K}, I\cap J=\myempty \}.\] 
	The \textit{boundary of the star} of $I\in \mathcal{K}$ is
	$
	\partial \st_{\mathcal{K}}I=\{J\in \mathcal{K} \suchthat I \cup J \in \mathcal{K}, I\not \subset J \}.
	$
	Let $\mathring{\st}_{\mathcal{K}}I=\st_{\mathcal{K}}I \setminus \partial \st_{\mathcal{K}}I$.
	
	\begin{defn}
		The \textit{star deletion} $\sd_I\mathcal{K}$ of $\mathcal{K}$ at $I$ is $\sd_I\mathcal{K}=\mathcal{K} \setminus \mathring{\st}_{\mathcal{K}}I$.
	\end{defn}
	
	Alternatively, $\sd_I\mathcal{K}=\{ J\in \mathcal{K} \suchthat I\not\subset J  \}$.
	Star deletions $\sd_{I_1}$ and $\sd_{I_2}$ can be applied iteratively providing that $I_1\not\subset I_2$ and $I_2\not\subset I_1$.
	We show that the order of star deletions on a simplicial complex does not affect the result.
	
	\begin{lemma} \label{lem: order of deletes doesn't matter}
		Let $\mathcal{K}$ be a simplicial complex. Let $I_1, I_2\in \mathcal{K}$ be simplices such that $I_1 \cap I_2 \neq I_1, I_2$. Then $\sd_{I_2}\sd_{I_1} \mathcal{K}=\sd_{I_I}\sd_{I_2} \mathcal{K}$. 
	\end{lemma}
	\begin{proof}
		Since $I_1 \cap I_2 \neq I_1, I_2$, neither $I_1\subset I_2$ nor $I_2\subset I_1$. Thus $I_1\in \sd_{I_2} \mathcal{K}$ and $I_2\in \sd_{I_1} \mathcal{K}$.
		So $\sd_{I_2}\sd_{I_1} \mathcal{K}=\mathcal{K} \setminus (\mathring{\st}_{\mathcal{K}}I_1 \cup \mathring{\st}_{\mathcal{K}}I_2)=\sd_{I_1}\sd_{I_2} \mathcal{K}$.
	\end{proof}
	
	\begin{figure}[ht]
		\centering
		
		\begin{minipage}{0.3\textwidth}
			\centering
			\begin{tikzpicture} [scale=1.5, inner sep=2mm]
			\coordinate (4) at (0,0);
			\coordinate (1) at (1,0);
			\coordinate (2) at (0.6,0.5);
			\coordinate (3) at (1.55,0.5);
			\coordinate (6) at (0.8,1.5);
			\coordinate (5) at (0.8,-1.2);
			
			\fill[lightgray, fill opacity=0.4] (6)--(4)--(5)--(3)--(6);
			\fill[lightgray, fill opacity=0.6] (4)--(1)--(3)--(5);
			
			\draw (3) node[right] {3} -- (1) node[below right] {1} --  (4)--(6) node[above] {6}; 
			\draw (4)  node[left] {4} -- (5) node[below] {5} -- (3); \draw (1) -- (5);
			\draw[] (6) -- (2) node[above right] {2};
			\draw[dashed] (2) -- (5); \draw[] (4) -- (2); 
			\draw[] (2) -- (3);
			\draw[] (3)--(6);
			\draw (2)--(0.65,0);

			\foreach \i in {1,...,6} {\fill (\i) circle (1.8pt);}; 
			\end{tikzpicture} 
			\subcaption{$\sd_{\{ 1,6\}} \mathcal{K}$} 
			\label{fig: star deletion 16}
		\end{minipage}
		\begin{minipage}{0.3\textwidth}
			\centering
			\begin{tikzpicture} [scale=1.5, inner sep=2mm]
			\coordinate (4) at (0,0);
			\coordinate (1) at (1,0);
			\coordinate (2) at (0.6,0.5);
			\coordinate (3) at (1.55,0.5);
			\coordinate (6) at (0.8,1.5);
			\coordinate (5) at (0.8,-1.2);
			
			\fill[lightgray, fill opacity=0.4] (6)--(4)--(5)--(3)--(2)--(6);
			\fill[lightgray, fill opacity=0.6] (4)--(6)--(1)--(3)--(5)--(4);
			
			\draw (3) node[right] {3} -- (1) node[below right] {1} --  (4)--(6) node[above] {6}; 
			\draw (4)  node[left] {4} -- (5) node[below] {5} -- (3); \draw (1) -- (5);
			\draw[dashed] (6) -- (2) node[above right] {2};
			\draw[dashed] (2) -- (5); \draw[dashed] (4) -- (2); 
			\draw[name path=2-3, dashed] (2) -- (3);
			\draw[name path=1-5] (1)--(6);
			\path [name intersections={of=1-5 and 2-3,by=int}];
			\draw (3)--(int);
			
			\foreach \i in {1,...,6} {\fill (\i) circle (1.8pt);}; 
			\end{tikzpicture} 
			\subcaption{$\sd_{\{3,6\}} \mathcal{K}$} \label{fig: octahedron star delete 36}
		\end{minipage}
		\begin{minipage}{0.3\textwidth}
			\centering
			\begin{tikzpicture} [scale=1.5, inner sep=2mm]
			\coordinate (4) at (0,0);
			\coordinate (1) at (1,0);
			\coordinate (2) at (0.6,0.5);
			\coordinate (3) at (1.55,0.5);
			\coordinate (6) at (0.8,1.5);
			\coordinate (5) at (0.8,-1.2);
			
			\fill[lightgray, fill opacity=0.4] (6)--(4)--(5)--(3)--(2)--(6);
			\draw (2)--(0.65,0);
			\fill[lightgray, fill opacity=0.6] (4)--(1)--(3)--(5);
			
			\draw (3) node[right] {3} -- (1) node[below right] {1} --  (4)--(6) node[above] {6}; 
			\draw (4)  node[left] {4} -- (5) node[below] {5} -- (3); \draw (1) -- (5);
			\draw[] (6) -- (2) node[above right] {2};
			\draw[dashed] (2) -- (5); \draw[] (4) -- (2) -- (3);
			
			\foreach \i in {1,...,6} {\fill (\i) circle (1.8pt);}; 
			\end{tikzpicture} 
			\subcaption{$\sd_{\{3,6\}}\sd_{\{ 1,6\}} \mathcal{K}\\=\sd_{\{ 1,6\}}\sd_{\{3,6\}} \mathcal{K}$} \label{fig: star deletion example octahedron}
		\end{minipage}
		\caption{The star deleted complex is not affected by the order of star deletions.}
		\label{fig:deletion order}
	\end{figure}
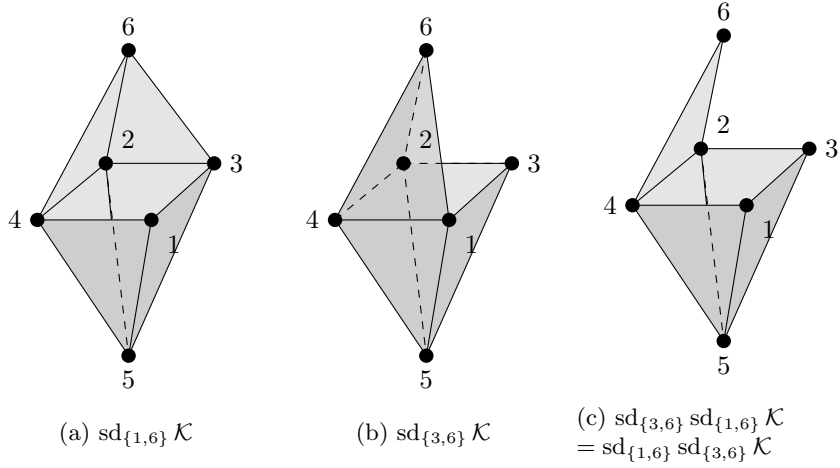
	
	\begin{exmp}
		Let $\mathcal{K}$ be the boundary of an octahedron with opposing vertices labelled $i, i+1$ for $i=1, 3, 5$. Let $I_1=\{1,6\}$ and $I_2=\{3,6\}$. 
		The star $\st_{\mathcal{K}}I_1$ contains maximal simplices $\{1,4,6\}$ and $\{1,3,6\}$, and $\st_{\mathcal{K}}I_2$ contains $\{1,3,6\}$ and $\{2,3,6\}$. 
		If the star of $I_1$ is deleted from $\mathcal{K}$ first, then $\st_{\sd_{I_1}\mathcal{K}}I_2$ contains the maximal simplex $\{2,3,6\}$.
		Hence $\sd_{I_2}\sd_{I_1} \mathcal{K}$ removes the simplices $\{1,4,6\}$, $\{1,3,6\}$ and $\{2,3,6\}$ from $\mathcal{K}$.
		The same simplices are removed from $\mathcal{K}$ in  $\sd_{I_1}\sd_{I_2} \mathcal{K}$, as shown in Figure~\ref{fig:deletion order}.	
	\end{exmp}
	
	\begin{remark}
		Star deletion is equivalent to doing a stellar subdivision $\stellar_I \mathcal{K} = (\mathcal{K} \setminus \mathring{\st}_{\mathcal{K}}I)\cup_{\partial \st_{\mathcal{K}}I} \text{cone}(\partial \st_{\mathcal{K}}I)$ then restricting to the original vertices $V(\mathcal{K})$.
		For example, see Figure~\ref{fig: stellar subdivision example} compared to Figure~\ref{fig: star deletion 16}.
		If $\mathcal{K}$ is a triangulation of an $n$-sphere on $m$ vertices, then $\mathcal{Z}_\mathcal{K}$ is an $(m+n+1)$-dimensional manifold. 
		As $\stellar_I \mathcal{K} \simeq \mathcal{K}$, $\stellar_I \mathcal{K}$ is a triangulation of an $n$-sphere on $m+1$ vertices.
		Hence $\mathcal{Z}_{\stellar_I\mathcal{K}}$ is an $(m+n+2)$-dimensional manifold.
		Since Massey products are obstructions to formality, 
		a non-trivial Massey product in $H^*(\mathcal{Z}_{\stellar_I\mathcal{K}})$ implies that  $\mathcal{Z}_{\stellar_I\mathcal{K}}$ is a non-formal. 
	\end{remark}
	
	\begin{figure}[ht]
		\centering
		\begin{minipage}{0.3\textwidth}
			\centering
			\begin{tikzpicture} [scale=1.5, inner sep=2mm]
			\coordinate (4) at (0,0);
			\coordinate (1) at (1,0);
			\coordinate (2) at (0.6,0.5);
			\coordinate (3) at (1.55,0.5);
			\coordinate (6) at (0.8,1.5);
			\coordinate (5) at (0.8,-1.2);
			\coordinate (7) at (1.1, 0.7);
			
			\fill[lightgray, fill opacity=0.8] (6)--(4)--(5)--(3)--(6);
			
			\draw (3) node[right] {3} -- (1) node[below right] {1} --  (4)--(6) node[above] {6}; 
			\draw (4)  node[left] {4} -- (5) node[below] {5} -- (3); \draw (1) -- (5);
			\draw[dashed] (6) -- (2) node[above left, inner sep=1mm] {2};
			\draw[dashed] (2) -- (5); \draw[dashed] (4) -- (2); 
			\draw[dashed] (2) -- (3);
			\draw[] (3)--(6);
			\draw (1)--(7)--(6);
			\draw (4)--(7)--(3);

			\foreach \i in {1,...,6,7} {\fill (\i) circle (1.8pt);}; 
			\end{tikzpicture} 
		\end{minipage}
		\caption{A stellar subdivision at $\{1,6\}$ in the octahedron}
		\label{fig: stellar subdivision example}
	\end{figure}
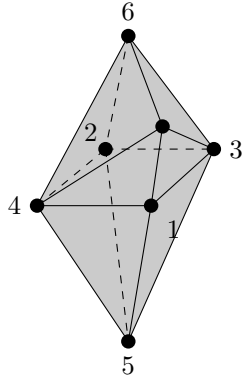
	
	\subsection{A construction of non-trivial Massey products}
	
	We aim to construct a simplicial complex $\mathcal{K}$ such that there is a non-trivial $n$-Massey product in $H^*(\mathcal{Z}_\mathcal{K})$.
	We start with the join of $n$-simplicial complexes $\mathcal{K}^1 * \cdots * \mathcal{K}^n$ and classes $\alpha_i \in \widetilde{H}^{p_i}(\mathcal{K}^i_{J_i})$ for each $i\in \{1, \ldots, n\}$.
	In $\mathcal{K}^1 * \cdots * \mathcal{K}^n$, all cup products between $\alpha_i$s are non-trivial, so in order to define a higher Massey product we first remove simplices to make certain cup products trivial. 
	To define which simplices to remove, 
	we define two sets of simplices, $S_{a_i}\subset \mathcal{K}^i$ and $P_{a_i}\subset \mathcal{K}^i$ for each $\mathcal{K}^i$.
	In order to create $\mathcal{K}$, we star delete $\mathcal{K}^1 * \cdots * \mathcal{K}^n$ at every simplex $\sigma_i\cup \sigma_k$ for $\sigma_i \in S_{a_i}$ and $\sigma_k \in P_{a_k}$, $1\leq i<k\leq n$, $(i,k)\neq (1,n)$.
	The star deletions at $\sigma_1\cup \sigma_2$ and $\sigma_2\cup \sigma_3$ trivialise the cup products $\alpha_1 \alpha_2$ and $\alpha_2\alpha_3$ respectively, which is required to define a triple Massey product $\langle \alpha_1, \alpha_2, \alpha_3 \rangle$. 
	By star deleting at $\sigma_1\cup \sigma_3$, we trivialise $\langle \alpha_1, \alpha_2, \alpha_3 \rangle$.
	If we also star delete at simplices $\sigma_3\cup \sigma_4$ and $\sigma_2\cup \sigma_4$, then $\langle \alpha_2, \alpha_3, \alpha_4 \rangle$ is defined and trivial, so the $4$-Massey product $\langle \alpha_1, \alpha_2, \alpha_3, \alpha_4 \rangle$ is defined.
	We define the Massey product $\langle \alpha_1, \ldots, \alpha_n \rangle\subset H^*(\mathcal{Z}_\mathcal{K})$ by iterating this process.

	\begin{cons}\label{cons: simplicial complex for n-Massey}
		For $i\in \{1, \ldots, n\}$, let $\mathcal{K}^i$ be a simplicial complex on $[m_i]$ vertices that is not an $(m_i-1)$-simplex. 
		Since $\mathcal{K}^i$ is not a simplex, there is a non-zero cohomology class $\alpha_i \in \widetilde{H}^{p_i}(\mathcal{K}^i_{J_i})$ for $p_i \in \mathbb{N}$, $J_i\subseteq [m_i]$.
		Let $a_i\in C^{p_i}(\mathcal{K}^i_{J_i})$ be a cocycle representative for $\alpha_i$ that is supported on $p_i$-simplices $S_{a_i}\subset \mathcal{K}$ so that 
		$
		a_i=\sum_{\sigma_i \in S_{a_i}} c_{\sigma_i} \chi_{\sigma_i} \in C^{p_i}(\mathcal{K}^i_{J_i})
		$
		for a non-zero coefficient $c_{\sigma_i} \in \mathbf{k}$.
		For every simplex $\sigma_i \in S_{a_i}$, let $v_{\sigma_i}$ denote one vertex in $\sigma_i$.
		Let $P_{\sigma_i}$ be the set  
		\begin{equation*}
			P_{\sigma_i}=\{\sigma_i' \in \mathcal{K}^i \suchthat \sigma_i' \text{ is a } p_i \text{-simplex}, \sigma_i\cap \sigma_i'=\sigma_i\setminus v_{\sigma_i}\}.
		\end{equation*}
		An example is shown in Figure~\ref{fig: P_sigma example}.
		
		\begin{figure}[ht]
			\centering
			
			\begin{tikzpicture} [scale=1.6, inner sep=2mm]
			\coordinate (1) at (0,0);
			\coordinate (2) at (1,0);
			\coordinate (3) at (-0.5, 1);
			\coordinate (4) at (0.5, 1);
			\coordinate (5) at (1.5, 1);
			\coordinate (l1) at (0.5,0.5);
			\coordinate (l2) at (1, 0.5);
			\coordinate (l3) at (0, 0.5);
			
			\fill[lightgray, fill opacity=0.4] (1)--(3)--(4)--(5)--(2)--(1);
			\draw (1)node[left] {$v_\sigma$} -- (3) -- (4) -- (1) --(2)--(4)--(5)--(2);
			
			\draw (l1) node[below, inner sep=0mm] {$\sigma$};
			\draw (l2) node[above, inner sep=0mm] {$\sigma'$};
			\draw (l3) node[above, inner sep=0mm] {$\sigma''$};
			\foreach \i in {1,...,5} {\fill (\i) circle (2pt);}; 
			\end{tikzpicture} 
			
			\caption{For this choice of vertex $v_{\sigma}\in \sigma$, $\sigma'\in P_{\sigma}$ but $\sigma''\notin P_{\sigma}$}
			\label{fig: P_sigma example}
		\end{figure}
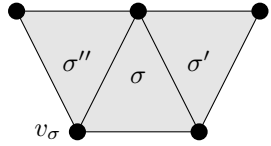
		
		We start by constructing the set $P_{a_i}$ for each $i$, in order to define star deletions of $\mathcal{K}^1 * \cdots * \mathcal{K}^n$. 
		We fix an order on the simplices in $S_{a_i}$ and define an ordered subsequence of simplices $\sigma_i^{(1)}, \ldots, \sigma_i^{(l)}\subset S_{a_i}$.
		Let $\sigma_i^{(1)}$ be the first element of $S_{a_i}$. 
		Then let $S_{a_i}^{(1)}=S_{a_i}\setminus P_{\sigma_i^{(1)}}$.
		Let $\sigma_i^{(2)}$ be the next element after $\sigma_i^{(1)}$ in $S_{a_i}^{(1)}$. 
		Then let $S_{a_i}^{(2)}=S_{a_i}^{(1)}\setminus P_{\sigma_i^{(2)}}$.
		We continue repeatedly until $\sigma_i^{(l)}$ is the last element of $S_{a_i}^{(l-1)}$, and let 
		\begin{equation}\label{eq: where to delete}
			P_{a_i}=P_{\sigma_i^{(1)}} \cup \cdots \cup P_{\sigma_i^{(l)}}.
		\end{equation}
		Since $\sigma_i^{(l)}\notin P_{\sigma_i^{(l)}}$,  the set $S_{a_i}\setminus P_{a_i}=S_{a_i}^{(l-1)}\setminus P_{\sigma_i^{(l)}}$ contains at least the last element $\sigma_i^{(l)}$. So $P_{a_i}\neq S_{a_i}$.
		
		Let $\bar{\mathcal{K}}=\mathcal{K}^1 * \cdots * \mathcal{K}^n$, so $\bar{\mathcal{K}}_{[m_i]}=\mathcal{K}^i$ for every $i\in \{1, \ldots, n\}$. 
		The vertices in each vertex set $V(\mathcal{K}^i)=[m_i]$ have an order. Suppose that the vertex set $V(\bar{\mathcal{K}})=\bigsqcup_{i\in \{1,\ldots, n\}} V(\mathcal{K}^i)$ is ordered so that $u < v$ for all $u\in V(\mathcal{K}^i)$ and $v\in V(\mathcal{K}^{j})$ for $i<j$. 
		We construct a simplicial complex $\mathcal{K}$ by star deleting $\bar{\mathcal{K}}$ at each simplex $\sigma_i\cup \sigma_k$ one at a time, 
		where $\sigma_i\in S_{a_i}$ and $\sigma_k \in P_{a_k}$, $1\leq i<k\leq n$, $(i,k)\neq (1,n)$.
		Let $\mathcal{K}$ denote the resulting simplicial complex.
	\end{cons}
	
	\begin{lemma}
		For any $i\in \{1, \ldots, n\}$, the set $P_{a_i}$ is non-empty.
	\end{lemma}
	\begin{proof}
		If $p_i=0$ and $\widetilde{H}^{0}(\mathcal{K})\neq 0$, then $\mathcal{K}$ is a disjoint union of at least two vertices. For any $v, w \in \mathcal{K}$, $v\cap w=\myempty=v\setminus v$. Hence $w\in P_v$.
		Alternatively let $p_i>0$.
		Since $\alpha_i\in \widetilde{H}^{p_i}(\mathcal{K}_{J_i})$ is non-zero, there is a non-zero cycle $x\in C_{p_i}(\mathcal{K}_{J_i})$ such that $a_i(x)\neq 0$. 
		Let $x=\sum_{\tau\in S_x} c_{\tau} \Delta_{\tau}$ 	with non-zero coefficients $c_{\tau}$ and $p_i$-simplices $\tau$.
		Let $\sigma\in S_{a_i}\cap S_x$.
		Let $\partial\colon C_{p_i}(\mathcal{K}_{J_i})\to C_{p_i-1}(\mathcal{K}_{J_i})$ be the boundary homomorphism.
		Since $x$ is a cycle and $\partial \Delta_\sigma\neq 0$, for any vertex $v\in \sigma$ there exists a different simplex $\tau\in S_x$ with $\sigma\setminus v= \sigma\cap \tau=\tau\setminus u$ for some vertex $u\in \tau$.
		Hence for any $\sigma\in S_{a_i}\cap S_x$, $\tau\in P_{\sigma}$ and so $P_{a_i}$ is non-empty.
	\end{proof}

	\begin{exmp}[a]\label{ex: joins construction example}
		Let $\mathcal{K}^1$ be the disjoint union of two vertices $\{1\}, \{2\}$ and $\mathcal{K}^2$ the simplicial complex in Figure~\ref{fig: star delete K}. 
		The join $\mathcal{K}^1*\mathcal{K}^2$ is homotopy equivalent to $S^2 \vee S^1$. Let $\alpha_1\in \widetilde{H}^{0}(\mathcal{K}^1)$, $\alpha_2 \in \widetilde{H}^{0}(\mathcal{K}^2)$ be represented by the cochains $a_1=\chi_1 $ and $a_2=\chi_3+\chi_4+\chi_5$, respectively. 
		Then $S_{a_1}=\{\{1\}\}$, and $S_{a_2}=\{\{3\}, \{4\}, \{5\}\}$.
		Following the construction above, for $\sigma_2=\{3\}$ there is only one choice of a vertex $v=3$. 
		Then
		$
		P_{\{3\}}=\{\{4\}, \{5\}, \{6\} \}
		$
		so $S_{a_2}^{(1)}=\{\{3\}\}$ and 
		$P_{a_2}=P_{\{3\}}$.
		The simplicial complex 
		\[
		\mathcal{K}=\sd_{\{1,6\}}\sd_{\{1,5\}}\sd_{\{1,4\}} \mathcal{K}^1*\mathcal{K}^2
		\]
		is shown in Figure~\ref{fig: star delete example 1}. 
		Since $\mathcal{K}$ is contractible, the cup product $\alpha_1 \alpha_2$ is trivial.
		
		(b). 
		In addition to $\mathcal{K}^1$ and $\mathcal{K}^2$ in Part (a), 
		let $\mathcal{K}^3$ be the disjoint union of two vertices $\{7\}, \{8\}$. Let $\alpha_3\in \widetilde{H}^{0}(\mathcal{K}^3)$ be represented by $a_3=\chi_7$. 
		Then $S_{a_3}=\{\{7\}\}$ and
		$P_{a_3}=P_{\{7\}}=\{\{8\}\}$. 
		By Construction~\ref{cons: simplicial complex for n-Massey}, we star delete $\mathcal{K}^1*\mathcal{K}^2*\mathcal{K}^3$
		at $\sigma_i\cup \sigma_k$ for every $\sigma_i\in S_{a_i}$ and $\sigma_k\in P_{a_k}$ for $i=1,2$ and $k=i+1$. 
		Since $S_{a_2}=\{\{3\}, \{4\}, \{5\}\}$, we obtain the simplicial complex 
		\[
		\mathcal{K}'=\sd_{\{5,8\}}\sd_{\{4,8\}}\sd_{\{3,8\}}\sd_{\{1,6\}}\sd_{\{1,5\}}\sd_{\{1,4\}} \mathcal{K}^1*\mathcal{K}^2*\mathcal{K}^3.
		\]
		The full subcomplex $\mathcal{K}'_{3,4,5,6,7,8}$ is shown in Figure~\ref{fig: star delete example 2}.
		Theorem~\ref{thm: joins} will show that there is a non-trivial triple Massey product in $H^*(\mathcal{Z}_{\mathcal{K}'})$. 
	\end{exmp}

	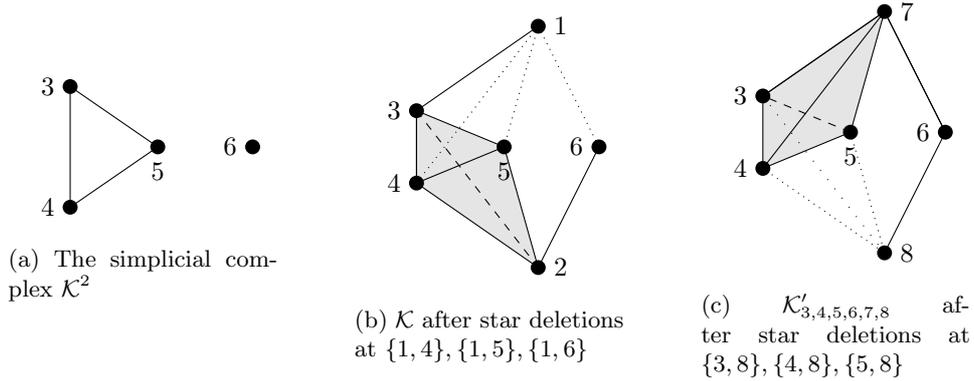
\begin{figure}[ht]
		\centering
		\begin{minipage}{0.28\textwidth}
			\centering
			\begin{tikzpicture} [scale=0.8, inner sep=2mm]
			\coordinate (3) at (0, 1);
			\coordinate (4) at (0, -1);
			\coordinate (5) at (1.44, 0);
			\coordinate (6) at (3, 0) ;
			
			\draw (3) node[left] {3} -- (4) node[left] {4}-- (5) node[below] {5} -- (3);
			\draw (6) node[left] {6};
			\foreach \i in {3,...,6} {\fill (\i) circle (3.5pt);}; 
			\end{tikzpicture} 
			\subcaption{The simplicial complex $\mathcal{K}^2$} \label{fig: star delete K}
		\end{minipage} \hfill
		\begin{minipage}{0.28\textwidth}
			\centering
			\begin{tikzpicture} [scale=0.8, inner sep=2mm]
			\coordinate (3) at (0, 0.6);
			\coordinate (4) at (0, -0.6);
			\coordinate (5) at (1.44, 0);
			\coordinate (6) at (3, 0) ;
			\coordinate (1) at (2, 2);
			\coordinate (2) at (2, -2);
			
			\fill[lightgray, fill opacity=0.4] (3)--(4)--(2)--(5);
			
			\draw (5)--(3) node[left] {3} -- (4) node[left] {4}-- (5) node[below] {5}--(2);
			\draw (6) node[left] {6} -- (2) node[right] {2}-- (4) --(3) --(1);
			\draw[dashed] (3)--(2);
			\draw[dotted] (4)--(1)--(5); \draw[dotted] (1)node[right] {1}--(6);
			\foreach \i in {1,...,6} {\fill (\i) circle (3.5pt);}; 
			\end{tikzpicture} 
			\subcaption{$\mathcal{K}$ after star deletions at $\{1,4\}, \{1,5\}, \{1,6\}$} \label{fig: star delete example 1}
		\end{minipage} \hfill
		\begin{minipage}{0.28\textwidth}
			\centering
			\begin{tikzpicture} [scale=0.8, inner sep=2mm]
			\coordinate (3) at (0, 0.6);
			\coordinate (4) at (0, -0.6);
			\coordinate (5) at (1.44, 0);
			\coordinate (6) at (3, 0) ;
			\coordinate (1) at (2, 2);
			\coordinate (2) at (2, -2);
			
			\fill[lightgray, fill opacity=0.4] (3)--(4)--(5)--(1);
			
			\draw (3) node[left] {3} -- (4) node[left] {4}-- (5) node[below] {5} -- (1) --(3);
			\draw (2) node[right] {8}; \draw (6) node[left] {6} -- (1)--(3);\draw (4)--(1) node[right] {7}--(6)--(2);
			\draw[dashed] (5)--(3);
			\draw[dotted] (5)--(2)--(4); \draw[loosely dotted] (3)--(2);
			\foreach \i in {1,...,6} {\fill (\i) circle (3.5pt);}; 
			\end{tikzpicture} 
			\subcaption{$\mathcal{K}'_{3,4,5,6,7,8}$ after star deletions at $\{3,8\}, \{4,8\}, \{5,8\}$} \label{fig: star delete example 2}
		\end{minipage} 
		\caption{Example of the star deletions in  Construction~\ref{cons: simplicial complex for n-Massey}}
		\label{fig: star delete example}
	\end{figure}

	\begin{lemma}
		The simplicial complex $\mathcal{K}$ is independent of the order of simplices in $P_{a_k}$.
	\end{lemma}
	\begin{proof}
		For any $\sigma_k, \sigma_k'\in P_{a_k}$, we have $(\sigma_i\cup \sigma_k)\cap (\sigma_i\cup \sigma_k') \neq \sigma_i\cup \sigma_k, \sigma_i\cup \sigma_k'$. So by Lemma~\ref{lem: order of deletes doesn't matter}, the order of $P_{a_k}$ does not affect $\mathcal{K}$.
	\end{proof}
	
	\begin{lemma}
		The simplicial complex $\mathcal{K}$ is independent of the order in which the pairs $\{i,k\}$, $1\leqslant i < k \leqslant n$, are chosen.
	\end{lemma}
	
	\begin{proof}
		Let $\{i_1,k_1\}$ and $\{i_2,k_2\}$ be two pairs of indices. 
		For $1\leqslant i_j < k_j \leqslant n$, $j=1,2$, let  $\sigma_{i_j} \in S_{a_{i_j}}$ and  $\sigma_{k_j}\in P_{a_{k_j}}$. 
		The intersection $(\sigma_{i_1} \cup \sigma_{k_1})\cap (\sigma_{i_2} \cup \sigma_{k_2})$ is empty since  the vertices of any $\sigma_j\in S_{a_j}$ are a subset of $J_j$ for every $j\in \{1,\ldots,n\}$ and
		$(J_{i_1} \cup J_{k_1})\cap  (J_{i_2} \cup J_{k_2})=\myempty$.
		Therefore by Lemma~\ref{lem: order of deletes doesn't matter}, we can star delete at simplices $\sigma_{i_1} \cup \sigma_{k_1}$ and simplices $\sigma_{i_2} \cup \sigma_{k_2}$ in either order.
	\end{proof}

	\begin{lemma}
		In Construction~\ref{cons: simplicial complex for n-Massey}, the simplicial complex $\mathcal{K}$ depends on the order of simplices in $S_{a_k}$.
	\end{lemma}	
	
	\begin{proof}
		Suppose that $\sigma_k\in S_{a_k}$, $\sigma_k'\in P_{\sigma_k}$ and let $\sigma_i\in S_{a_i}$ for an $i\in \{1, \ldots, k-1\}$.
		If $\sigma_k'\in S_{a_k}\cap P_{\sigma_k}$, then either $\sigma_k' > \sigma_k$ or $\sigma_k'<\sigma_k$ in the order of simplices in $S_{a_k}$. 
		In the first case, $\sigma_k'\in P_{a_k}$ so we perform a star deletion at $\sigma_i\cup \sigma_k'$.
		If there is no simplex $\sigma_k''\in S_{a_k}$ such that $\sigma_k'' > \sigma_k$ and $\sigma_k\in P_{\sigma_k''}$, then $\sigma_k\not \in P_{a_k}$. So $\sigma_i\cup \sigma_k\in \mathcal{K}$ and $\sigma_i\cup \sigma_k'\notin \mathcal{K}$. 
		In the second case, if the chosen vertex $v_{k'}\in\sigma_k'$ is such that $\sigma_k'\setminus v_{k'}=\sigma_k\setminus v_k$, then $\sigma_k\in P_{\sigma_k'}$. 
		Since $\sigma_k'<\sigma_k$, $\sigma_k\in P_{a_k}$ and therefore  $\sigma_i\cup \sigma_k\notin \mathcal{K}$. 
		Hence $\mathcal{K}$ is different in the two cases.
	\end{proof}
	
	\begin{lemma}
		The choice of vertex $v_{\sigma_k}\in \sigma_k$ affects the number of stars deletions performed in Construction~\ref{cons: simplicial complex for n-Massey}.
	\end{lemma}
	\begin{proof}
		We demonstrate this with an example. Consider the join $\mathcal{K}^1* \mathcal{K}^2 * \mathcal{K}^3$ of three simplicial complexes. 
		Suppose that $\mathcal{K}^2$ is the boundary of a tetrahedron on the vertices $1, 2, 3, 4$. 
		Also suppose that $a_2\in C^1(\mathcal{K}^2)$ is $\chi_{123}+\chi_{234}$. We fix the order on $S_{a_2}=\{ \{1,2,3\}, \{2,3,4 \} \}$.
		First let $v_{\{123\}}=3\in \{1,2,3\}$ and $v_{\{234 \} }=2\in \{2,3,4 \}$. By definition, $P_{\{123\}}=\{\sigma\in \mathcal{K}^2 \suchthat  \sigma \text{ is a } 1 \text{-simplex and } \sigma \cap \{1,2,3\}=\{1,2,3\}\setminus v_{\{123\}} \} =\{1,2,4\}$. 
		Similarly $P_{\{234\}}=\{1,3,4 \}$.
		Therefore $P_{a_2}=\{\{1,2,4\}, \{1,3,4\}\}$.
		To construct $\mathcal{K}$ from $\mathcal{K}^1* \mathcal{K}^2 * \mathcal{K}^3$, we perform $|S_{a_1}||P_{a_2}|+|S_{a_2}||P_{a_3}|=2|S_{a_1}|+|S_{a_2}||P_{a_3}|$ star deletions.
		
		Compare this to the case when $v_{\{123\}}=1$, so $P_{\{123\}}=\{2,3,4\}$.
		Since $\{1,2,3\}$ comes before $\{2,3,4 \}$ in $S_{a_2}$ and $S_{a_2}\setminus P_{\{123\}}= \{1,2,3\}$, 
		$P_{a_2}=P_{\{123\}}=\{2,3,4\}$.
		In this case, to construct $\mathcal{K}$ we perform $|S_{a_1}|+|S_{a_2}||P_{a_3}|$ star deletions.
		Since $S_{a_1}$ and $P_{a_3}$ do not depend on $v_{\{123\}}$ or $v_{\{234\}}$, this is fewer star deletions than when $v_{\{123\}}=3$.
	\end{proof}
	
	We will prove that the Massey product $\langle \alpha_1, \ldots, \alpha_n \rangle \subset H^{*}(\mathcal{Z}_\mathcal{K})$ is non-trivial in several steps, first showing that it is defined.
	
	\begin{prop} \label{prop: joins Massey defined}
		Let $\mathcal{K}$ be a simplicial complex constructed in Construction~\ref{cons: simplicial complex for n-Massey}. Then $\langle \alpha_1, \ldots, \alpha_n \rangle \subset H^{*}(\mathcal{Z}_\mathcal{K})$ is defined. 
	\end{prop}
	
	\begin{proof}
		Let $a_i=\sum_{\sigma_i \in S_{a_i}} c_{\sigma_i} \chi_{\sigma_i}$ be a representative cocycle for $\alpha_i \in \widetilde{H}^{p_i}(\mathcal{K}_{J_i})$ for each $i\in \{1, \ldots, n-1\}$. 
		We construct a defining system $(a_{i,k})$ for the Massey product $\langle \alpha_1, \ldots, \alpha_n \rangle\subset H^{p_1+\cdots+p_n+|J_1\cup \cdots \cup J_n|+2}(\mathcal{Z}_\mathcal{K}).$ 
		
		For $1\leqslant i \leqslant k \leqslant n$, $(i,k)\neq (1,n)$, let $a_{i,k}\in C^{p_i + \cdots + p_k}(\mathcal{K}_{J_i \cup \cdots \cup J_k})$ be the cochain given by
		\begin{equation} \label{eq: a_ik joins}
			a_{i,k}=  \sum_{\sigma_i \in S_{a_i}} 
			\sum_{\sigma_{i+1} \in \widetilde{S}_{a_{i+1}}} \cdots \sum_{\sigma_k \in \widetilde{S}_{a_k}}
			c_{\sigma_i}  \ldots c_{\sigma_k} \; \theta_{i,k} \;
			\chi_{\sigma_i \cup \cdots \cup \sigma_k \setminus (v_{i+1} \cup \cdots \cup v_{k})}
		\end{equation}
		where $\widetilde{S}_{a_{i}}=S_{a_i}\setminus P_{a_i}$, vertices $v_i=v_{\sigma_i}\in \sigma_i$ are fixed,  and $\theta_{i,k}=1$ when $i=k$ or otherwise
		\begin{multline} \label{eq: theta_ik}
			\theta_{i,k}=
			(-1)^{k-i +|J_i|(p_{i+1}+\cdots +p_k)+|J_{i+1}|(p_{i+2}+\cdots +p_k)+
				\cdots + |J_{k-1}|p_k} \\
			\cdot \varepsilon(v_{i+1},\sigma_{i+1}) \ldots \varepsilon(v_k, \sigma_k).
		\end{multline}
		For any $\sigma_i \in \widetilde{S}_{a_i}$ and $\sigma_k\in \widetilde{S}_{a_k}$, $\sigma_i\cup \sigma_k\in \mathcal{K}$ and so  $\sigma_i \cup \cdots \cup \sigma_k \setminus (v_{i+1} \cup \cdots \cup v_{k}) \in \mathcal{K}$. 
		Since every coefficient $c_{\sigma_i}$  is non-zero and each $\chi_{\sigma_i \cup \cdots \cup \sigma_k \setminus (v_{i+1} \cup \cdots \cup v_{k})}$ is a different basis element of $C^{p_i + \cdots + p_k}(\mathcal{K}_{J_i \cup \cdots \cup J_k})$, the cochain $a_{i,k}$ is not trivial.

		We will verify that $d(a_{i,k})=\sum_{r=i}^{k-1} \doubleoverline{a_{i,r}} \cdot a_{r+1,k}$.	
		By the definition of the coboundary map, 
		\begin{equation}\label{eq: d(a_ik) for joins}
			\begin{multlined}
				d(a_{i,k}) = \sum_{\sigma_i \in S_{a_i}}  \sum_{\sigma_{i+1} \in \widetilde{S}_{a_{i+1}}} \cdots \sum_{\sigma_k \in \widetilde{S}_{a_k}} c_{\sigma_i} \ldots c_{\sigma_k}\; \theta_{i,k} 
				\cdot \\
				\left(\sum_{j\in B 
				} 
				\varepsilon(j, j \cup \sigma_i\cup \cdots \cup \sigma_k\setminus (v_{i+1}\cup \cdots \cup v_{k}) ) \;
				\chi_{j \cup \sigma_i\cup \cdots \cup \sigma_k\setminus (v_{i+1}\cup \cdots \cup v_{k})}\right)
			\end{multlined}
		\end{equation}
		where $B$ is the set $\{j\in J_i\cup \cdots \cup J_{k} \setminus (\sigma_i\cup \cdots \cup \sigma_k \setminus (v_{i+1}\cup \cdots \cup v_{k}))\suchthat  j \cup \sigma_i\cup \cdots \cup \sigma_k\setminus (v_{i+1}\cup \cdots \cup v_{k})\in \mathcal{K} \}$.
		First we show that the only non-zero summands are when $j\in v_{i+1}\cup \cdots \cup v_{k}$. 
		For fixed $\sigma_i, \ldots, \sigma_k$, suppose that there is a vertex $ j\in J_i\cup \cdots \cup J_k\setminus  (\sigma_i\cup \cdots \cup \sigma_k)$ such that $j \cup \sigma_i\cup \cdots \cup \sigma_k\setminus (v_{i+1}\cup \cdots \cup v_k) \in \mathcal{K}$.
		So $j\notin v_{i+1}\cup \cdots \cup v_{k}$.
		Consider two cases, either $j\in J_i$ or $j\in J_l\setminus \sigma_l$ for $l\in \{i+1, \ldots, k\}$.
		
		(\romannumeral 1) In the first case, $j\in J_i$.
		By the definition of the coboundary map and since $a_i$ is a cocycle, 
		\begin{equation*}
			d(a_i)= \sum_{\sigma\in S_{a_i}} c_{\sigma} 
			\sum_{j\in J_i\setminus V(\sigma)} \varepsilon(j, j\cup \sigma)\chi_{j\cup \sigma}=0.
		\end{equation*}
		We extend this sum by taking the union of each $j\cup \sigma$ with $\sigma_{i+1} \cup \cdots \cup \sigma_k\setminus (v_{i+1}\cup \cdots \cup v_{k})$. 
		Since $\sigma_l\notin P_{a_l}$ for every $l\in \{i+1, \ldots, k\}$,
		$j\cup \sigma\cup \sigma_{i+1} \cup \cdots \cup \sigma_k\setminus (v_{i+1}\cup \cdots \cup v_{k})\in \mathcal{K}$ for any $j\cup \sigma\in \mathcal{K}_{J_i}$. 
		Hence
		\begin{multline*}
			\sum_{\sigma\in S_{a_i}} c_{\sigma} 
			\sum_{j\in J_i\setminus V(\sigma)} \varepsilon(j, 
			j\cup \sigma \cup \sigma_{i+1} \cup \cdots \cup \sigma_k\setminus (v_{i+1}\cup \cdots \cup v_{k})) \cdot \\ \cdot
			\chi_{j\cup \sigma\cup \sigma_{i+1} \cup \cdots \cup \sigma_k\setminus (v_{i+1}\cup \cdots \cup v_{k})}=0.
		\end{multline*}
		
		(\romannumeral 2) In the second case, $j\in J_l$ for  $l\in \{i+1, \ldots, k\}$, so $j\cup \sigma_l\setminus v_l\in \mathcal{K}_{J_l}$ and hence $j\cup \sigma_l \setminus v_{l} \in P_{\sigma_l}\subset P_{a_l}$.  
		By Construction~\ref{cons: simplicial complex for n-Massey}, $\sigma_i\cup j\cup \sigma_l \setminus v_{l}\notin \mathcal{K}$.
		Hence $j \cup \sigma_i\cup \cdots \cup \sigma_k\setminus (v_{i+1}\cup \cdots \cup v_k)\notin\mathcal{K}$ for any $ j\in J_i\cup \cdots \cup J_k\setminus  (v_{i+1}\cup \cdots \cup v_{k})$.
		
		Since the only non-zero summands in \eqref{eq: d(a_ik) for joins} are when $j\in v_{i+1}\cup \cdots \cup v_{k}$,
		$d(a_{i,k})$ reduces to
		\begin{multline*}
			d(a_{i,k}) = \sum_{\sigma_i \in S_{a_i}}  \sum_{\sigma_{i+1} \in \widetilde{S}_{a_{i+1}}} \cdots \sum_{\sigma_k \in \widetilde{S}_{a_k}}
			c_{\sigma_i} \ldots c_{\sigma_k} \theta_{i,k}  \cdot\\
			\cdot \!\!\! \sum_{\substack{j\in v_{i+1}\cup \cdots \cup v_{k} | \\ j \cup \sigma_i\cup \cdots \cup \sigma_k\setminus (v_{i+1}\cup \cdots \cup v_{k})\in \mathcal{K}}} \!\!\!\!\!\!\!\!\!\!\!\!\!\!\!
			\varepsilon(j, j \cup \sigma_i\cup \cdots \cup \sigma_k\setminus (v_{i+1}\cup \cdots \cup v_{k})) \; 
			\chi_{j \cup \sigma_i\cup \cdots \cup \sigma_k\setminus (v_{i+1}\cup \cdots \cup v_{k})}.
		\end{multline*}
		Denote $j\in v_{i+1}\cup \cdots \cup v_{k}$ by $v_{r+1}$ for $r\in\{i, \ldots, k-1\}$, and rewrite $d(a_{i,k})$ as
		\begin{equation}\label{eq: derived d(a_ik) for joins}
			\begin{multlined}[0.9\displaywidth]
				d(a_{i,k}) = \sum_{r=i}^{k-1} \theta_{i,k}
				\sum_{\sigma_i \in S_{a_i}}  \sum_{\sigma_{i+1} \in \widetilde{S}_{a_{i+1}}} \cdots \sum_{\sigma_k \in \widetilde{S}_{a_k}}
				c_{\sigma_i} \ldots c_{\sigma_k} \cdot \\
				\cdot \varepsilon (v_{r+1}, \sigma_i \cup \cdots \cup \sigma_k \setminus (v_{i+1} \cup \cdots \cup \hat{v}_{r+1} \cup \cdots\cup v_k))  \cdot\\\cdot
				\chi_{\sigma_i \cup \cdots \cup \sigma_k \setminus (v_{i+1} \cup \cdots \cup \hat{v}_{r+1} \cup \cdots\cup v_k)} 
			\end{multlined}
		\end{equation}
		where $\hat{v}_{r+1}$ denotes that $v_{r+1}$ is deleted from the sequence $v_{i+1}, \ldots, v_k$.
		
		To show that $d(a_{i,k}) =\sum_{r=i}^{k-1} \doubleoverline{a_{i,r}} \cdot a_{(r+1),k}$, we write out $a_{i,r}$ and $a_{(r+1),k}$ so that $\sum_{r=i}^{k-1} \doubleoverline{a_{i,r}} \cdot a_{(r+1),k}$ is
		\begin{align*}
			\sum_{r=i}^{k-1} &
			(-1)^{1+\mydeg(a_{i,r})}
			\left(\sum_{\sigma_i \in S_{a_i}}  \sum_{\sigma_{i+1} \in \widetilde{S}_{a_{i+1}}} \cdots \sum_{\sigma_r \in \widetilde{S}_{a_r}}
			c_{i,r}\,
			\chi_{\sigma_i \cup \cdots \cup \sigma_r \setminus (v_{i+1} \cup \cdots \cup v_{r})}\right) \cdot\\
			&\cdot\left(\sum_{\sigma_{r+1} \in S_{a_{r+1}}}  \sum_{\sigma_{r+2} \in \widetilde{S}_{a_{r+2}}} \cdots \sum_{\sigma_k \in \widetilde{S}_{a_k}}
			c_{r+1, k}\,
			\chi_{\sigma_{r+1} \cup \cdots \cup \sigma_k \setminus (v_{r+2} \cup \cdots \cup v_{k})}\right)
		\end{align*}
		where $c_{i,r}=c_{\sigma_i}  \ldots c_{\sigma_r} \, \theta_{i,r}$ and $c_{r+1, k}=c_{\sigma_{r+1}}  \ldots c_{\sigma_k} \, \theta_{{r+1},k}$.
		For any $\sigma_{r+1}\in S_{a_{r+1}}\setminus \widetilde{S}_{a_{r+1}}$, by definition  $\sigma_{r+1}\in P_{a_{r+1}}$ and $\sigma_i\cup \sigma_{r+1}\notin \mathcal{K}$. 
		Therefore $\left(\sigma_i \cup \cdots \cup \sigma_r \setminus (v_{i+1} \cup \cdots \cup v_{r})\right) \cup  \left(\sigma_{r+1} \cup \cdots \cup \sigma_k \setminus (v_{r+2} \cup \cdots \cup v_{k})\right) \in~\mathcal{K}$ only if $\sigma_{r+1}\in \widetilde{S}_{a_{r+1}}$. 
		Then by expanding the above expression and using the sign from Lemma~\ref{lem: multiplication of general cochains}, $\sum_{r=i}^{k-1} \doubleoverline{a_{i,r}} \cdot a_{(r+1),k}$ is
		\begin{equation}\label{eq: a_ir a_{r+1}k for joins}
			\begin{multlined}
				\sum_{r=i}^{k-1} 
				\sum_{\sigma_i \in S_{a_i}}  \sum_{\sigma_{i+1} \in \widetilde{S}_{a_{i+1}}} \cdots \sum_{\sigma_k \in \widetilde{S}_{a_k}}  
				(-1)^{1+\mydeg(a_{i,r})+ |J_i \cup \cdots \cup J_{r}|(p_{r+1} + \cdots + p_k+1)} \cdot \\
				\cdot 
				c_{\sigma_i}  \ldots c_{\sigma_k} \;  
				\theta_{i,r} \; \theta_{{r+1},k} \;
				\chi_{\sigma_i \cup \cdots \cup \sigma_k \setminus (v_{i+1} \cup \cdots \cup \hat{v}_{r+1} \cup \cdots \cup v_{k})}.
			\end{multlined}
		\end{equation}
		Since $\mydeg(a_{i,r})=|J_i \cup \cdots \cup J_r| +p_i + \cdots +p_r+1$, 
		\begin{align*}
			(-1)^{1+\mydeg(a_{i,r})+|J_i \cup \cdots \cup J_{r}|(p_{r+1} + \cdots + p_k+1)} = 
			(-1)^{(p_i + \cdots +p_r) + |J_i \cup \cdots \cup J_{r}|(p_{r+1} + \cdots + p_k)}.
		\end{align*}
		We next prove that \eqref{eq: derived d(a_ik) for joins} is equal to \eqref{eq: a_ir a_{r+1}k for joins} by showing that 
		\begin{multline}\label{eq: d(a_ik) and join}
			\theta_{i,k} \; \varepsilon (v_{r+1}, \sigma_i \cup \cdots \cup \sigma_k \setminus (v_{i+1} \cup \cdots \cup \hat{v}_{r+1} \cup \cdots\cup v_k)) 
			\\
			= 
			(-1)^{(p_i + \cdots +p_r) + |J_i \cup \cdots \cup J_{r}|(p_{r+1} + \cdots + p_k)} \theta_{i,r} \; \theta_{{r+1},k}. 
		\end{multline}
		Since
		\[
		\theta_{i,r}=(-1)^{r-i + |J_i|(p_{i+1}+\cdots +p_r)+ \cdots + |J_{r-1}|p_r} \; \varepsilon(v_{i+1},\sigma_{i+1}) \cdots \varepsilon(v_r, \sigma_r)
		\]
		and 
		\[
		\theta_{r+1,k}=(-1)^{k-r-1 + |J_{r+1}|(p_{r+2}+\cdots +p_k) + \cdots + |J_{k-1}|p_k} \; \varepsilon(v_{r+2},\sigma_{r+2}) \cdots \varepsilon(v_k, \sigma_k)
		\] 
		the right hand side of \eqref{eq: d(a_ik) and join} becomes
		\begin{multline*}
			(-1)^{k-i-1 + (p_i + \cdots +p_r) +
				|J_i|(p_{i+1}+\cdots +p_k) +
				|J_{i+1}|(p_{i+2}+\cdots +p_k) +  \cdots +|J_{k-1}|p_k}  \\
			\cdot \varepsilon(v_{i+1}, \sigma_{i+1}) \ldots \varepsilon(v_{r},\sigma_{r}) \varepsilon(v_{r+2},\sigma_{r+2})   \ldots \varepsilon(v_k, \sigma_k). 
		\end{multline*}
		This is simplified as 
		\begin{equation}\label{eq: reduced sign of a_ir a_{r+1}k in joins}
			(-1)^{p_i + \cdots +p_r-1} \;\varepsilon(v_{r+1},\sigma_{r+1}) \;\theta_{i,k}.
		\end{equation}
		
		Next consider the left hand side of \eqref{eq: d(a_ik) and join}. 
		For any $r\in \{i,\ldots, k-1\}$, suppose that $v_{r+1}\in \sigma_{r+1}$ is the $l$th vertex in the vertex set of $ \sigma_i \cup \cdots \cup \sigma_k \setminus (v_{i+1}\cup \cdots \cup \hat{v}_{r+1}\cup \cdots \cup v_k)$. Then by \eqref{eq: varepsilon},
		\[
		\varepsilon(v_{r+1}, \sigma_i \cup \cdots \cup \sigma_k \setminus (v_{i+1}\cup \cdots \cup \hat{v}_{r+1}\cup \cdots \cup v_k))=(-1)^{l-1}.
		\]	
		Since $v_{r+1}\in \sigma_{r+1}$, $l$ is given by
		\[
		l=|\sigma_i|+(|\sigma_{i+1}|-1)+\cdots+(|\sigma_{r}|-1)+l_{r+1}
		\]
		where $l_{r+1}$ is the position of $v_{r+1}$ in $\sigma_{r+1}$.
		Since $|\sigma_i|=p_i+1$ for every $i$, $l=(p_i+1) +p_{i+1}+\cdots + p_{r}+l_{r+1}$, and hence
		\begin{equation}\label{eq: breaking down varepsilon joins}
			\varepsilon(v_{r+1}, \sigma_i \cup \cdots \cup \sigma_k \setminus (v_{i+1}\cup \cdots \cup \hat{v}_{r+1}\cup \cdots \cup v_k))=(-1)^{p_i+\cdots+p_{r} +1} \; \varepsilon(v_{r+1}, \sigma_{r+1}).
		\end{equation}
		Thus \eqref{eq: d(a_ik) and join} may be rewritten as
		$
		(-1)^{p_i+\cdots+p_{r} +1} \; \theta_{i,k} \; \varepsilon(v_{r+1}, \sigma_{r+1}),
		$
		which is equal to \eqref{eq: reduced sign of a_ir a_{r+1}k in joins}.
		Hence \eqref{eq: derived d(a_ik) for joins} is equal to \eqref{eq: a_ir a_{r+1}k for joins} so $d(a_{i,k})=\sum_{r=i}^{k-1} \doubleoverline{a_{i,r}} \cdot a_{(r+1),k}$, which proves that $(a_{i,k})$ corresponds to a defining system  for $\langle \alpha_1, \ldots, \alpha_n \rangle$.
	\end{proof}
	
	We aim to show that the constructed $n$-Massey product $\langle \alpha_1, \ldots, \alpha_n \rangle$ is non-trivial. 
	We build a cycle $x\in C_{p_1 + \ldots + p_n +1}(\mathcal{K}_{J_1 \cup \cdots \cup J_n})$ and show that for any $[\omega]\in \langle \alpha_1, \ldots, \alpha_n \rangle$ there is a cycle $x'$ homologous to $x$ such that $\omega(x')\neq 0$.
	This will conclude that $[\omega]\neq0$.

	\begin{cons}\label{cons: joins x}
		Fix $\sigma_1\in S_{a_1}$, $\sigma_i\in \widetilde{S}_{a_i}=S_{a_i}\setminus P_{a_i}$ for $2\leqslant i < n$ and $\sigma_n\in P_{a_n}$.
		Since $\alpha_1\in \widetilde{H}^{p_1}(\mathcal{K}_{J_1})$ is non-zero, there is a cycle $x_1\in C_{p_1}(\mathcal{K}_{J_1})$ such that $a_1(x_1)\neq 0$.
		We write the cycle $x_1$ as
		\[
		x_1=\sum_{\tilde{\sigma}_1\in S_{x_1}} c_{\tilde{\sigma}_1} \Delta_{\tilde{\sigma}_1}
		\]
		for a collection of $p_1$-simplices $S_{x_1}\subset \mathcal{K}_{J_1}$ and non-zero coefficients $c_{\tilde{\sigma}_1}$, where $\Delta_{\tilde{\sigma}_1}$ is a basis element of $C_{p_1}(\mathcal{K}_{J_1})$.
		
		After the star deletion of $\sigma_2\cup \sigma_n$, the boundary complex $\partial(\sigma_2\cup \sigma_n)$  is contained in $\mathcal{K}$.
		Let $x_2\in C_{p_2+p_n}(\partial(\sigma_2\cup \sigma_n))$ be the cycle
		\[
		x_2=\sum_{w_2\in \sigma_2\cup \sigma_n} c_{w_2} \Delta_{\sigma_2\cup \sigma_n \setminus w_2}
		\]
		for vertices $w_2\in \sigma_2\cup \sigma_n$ and non-zero coefficients $c_{w_2}$.
		Similarly for $3\leqslant i \leqslant n-1$, 
		let $x_i \in  C_{p_i-1}(\partial(\sigma_i))$ be the cycle given by
		\[
		x_i=\sum_{w_i\in \sigma_i} c_{w_i} \Delta_{\sigma_i \setminus w_i}
		\]
		for vertices $w_i\in \sigma_i$ and non-zero coefficients $c_{w_i}$.
		
		Let $x\in C_{p_1 + \ldots + p_n +1}(\mathcal{K}_{J_1 \cup \cdots \cup J_n})$ be the chain
		\begin{multline*}\label{eq: form of x joins}
			x=\sum_{\tilde{\sigma}_1\in S_{x_1}}  \sum_{w_2\in \sigma_2\cup \sigma_n} \sum_{w_3\in \sigma_3} \cdots \sum_{w_{n-1}\in \sigma_{n-1}}
			c_{\tilde{\sigma}_1} c_{w_2} \cdots c_{w_{n-1}} \cdot\\\cdot
			\Delta_{\tilde{\sigma}_1 \cup \sigma_2 \cup \cdots \cup \sigma_{n-1} \cup \sigma_n\setminus (w_2 \cup \cdots \cup w_{n-1})}.
		\end{multline*}
		Let $S_x$ be the support of $x$, consisting of simplices 
		\begin{equation}\label{eq: simplices of x joins}
			\sigma=\tilde{\sigma}_1 \cup \sigma_2 \cup \cdots \cup \sigma_{n-1} \cup \sigma_n\setminus (w_2 \cup \cdots \cup w_{n-1})
		\end{equation}
		for a $p_1$-simplex $\tilde{\sigma}_1\in S_{x_1}$, and a choice of vertices $w_2 \in \sigma_2 \cup \sigma_n$,  $w_i \in \sigma_i$ for $3\leqslant i \leqslant n-1$.
	\end{cons}
	
	\begin{lemma}\label{lem: joins cycle x}
		The cochain $x\in C_{p_1 + \ldots + p_n +1}(\mathcal{K}_{J_1 \cup \cdots \cup J_n})$ is a cycle.
	\end{lemma}
	\begin{proof}
		We show that $x$ is a cycle by explicitly calculating $\partial(x)$.
		By the definition of the boundary map, 
		\begin{align*}
			\partial(x)&=
			\sum_{\tilde{\sigma}_1\in S_{x_1}}  \sum_{w_2\in \sigma_2\cup \sigma_n} \sum_{w_3\in \sigma_3} \cdots \sum_{w_{n-1}\in \sigma_{n-1}}
			\sum_{v \in \sigma} 
			\varepsilon(v, \sigma) \;c_{\tilde{\sigma}_1} c_{w_2} \cdots c_{w_{n-1}} 
			\Delta_{\sigma \setminus v} 
		\end{align*}
		where $\sigma\in S_x$ as in \eqref{eq: simplices of x joins}.
		Since $\tilde{\sigma}_1\subset J_1$, $\sigma_i\subset J_i$ for $2\leqslant i \leqslant n$, and $J_i\cap J_j=\myempty$ for $i\neq j$, any choice of  vertex $v\in \sigma$ is contained in a simplex $\tilde{\sigma}_1$ or $\sigma_i$ for $2\leqslant i \leqslant n$.
		If $v\in\tilde{\sigma}_1$, then $\varepsilon(v,\sigma)=\varepsilon(v,\tilde{\sigma}_1)$. 
		Also 
		if $v\in \sigma_i$ for $i>1$, then
		\begin{equation*}
			\varepsilon(v,\sigma)=
			\begin{cases}
				(-1)^{p_1+1}\; \varepsilon(v, \sigma_{2}) & \text{if } w_2\in \sigma_n \text{ and } i=2,  \\
				(-1)^{p_1+\cdots + p_{i-1}+2}\; \varepsilon(v, \sigma_{i}\setminus \tilde{w}_i) & \text{if } w_2\in \sigma_n \text{ and } i>2,  \\
				(-1)^{p_1+\cdots + p_{n-1}+1}\; \varepsilon(v, \sigma_{n}) & \text{if } w_2\in \sigma_2 \text{ and } i=n,  \\
				(-1)^{p_1+\cdots + p_{i-1}+1}\; \varepsilon(v, \sigma_{i}\setminus w_i) & \text{if } w_2\in \sigma_2 \text{ and } i<n
			\end{cases}
		\end{equation*}
		where $\tilde{w}_i=w_i$ for $1<i<n$, and $\tilde{w}_n=w_2$.
		We rewrite $\partial(x)$ as
		\small{\begin{align*}
			\partial(x)&=
			\sum_{\tilde{\sigma}_1\in S_{x_1}}  \sum_{w_2\in \sigma_2\cup \sigma_n} \sum_{w_3\in \sigma_3} \cdots \sum_{w_{n-1}\in \sigma_{n-1}}
			\: \sum_{i=1}^n \: \sum_{v \in \tilde{\sigma}_i\setminus \tilde{w}_i} 
			\varepsilon(v, \sigma) \;c_{\tilde{\sigma}_1} c_{w_2} \cdots c_{w_{n-1}} 
			\Delta_{\sigma \setminus v} 
		\end{align*}}
		where $\tilde{\sigma}_1\setminus \tilde{w}_1=\tilde{\sigma}_1$ and $\tilde{\sigma}_i=\sigma_i$ for $i>1$.
		Let $\Delta_{\sigma \setminus v|{J}}$ denote the restriction of $\Delta_{\sigma \setminus v}$ to its vertices in $J\subset V(\mathcal{K})$, where $V(\mathcal{K})$ is the vertex set of $\mathcal{K}$.
		Then	
		\begin{multline*}
			\partial(x)=
			\sum_{\tilde{\sigma}_1\in S_{x_1}}  \sum_{w_2\in \sigma_2\cup \sigma_n} \sum_{w_3\in \sigma_3} \cdots \sum_{w_{n-1}\in \sigma_{n-1}} \\
			\left( \sum_{i=1}^n \: \sum_{v \in \tilde{\sigma}_i\setminus \tilde{w}_i} 
			\varepsilon(v, \sigma) \;c_{\tilde{\sigma}_1} c_{w_2} \cdots c_{w_{n-1}} 
			(\Delta_{\sigma \setminus v|{J_i}}) (\Delta_{\sigma \setminus v|{V(\mathcal{K})\setminus J_i}})
			\right).
		\end{multline*}
		We rearrange $\partial(x)$ into four sums, one in which $v\in \tilde{\sigma}_1$, another for $v \in \sigma_2\cup \sigma_n \setminus w_2$, and two more when $v\in \sigma_i\setminus w_i$ for $3\leqslant i\leqslant n-1$ where either $w_2\in \sigma_2$ or $w_2\in \sigma_2$. 
		Then expanding $\varepsilon(v,\sigma)$, $\partial(x)$ is 
		\begin{align*}
			&\begin{multlined}[t][0.9\textwidth]
				\sum_{w_2\in \sigma_2\cup \sigma_n} \sum_{w_3\in \sigma_3} \cdots \sum_{w_{n-1}\in \sigma_{n-1}} c_{w_2} \cdots c_{w_{n-1}} \cdot\\
				\cdot (\Delta_{\sigma\setminus v|{V(\mathcal{K})\setminus J_1}}) 
				\left( \sum_{\tilde{\sigma}_1\in S_{x_1}}  \sum_{v \in \tilde{\sigma}_1} 
				\varepsilon(v, \tilde{\sigma}_1) \,c_{\tilde{\sigma}_1} 
				(\Delta_{\sigma \setminus v|{J_1}}) 
				\right)+
			\end{multlined}\\
			&\begin{multlined}[t][0.9\textwidth]
				+ \sum_{\tilde{\sigma}_1\in S_{x_1}}\sum_{w_3\in \sigma_3} \cdots \!\!\! \sum_{w_{n-1}\in \sigma_{n-1}} \!\!\! c_{\tilde{\sigma}_1} c_{w_3} \cdots c_{w_{n-1}} (-1)^{p_1+p_3+\cdots + p_{n-1}+1} (\Delta_{\sigma \setminus v|{V(\mathcal{K})\setminus J_2 \cup J_n}}) \cdot \\ \cdot
			     \left( \sum_{w_2\in \sigma_2\cup \sigma_n}  \sum_{v \in \sigma_2\cup \sigma_n \setminus w_2} 
				\varepsilon(v, \sigma_2\cup \sigma_n \setminus w_2) \,c_{w_2} 
				(\Delta_{\sigma \setminus v|{J_2\cup J_n}}) 
				\right)+
			\end{multlined}\\
			&\begin{multlined}[0.9\textwidth]
				+ \sum_{\tilde{\sigma}_1\in S_{x_1}}\sum_{w_2\in \sigma_2} \sum_{w_3\in \sigma_3} \cdots \sum_{w_{n-1}\in \sigma_{n-1}} c_{\tilde{\sigma}_1} c_{w_2} \cdots c_{w_{n-1}} \cdot
				\\ \cdot \left( 
				\sum_{i=3}^{n-1} (-1)^{p_1+\cdots + p_{i-1}+1} 
				(\Delta_{\sigma \setminus v|{V(\mathcal{K})\setminus J_i}}) \left(   \sum_{v \in \sigma_i \setminus w_i} 
				\varepsilon(v, \sigma_i \setminus w_i) 	(\Delta_{\sigma \setminus v|{J_i}}) 
				\right) \right)+
			\end{multlined}\\
			&\begin{multlined}[0.9\textwidth]
				+
				\sum_{\tilde{\sigma}_1\in S_{x_1}}\sum_{w_2\in \sigma_n} \sum_{w_3\in \sigma_3} \cdots \sum_{w_{n-1}\in \sigma_{n-1}}
				c_{\tilde{\sigma}_1} c_{w_2} \cdots c_{w_{n-1}} \cdot \\ \cdot\left( 
				\sum_{i=3}^{n-1} (-1)^{p_1+\cdots + p_{i-1}+2} 
				(\Delta_{\sigma\setminus v|{V(\mathcal{K})\setminus J_i}}) \left(   \sum_{v \in \sigma_i \setminus w_i} 
				\varepsilon(v, \sigma_i \setminus w_i) 	(\Delta_{\sigma\setminus v|{J_i}}) 
				\right) \right)_.
			\end{multlined}
		\end{align*}
		Each sum can be written in terms of $\partial(x_i)$, that is,
		\begin{align*}
			\partial(x)=&
			\begin{multlined}[t]
				\sum_{w_2\in \sigma_2\cup \sigma_n} \sum_{w_3\in \sigma_3} \cdots \sum_{w_{n-1}\in \sigma_{n-1}} 
				c_{w_2} \cdots c_{w_{n-1}} (\Delta_{\sigma\setminus v|{V(\mathcal{K})\setminus J_1}}) 
				\, \partial(x_1)+
			\end{multlined}\\
			&\begin{multlined}[0.9\textwidth]
				+ \sum_{\tilde{\sigma}_1\in S_{x_1}}\sum_{w_3\in \sigma_3} \cdots \sum_{w_{n-1}\in \sigma_{n-1}} c_{\tilde{\sigma}_1} c_{w_3} \cdots c_{w_{n-1}}\cdot \\ \cdot(-1)^{p_1+p_3\cdots + p_{n-1}+1} (\Delta_{\sigma\setminus v|{V(\mathcal{K})\setminus J_2 \cup J_n}}) 
				\, \partial(x_2)
				+
			\end{multlined}\\
			&\begin{multlined}[0.9\textwidth]
				+ \sum_{\tilde{\sigma}_1\in S_{x_1}}\sum_{w_2\in \sigma_2} \sum_{i=3}^{n-1}
				\sum_{w_3\in \sigma_3} \cdots \widehat{\sum_{w_i\in \sigma_i}}\cdots \sum_{w_{n-1}\in \sigma_{n-1}} c_{\tilde{\sigma}_1} c_{w_2} \cdots \widehat{c_{w_i}} \cdots c_{w_{n-1}} \cdot
				\\ \cdot \left( 
				(-1)^{p_1+\cdots + p_{i-1}+1} 
				(\Delta_{\sigma\setminus v|{V(\mathcal{K})\setminus J_i}})
				\, \partial(x_i) \right)+
			\end{multlined}\\
			&\begin{multlined}[0.9\textwidth]
				+
				\sum_{\tilde{\sigma}_1\in S_{x_1}}\sum_{w_2\in \sigma_n} \sum_{i=3}^{n-1}
				\sum_{w_3\in \sigma_3} \cdots \widehat{\sum_{w_i\in \sigma_i}}\cdots \sum_{w_{n-1}\in \sigma_{n-1}} c_{\tilde{\sigma}_1} c_{w_2} \cdots \widehat{c_{w_i}} \cdots c_{w_{n-1}} 
				\cdot \\ \cdot \left( 
				(-1)^{p_1+\cdots + p_{i-1}+2} 
				(\Delta_{\sigma \setminus v|{V(\mathcal{K})\setminus J_i}})
				\, \partial(x_i) \right)
			\end{multlined}
		\end{align*}
		where $\ \widehat{\ }\ $ denotes omission. 
		Since $\partial(x_i)=0$ for every $i$, $x$ is a cycle as well.
	\end{proof}
	
	\begin{exmp} \label{ex: star deletion x cycle example}
		Let $\mathcal{K}$ be the simplicial complex in Figure~\ref{fig: example of s}, where the simplices $\sigma_1 \cup \sigma_2'$, $\sigma_2 \cup \sigma_3'$ were star deleted and $S_{a_1}=\{\sigma_1\}$, $S_{a_2}=\{\sigma_2\}$, $S_{a_3}=\{\sigma_3\}$, $P_{a_2}=\{\sigma_2'\}$, $P_{a_3}=\{\sigma_3'\}$. The cycle $x$ is supported on simplices of the form
		\begin{equation*}
			\sigma=\tilde{\sigma}_1 \cup \sigma_2  \cup \sigma_3'\setminus (w_2)
		\end{equation*}
		where $\tilde{\sigma}_1$ is either $\sigma_1$ or $\sigma_1'$ and $w_2 \in \sigma_2 \cup \sigma_3'$. 	
		Therefore $S_x$ contains $\sigma_1\cup \sigma_2$, $\sigma_1'\cup\sigma_2$, $\sigma_1'\cup \sigma_3'$ and $\sigma_1 \cup \sigma_3'$, as shown in Figure~\ref{fig: cycle x}. 
		
		If $a_1=\chi_{\sigma_1}\in C^0(\mathcal{K}_{\sigma_1,\sigma_1'})$, $a_2=\chi_{\sigma_2}\in C^0(\mathcal{K}_{\sigma_2,\sigma_2'})$ and $a_3=\chi_{\sigma_3}\in C^0(\mathcal{K}_{\sigma_3,\sigma_3'})$, then the rest of the defining system constructed in Proposition~\ref{prop: joins Massey defined} is
		$a_{12}=-\chi_{\sigma_1}$ and $a_{23}=-\chi_{\sigma_2}$. 
		The associated cocycle to this defining system is
		\[
		\omega=-\chi_{\sigma_1\cup\sigma_3} - \chi_{\sigma_1\cup\sigma_2}.
		\]
		There is exactly one simplex $\sigma_1\cup\sigma_2=S_x\cap S_\omega$. 
		So by evaluating $\omega$ on $x$, $\omega(x)\neq 0$.
		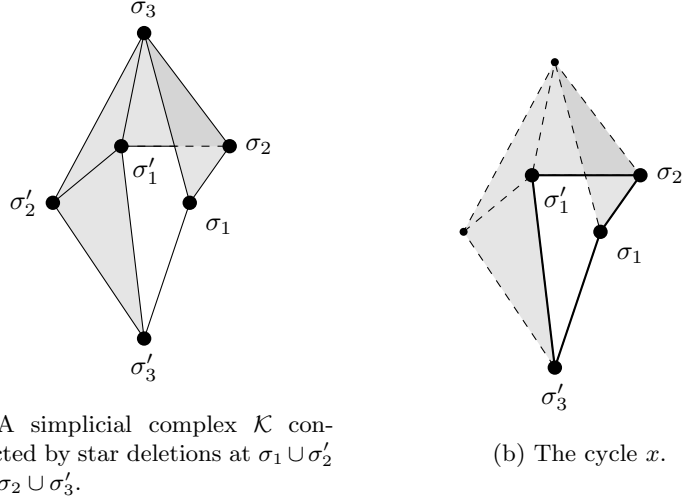
\begin{figure}[ht]
			\centering
			\begin{minipage}{0.4\textwidth}
				\centering
				\begin{tikzpicture} [scale=1.5, inner sep=2mm]
				\coordinate (4) at (0,0);
				\coordinate (1) at (1.2,0);
				\coordinate (2) at (0.6,0.5);
				\coordinate (3) at (1.55,0.5);
				\coordinate (5) at (0.8,1.5);
				\coordinate (6) at (0.8,-1.2);
				
				\fill[lightgray, fill opacity=0.4] (3)--(5)--(4)--(6)--(2)--(3);
				\draw (1.1, 0.5) -- (2) -- (5);\draw (4) -- (2)  node[below right, inner sep=1.4mm] {$\sigma_1'$}-- (6);
				\fill[lightgray, fill opacity=0.4] (1)--(5)--(3);
				
				\draw (1) node[below right] {$\sigma_1$} -- (3) node[right] {$\sigma_2$} -- (5) node[above] {$\sigma_3$} -- (4)  node[left] {$\sigma_2'$} -- (6) node[below] {$\sigma_3'$} -- (1)--(5); 
				\draw[dashed] (2) -- (3);
				
				\foreach \i in {1,...,6} {\fill (\i) circle (1.8pt);}; 
				\end{tikzpicture} 
				\subcaption{A simplicial complex $\mathcal{K}$ constructed by star deletions at  $\sigma_1 \cup \sigma_2'$ and $\sigma_2 \cup \sigma_3'$.}
				\label{fig: example of s}
			\end{minipage}\qquad
			\begin{minipage}{0.4\textwidth}
				\centering
				\begin{tikzpicture} [scale=1.5, inner sep=2mm]
				\coordinate (1) at (1.2,0);
				\coordinate (2) at (0.6,0.5);
				\coordinate (3) at (1.55,0.5);
				\coordinate (4) at (0,0);
				\coordinate (5) at (0.8,1.5);
				\coordinate (6) at (0.8,-1.2);
				\coordinate (bottom) at (0.5, -1.5);
				\coordinate (top) at(0.8, 1.9);
				
				\draw[white] (top)--(bottom);	
				\fill[lightgray, fill opacity=0.4] (3)--(5)--(4)--(6)--(2)--(3);
				\draw[dashed] (3) -- (2) -- (5);\draw[dashed] (4) -- (2) -- (6);
				\fill[lightgray, fill opacity=0.4] (1)--(5)--(3);
				
				\draw[dashed] (1)  -- (3) -- (5)  -- (4)   -- (6)  -- (1)--(5); 
				\draw[dashed] (2) -- (3);
				
				\foreach \i in {1,...,6} {\fill (\i) circle (1pt);}; 
				
				\draw[thick] (6) node[below] {$\sigma_3'$}--(1) node[below right] {$\sigma_1$} -- (3) node[right] {$\sigma_2$} -- (2)  node[below right, inner sep=1.4mm] {$\sigma_1'$} -- (6); 
				
				
				\foreach \i in {1,...,3, 6} {\fill (\i) circle (1.8pt);}; 
				\end{tikzpicture} 
				\subcaption{The cycle $x$.} \label{fig: cycle x} 
			\end{minipage}
			\caption{Example of the cycle $x$ defined in Construction~\ref{cons: joins x}}
		\end{figure}
	\end{exmp}

	\begin{prop} \label{prop: joins Massey non-trivial}
		The $n$-Massey product $\langle \alpha_1, \ldots, \alpha_n \rangle \subset H^*(\mathcal{Z}_\mathcal{K})$ is non-trivial.
	\end{prop}
	\begin{proof}

		For any $[\omega]\in \langle \alpha_1, \ldots, \alpha_n \rangle$, we consider a corresponding cocycle $\omega \in C^{p_1 + \cdots + p_n +1}(\mathcal{K}_{J_1 \cup \cdots \cup J_n})$  with the cycle $x$ from Construction~\ref{cons: joins x} and aim to show that $S_{\omega}\cap S_{x}$ contains only one simplex.
		This implies that $\omega(x)$ is non-zero, and therefore $[\omega]\neq0$.
		
		First we define a subcollection of simplices in $S_\omega$.
		Let $(a_{i,k})$ be any defining system of $\langle \alpha_1, \ldots, \alpha_n \rangle$.
		Let $S_{a_{i,k}}$ be the support of $a_{i,k}$ so that
		\begin{equation*}
			a_{i,k}=\sum_{\sigma \in S_{a_{i,k}}} c_{\sigma} \chi_{\sigma}
		\end{equation*}	
		for non-zero coefficients $c_{\sigma}\in \mathbf{k}$. 
		The image of the coboundary map is a cochain obtained by adding a vertex to the simplices in $S_{a_{i,k}}$. 
		Since $d(a_{i,k})=\sum_{r=i}^{k-1} \doubleoverline{a_{i,r}}a_{r+1,k}$, 
		for any $\sigma_{i,r}\in S_{a_{i,r}}$ and $\sigma_{r+1,k} \in S_{a_{r+1,k}}$ there is a simplex $\sigma\in S_{a_{i,k}}$ and vertex $u_i\in \sigma_{i,r}\cup \sigma_{r+1,k}$ such that
		$\sigma=\sigma_{i,r}\cup \sigma_{r+1,k} \setminus u_i$.
		We extend this principle to say that there is a simplex $\sigma\in S_{a_{2,n}}$ such that 
		$\sigma=\sigma_2 \cup \cdots \cup \sigma_n \setminus (u_2 \cup \cdots \cup u_{n-1})$  
		for $\sigma_i\in \widetilde{S}_{a_i}=S_{a_i}\setminus P_{a_i}$  and vertices $u_i\in \sigma_2 \cup \cdots \cup \sigma_n$ for $2\leqslant i \leqslant n$, $u_i\neq u_j$.
		Let $\omega$ be the associated cocycle for this defining system, 
		\[
		\omega=\sum_{\tau \in S_{\omega} } c_\tau \chi_\tau
		\]
		for non-zero coefficients $c_{\tau}\in \mathbf{k}$.
		The support of the first summand $\doubleoverline{a}_1a_{2,n}$ of $\omega$ contains a simplex of the form
		\begin{equation}\label{eq: simplex supporting omega}
			\tau=\sigma_1 \cup \sigma_2 \cup \cdots \cup \sigma_{n-1} \cup \sigma_n \setminus (u_2 \cup \cdots \cup u_{n-1})
		\end{equation}
		for $\sigma_1\in S_{a_1}$, $\sigma_i\in \widetilde{S}_{a_i}$ and vertices $u_i\in \sigma_2 \cup \cdots \cup \sigma_n$ for $2\leqslant i \leqslant n$, $u_i\neq u_j$.
		Hence $\tau\in S_\omega$.
		
		We compare the simplices $\tau\in S_\omega$ in \eqref{eq: simplex supporting omega} and $\sigma\in S_x$ in \eqref{eq: simplices of x joins} where
		\[\sigma=\tilde{\sigma}_1 \cup \sigma_2 \cup \cdots \cup \sigma_{n-1} \cup \sigma_n'\setminus (w_2 \cup \cdots \cup w_{n-1})
		\]
		for $\tilde{\sigma}_1\in S_{x_1}$, $\sigma_i \in \widetilde{S}_{a_i}$ for $i\in \{2, \ldots, n-1\}$, $\sigma_n'\in P_{a_n}$ and a choice of vertices $w_2 \in \sigma_2 \cup \sigma_n$,  $w_i \in \sigma_i$ for $3\leqslant i \leqslant n-1$. 
		For $\sigma_1\in S_{a_1}$ and $\sigma_i\in \widetilde{S}_{a_i}$ for $2\leqslant i \leqslant n$, the simplex $\sigma_1\cup \cdots \cup \sigma_n\in \mathcal{K}$ was not removed by star deletion in Construction~\ref{cons: simplicial complex for n-Massey}.
		Both $\tau$ and $\sigma$ are $(p_1+\cdots +p_n+1)$-dimensional faces of $\sigma_1\cup \cdots \cup \sigma_n$.
		If there is no $\tau\in S_\omega$ and $\sigma\in S_x$ such that $\tau=\sigma$, then
		there is a cochain $b \in C^{p_1+\cdots +p_n}(\mathcal{K})$ whose support consists of $(p_1+\cdots+p_n)$-simplices contained in $\sigma_1\cup \cdots \cup \sigma_n$ and the support of $d(b)$ contains both $\tau$ and $\sigma$. 
		Let $\omega'=\omega+c_{\tau}c_{d(b),\sigma}d(b)$ where $c_{\tau}$ is the coefficient of $\tau\in S_{\omega}$ and $c_{d(b),\sigma}$ is the coefficient of $\sigma\in S_{d(b)}$. 
		Then $S_{\omega'}$ contains  $\sigma$ and does not contain $\tau$.
		Therefore $\sigma\in S_{\omega'}\cap S_x$. 
		However there could be other simplices in $S_{\omega'}\cap S_x$ that cancel, so we cannot conclude that $\omega'(x)$ is non-zero.
		To resolve this, we change the representatives of $[\omega]$ and $[x]$ so that there is only one term in their evaluation. 
		
		Suppose that there is $\tau'\in S_{\omega'}\cap S_x$, $\tau\neq\tau'$. 
		If $\link_{\mathcal{K}}(\tau')\neq \myempty$, then there is a $(p_1+\ldots+p_k+2)$-dimensional
		simplex $A \in \mathcal{K}_{J_1 \cup \cdots \cup J_n}$ containing $\tau'$ in its boundary.
		Suppose that $S_{\omega'}$ does not contain an additional face of $A$. Then replace $x$ by $x'$, where the simplex $\tau'\in S_x$ is replaced by the $(p_1+\ldots+p_k+1)$-simplices in $\partial(A)\setminus \tau'$ to form $S_{x'}$  as illustrated in Figure~\ref{fig: x and x' in joins}. Therefore $x'$ is the cycle $x-c_{\tau'} \, \epsilon(v, A)\, \partial(\Delta_A)$, where $c_{\tau'}$ is the coefficient of the summand $\Delta_{\tau'}$ in $x$, $v$ is the vertex such that $v\cup \tau'=A$, and $\epsilon(v, A)$ is the coefficient of $\Delta_\tau$ in $\partial(\Delta_A)$. Thus $[x]=[x']$ and $\tau'\notin S_{\omega'} \cap S_{x'}$.
		
		

		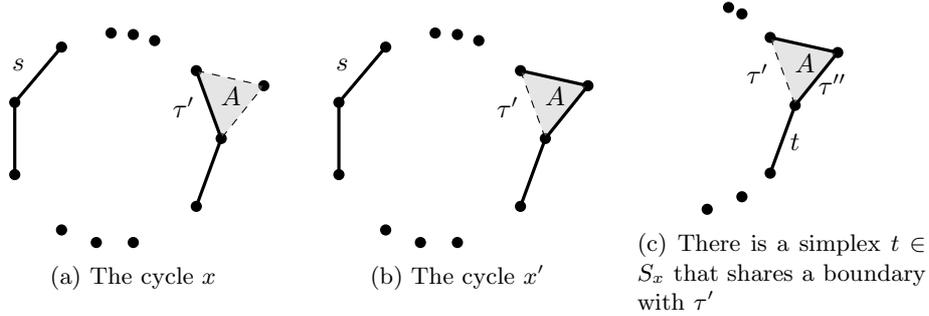
\begin{figure}[ht]
			\centering
			\begin{minipage}{0.3\textwidth}
				\centering
				\begin{tikzpicture}[scale=1.4, inner sep=2mm]
				\foreach \i in {1, 2, ..., 5, 1.7, 2.3, 6, 6.5, 7, 8, 9}
				\fill (\i*360/9:1) coordinate (\i) circle(1.5 pt);
				
				\fill (1.4,0.5) coordinate (a) circle(1.5 pt);
				
				\fill[lightgray, fill opacity=0.4](1) --(a)--(9) --(1);
				\draw[very thick] (3) -- (4) --(5); 
				\draw[very thick] (8)--(9) --(1)--(1.7);
				\draw[dashed] (1) --(a)--(9);
				
				\coordinate (s) at (-0.7,0.7);
				\draw (s) node[left ] {$s$};
				
				\coordinate (A) at (0.85, 0.4);
				\draw (A) node[right ] {$A$};
				
				\coordinate (t) at (0.9, 0.3);
				\draw (t) node[left ] {$\tau'$};
				
				\end{tikzpicture}
				\subcaption{The cycle $x$}
			\end{minipage}
			\quad
			\begin{minipage}{0.3\textwidth}
				\centering
				\begin{tikzpicture}[scale=1.4, inner sep=2mm]
				\foreach \i in {1, 2, ..., 5, 1.7, 2.3, 6, 6.5, 7, 8, 9}
				\fill (\i*360/9:1) coordinate (\i) circle(1.5 pt);
				
				\fill (1.4,0.5) coordinate (a) circle(1.5 pt);
				
				\fill[lightgray, fill opacity=0.4](1) --(a)--(9) --(1);
				\draw[very thick] (3) -- (4) --(5); 
				\draw[very thick] (8)--(9);
				\draw[very thick] (1) --(a)--(9);
				\draw[dashed] (1)--(9);
				
				\coordinate (s) at (-0.7,0.7);
				\draw (s) node[left ] {$s$};
				
				\coordinate (A) at (0.85, 0.4);
				\draw (A) node[right ] {$A$};
				
				\coordinate (t) at (0.9, 0.3);
				\draw (t) node[left ] {$\tau'$};
				
				\end{tikzpicture}
				\subcaption{The cycle $x'$}
			\end{minipage}
			\quad
			\begin{minipage}{0.3\textwidth}
				\centering
				\begin{tikzpicture}[scale=1.4, inner sep=2mm]
				\foreach \i in {1, 1.5, 1.7, 7, 7.5, 8, 9}
				\fill (\i*360/9:1) coordinate (\i) circle(1.5 pt);
				
				\fill (1.4,0.5) coordinate (a) circle(1.5 pt);
				
				\fill[lightgray, fill opacity=0.4](1) --(a)--(9) --(1);
				\draw[very thick] (8)--(9);
				\draw[very thick] (1) --(a)--(9);
				\draw[dashed] (1)--(9);
				
				\node at (1.35,0.2){$\tau''$};
				\node at (1.0, -0.35){$t$};
				
				\coordinate (A) at (0.85, 0.4);
				\draw (A) node[right ] {$A$};
				
				\coordinate (t) at (0.9, 0.3);
				\draw (t) node[left ] {$\tau'$};
				
				\end{tikzpicture}
				\subcaption{There is a simplex $t\in S_x$ that shares a boundary with $\tau'$} \label{fig: tau cap tau' contained in t}
			\end{minipage}
			\caption{If the link of $\tau'$ is non-empty, then the cycle $x$ can be changed to $x'$}
			\label{fig: x and x' in joins}
		\end{figure}
		
		Alternatively, suppose that $\link_{\mathcal{K}}(\tau')= \myempty$, or 
		$\link_{\mathcal{K}}(\tau')\neq \myempty$ and $S_{\omega'}$ contains an additional face $\tau''$ of $A$.
		Since $x$ is a cycle, there is another simplex $t\neq \tau'\in S_x$ such that $\tau'\cap t \neq \myempty$ 
		(as shown in Figure~\ref{fig: tau cap tau' contained in t}).
		Let $\omega''=\omega'-c_{\tau'}\; \varepsilon(\tau' \setminus \tau'\cap t, \tau')\; d(\chi_{\tau'\cap t})$  where $c_{\tau'}$ is the coefficient of the summand $\chi_{\tau'}$ in $\omega'$ and $\varepsilon(\tau' \setminus \tau'\cap t, \tau')$ is its coefficient in $d(\chi_{\tau'\cap t})$. 
		So $[\omega'']=[\omega']$ and $S_{\omega''}$ contains $t$ but $S_{\omega''}\cap S_x$ does not contain ${\tau'}$.
		
		
		
		By this process of replacing simplices in the intersection of the supports one-by-one, we obtain a cocycle $\omega'\in C^{p_1 + \ldots + p_n +1}(\mathcal{K}_{J_1 \cup \cdots \cup J_n})$ and a cycle $x'\in C_{p_1 + \ldots + p_n +1}(\mathcal{K}_{J_1 \cup \cdots \cup J_n})$ such that $[\omega']=[\omega]$, $[x']=[x]$ and $S_{\omega'}\cap S_{x'}$ contains only one simplex.
		Thus $\omega'(x')\neq 0$, and so $[\omega']=[\omega]$ is non-zero.
	\end{proof}
	


	
	By combining Propositions~\ref{prop: joins Massey defined} and \ref{prop: joins Massey non-trivial}, we have proved the main theorem.
	
	\begin{theorem}\label{thm: joins}
		For $i\in \{1, \ldots, n\}$, let $\mathcal{K}^i$ be a simplicial complex on $[m_i]$ that is not an $(m_i-1)$-simplex.
		Then there exists a simplicial complex $\mathcal{K}$, obtained by performing star deletions on $\mathcal{K}^1 * \cdots * \mathcal{K}^n$, with a non-trivial $n$-Massey product in $H^*(\mathcal{Z}_\mathcal{K})$. \qed
	\end{theorem}

	\begin{exmp}
		For $i=1,2,3$, let $\mathcal{K}^i$ be the simplicial complexes as in Example~\ref{ex: joins construction example} and let 
		\[
		\mathcal{K}=\sd_{\{5,8\}}\sd_{\{4,8\}}\sd_{\{3,8\}}\sd_{\{1,6\}}\sd_{\{1,5\}}\sd_{\{1,4\}} \mathcal{K}^1*\mathcal{K}^2*\mathcal{K}^3.
		\]
		Suppose that $a_1=\chi_1\in C^0(\mathcal{K}^1)$, $a_2=\chi_3+\chi_4+\chi_5\in C^0(\mathcal{K}^2)$, $a_3=\chi_7\in C^0(\mathcal{K}^3)$. 
		Then $S_{a_1}=\{1\}$, $S_{a_2}=\{\{3\}, \{4\}, \{5\}\}$, $S_{a_3}=\{\{7\}\}$ and $P_{a_2}=\{4,5,6\}$, $P_{a_3}=\{8\}$.
		The rest of the defining system constructed in \eqref{eq: a_ik joins} is
		\begin{align*}
			&a_{1,2}=\theta_{1,2} \chi_{1}=-\chi_1 \\
			&a_{2,3}=\theta_{2,3} (\chi_3 +\chi_4 +\chi_5)=-(\chi_3 +\chi_4 +\chi_5).
		\end{align*}
		The associated cocycle $\omega$ for this defining system is
		\[
		\omega=-\chi_1 (\chi_3 +\chi_4 +\chi_5) - \chi_{1} \chi_7.
		\]
		Therefore $\omega\in C^1(\mathcal{K})$ evaluates non-trivially on the $1$-cycle $x=\Delta_{\{1,3\}}-\Delta_{\{2,3\}}+\Delta_{\{2,8\}}-\Delta_{\{1,8\}}$.
		Another defining system could have $a_{2,3}'=\chi_8+\chi_6+\chi_7$.
		Then the associated cocycle $\omega'$ for this defining system is given by
		\[
		\omega'= \chi_1(\chi_6+\chi_7+\chi_8)  + - \chi_{1} \chi_7=\chi_{17}+\chi_{18}-\chi_{17}=\chi_{18}.
		\]
		Thus $\omega'$ also evaluates non-trivially on $x$. By Proposition~\ref{prop: joins Massey non-trivial}, the associated cocycle of any defining system evaluates non-trivially on some cycle.
		Hence $\langle [a_1], [a_2], [a_3] \rangle\subset H^{10}(\mathcal{Z}_\mathcal{K})$ is a non-trivial Massey product.
	\end{exmp}

	Two particular examples of Theorem~\ref{thm: joins} are the families of Baskakov and Limonchenko.
	
	\begin{exmp}[Baskakov's family \cite{Baskakov}] \label{ex: Recovering Baskakov's triple Massey}
		For $i=1,2,3$, let $\mathcal{K}^i$ be a triangulation of a $(n_i-1)$-sphere on $[m_i]$.
		Let $\sigma_1\in \mathcal{K}^1$, $\sigma_2, \sigma_2'\in \mathcal{K}^2$, $\sigma_3\in \mathcal{K}^3$ be maximal simplices such that $\sigma_2$ and $\sigma_2'$ are adjacent, that is, there is a vertex $v_{2'}\in \mathcal{K}^2$ such that $(\sigma_2\cap \sigma_2') \cup v_{2'} =\sigma_2'$. 
		Similarly, let $\sigma_3'\in \mathcal{K}^3$ be a maximal simplex adjacent to $\sigma_3$ so that there exists a vertex $v_{3'}\in \mathcal{K}^3$ such that $(\sigma_3\cap \sigma_3') \cup v_{3'} = \sigma_3'$.
		Let $a_1=\chi_{\sigma_1}$, $a_2=\chi_{\sigma_{2'}}$, and $a_3=\chi_{\sigma_{3'}}$ be cocycle representatives of $\alpha_i\in \widetilde{H}^{n_i-1}(\mathcal{K}^i)$ for $i=1,2,3$.
		Baskakov \cite{Baskakov} constructed $\mathcal{K}'=\stellar_{\{\sigma_1,\sigma_{2'}\}} \stellar_{\{\sigma_2, \sigma_3 \}} \mathcal{K}^1*\mathcal{K}^2*\mathcal{K}^3$ and showed that $\langle \alpha_1, \alpha_2, \alpha_3 \rangle$ is a non-trivial Massey product in $H^*(\mathcal{Z}_\mathcal{K})$ where $\mathcal{K}$ is the restriction of $\mathcal{K}'$ to the vertex set $[m_1]\cup [m_2]\cup [m_3]$.
		Since $\mathcal{K}=\sd_{\{\sigma_1,\sigma_{2'}\}} \sd_{\{\sigma_2, \sigma_3 \}} \mathcal{K}^1*\mathcal{K}^2*\mathcal{K}^3$,
		Theorem~\ref{thm: joins} recovers Baskakov's family of examples of non-trivial triple Massey products in $H^*(\mathcal{Z}_\mathcal{K})$. 
		The simplest example when $\mathcal{K}^1, \mathcal{K}^2, \mathcal{K}^3$ are $S^0$ is shown in Figure~\ref{fig: baskakov's first example} and its restriction to the original $6$ vertices is in Figure~\ref{fig: example of s} after swapping the labels $\sigma_3$, $\sigma_{3'}$
	\end{exmp}

	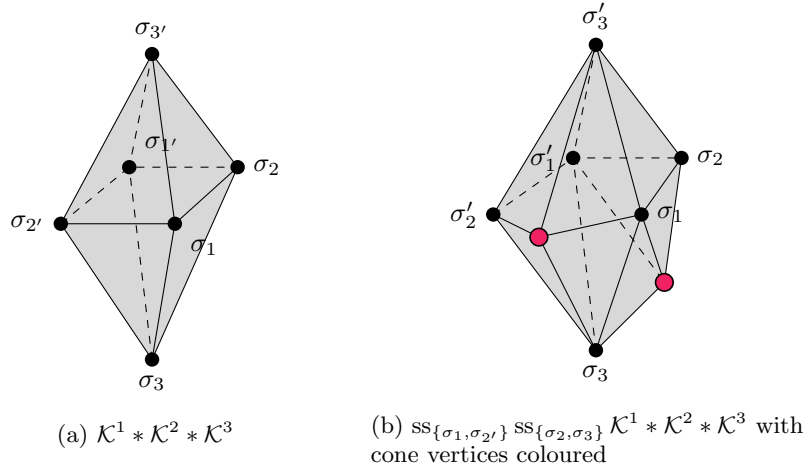
\begin{figure}[ht]
		\centering
		\begin{minipage}{0.45\textwidth}
			\centering
			\begin{tikzpicture} [scale=1.5, inner sep=2mm]
			\coordinate (4) at (0,0);
			\coordinate (1) at (1,0);
			\coordinate (2) at (0.6,0.5);
			\coordinate (3) at (1.55,0.5);
			\coordinate (6) at (0.8,1.5);
			\coordinate (5) at (0.8,-1.2);
			
			\fill[lightgray, fill opacity=0.6] (6)--(4)--(5)--(3)--(6);
			
			\draw (3) node[right] {$\sigma_{2}$} -- (1) node[below right] {$\sigma_1$} --  (4)--(6) node[above] {$\sigma_{3'}$}; 
			\draw (4)  node[left] {$\sigma_{2'}$} -- (5) node[below] {$\sigma_{3}$} -- (3); \draw (1) -- (5);
			\draw[] (1)--(6);
			\draw[] (3)--(6);
			\draw[dashed] (4)-- (2) node[above right] {$\sigma_{1'}$}-- (3);
			\draw[dashed] (5)--(2)--(6);
			
			\foreach \i in {1,...,6} {\fill (\i) circle (1.8pt);}; 
			\end{tikzpicture} 
			\subcaption{$\mathcal{K}^1*\mathcal{K}^2*\mathcal{K}^3$} 
			\label{fig: baskakov's octahedron}
		\end{minipage}
		\begin{minipage}{0.45\textwidth}
			\centering
			\begin{tikzpicture} [scale=1.5, inner sep=2mm]
			\coordinate (4) at (-0.1,0);
			\coordinate (1) at (1.2,0);
			\coordinate (2) at (0.6,0.5);
			\coordinate (3) at (1.55,0.5);
			\coordinate (5) at (0.8,1.5);
			\coordinate (6) at (0.8,-1.2);
			\coordinate (7) at (0.3,-0.2);
			\coordinate (8) at (1.4, -0.6);
			
			\fill[lightgray, fill opacity=0.6] (6)--(4)--(5)--(3)--(8)--(6);

			\draw (1)-- (3) -- (5) -- (4)  -- (6)  -- (1)--(5)--(7)--(6)--(8)--(3); 
			\draw (4)--(7)--(1)--(8);
			\draw[dashed] (4)-- (2) -- (3);
			\draw[dashed] (5)--(2)--(6);
			\draw[dashed] (2)--(8);
			
			\node[right] at (1) {$\sigma_1$};
			\node[left] at (2) {$\sigma_1'$};
			\node[right] at (3) {$\sigma_2$};
			\node[left] at (4) {$\sigma_2'$};
			\node[above] at (5) {$\sigma_3'$};
			\node[below] at (6) {$\sigma_3$};
			
			\foreach \i in {1,...,8} {\fill (\i) circle (1.8pt);}; 
			\foreach \i in {7,8} {\fill (\i) circle (2.4pt);}; 
			\foreach \i in {7,8} {\fill[colour3] (\i) circle (2pt);}; 
			\end{tikzpicture} 
			\subcaption{$\stellar_{\{\sigma_1,\sigma_{2'}\}} \stellar_{\{\sigma_2, \sigma_3 \}} \mathcal{K}^1*\mathcal{K}^2*\mathcal{K}^3$ with cone vertices coloured} \label{fig: baskakov's first example1}
		\end{minipage}
		\caption{The simplest example of both Baskakov and Limonchenko's families of non-trivial Massey products in moment-angle complexes}
		\label{fig: baskakov's first example}
	\end{figure}

	\begin{exmp}[Limonchenko's family \cite{Limonchenko}] \label{ex: Recovering Limonchenko's n-cubes}
		Let $F$ be a face of a polytope $P$ and suppose that there is a hyperplane $H$ that does not include any vertices of $P$ but separates the vertices of $F$ from the other vertices in $P$.
		If $H_1$, $H_2$ are the half spaces defined by $H$ and $F\subset H_2$, then the polytope $P\cap H_1$ is called a \textit{truncation} of $P$ at $F$.
		A family of non-trivial $n$-Massey products is constructed by truncating  the unit $n$-cube $I^n=I\times \cdots \times I$ as follows.
		Suppose that opposite facets of $I^n$ are labelled $F_l$, $F_{l'}$ for $l=1, \ldots, n$.  
		The boundary of the dual $\mathcal{K}=\mathcal{K}_{I^n}=\partial (I^n)^*$ is the join of $n$ copies of $S^0$, for example $\mathcal{K}_{I^3}$ is shown in Figure~\ref{fig: baskakov's octahedron}.
		To create a non-trivial $n$-Massey product, Limonchenko~\cite[Construction~1]{Limonchenko} truncated $I^n$ at the intersection of facets $F_i$ and $F_{k'}$ for $1\leqslant i< k\leqslant n$, $(i,k)\neq(1,n)$.
		For example see Figure~\ref{fig: truncated cube}.
		These truncations correspond to stellar subdividing $\mathcal{K}_{I^n}$ at the edges $\sigma_i\cup \sigma_{k'}$, where $\sigma_l, \sigma_{l'}\in\mathcal{K}_{I^n}$ are the vertices that are dual to the facets $F_l$, $F_{l'}$ in $I^n$. 
		Let $\mathcal{K}$ be the restriction of the stellar subdivided complex to the $2n$ vertices $\sigma_l, \sigma_{l'}$ for $l=1, \ldots, n$, and let $\alpha_l$ be the generator of $\widetilde{H}^0(\mathcal{K}_{\sigma_l, \sigma_{l'}})$.
		Limonchenko showed that the $n$-Massey product $\langle \alpha_1, \ldots, \alpha_n \rangle\subset H^*(\mathcal{Z}_\mathcal{K})$ is non-trivial.
		Since this construction is recovered by star deleting $K_{I^n}$ as described in Construction~\ref{cons: simplicial complex for n-Massey}, Theorem~\ref{thm: joins} gives an alternative proof that $\langle \alpha_1, \ldots, \alpha_n \rangle$ is non-trivial.
	\end{exmp}
	
	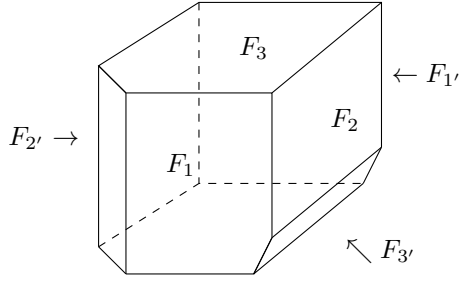
\begin{figure}
		\centering
		\begin{tikzpicture}[scale=1.2]
		\coordinate (a) at (0,0);
		\coordinate (b) at (2,0);
		\coordinate (c) at (0,2);
		\coordinate (d) at (2,2);
		\coordinate (e) at (1.2,1);
		\coordinate (f) at (3.2,1);
		\coordinate (g) at (1.2,3);
		\coordinate (h) at (3.2,3);
		
		\coordinate (a1) at (0.4,0);
		\coordinate (a2) at (0.1,0.3);
		\coordinate (c1) at (0.4,2);
		\coordinate (c2) at (0.1,2.3);
		
		\coordinate (b1) at (1.8,0);
		\coordinate (b2) at (2,0.4);
		\coordinate (f1) at (3,1);
		\coordinate (f2) at (3.2,1.4);
		
		\draw[dashed] (a2)--(e)--(f1);
		\draw[dashed] (g)--(e);
		\draw (f2)--(h)--(g)--(c2)--(c1)--(d)--(b2)--(b1)--(a1);
		\draw (h)--(d);
		\draw (a1)--(c1); \draw (a2)--(c2);
		\draw (b1)--(f1); \draw (b2)--(f2);
		\foreach \i in {a,b, c, f} {\draw (\i1)--(\i2);};

		\node at (1.0,1.2){$F_{1}$}; 
		\node at (2.8,1.7){$F_{{2}}$}; 
		\node at (1.8,2.5){$F_{3}$};
		\node at (1.95,1.2){}; 
		\node at (0.6,2.5){};
		\node at (-0.5,1.5){$F_{2'}\rightarrow$}; 
		\node at (3.7,2.2){$\leftarrow F_{{1'}}$};
		\node at (3.2,0.25){$\nwarrow F_{{{3'}}}$};
		\end{tikzpicture} 
		\caption{A $3$-cube truncated at the faces $F_{1}\cap F_{{2'}}$ and $F_{2}\cap F_{{3'}}$, which is dual to Figure~\ref{fig: baskakov's first example1} with the labels $\sigma_3, \sigma_{3'}$ swapped.}
		\label{fig: truncated cube}
	\end{figure}
	
	Theorem~\ref{thm: joins} does not just give alternative proofs of existing results about non-trivial Massey products in the cohomology of moment-angle complexes,
	it creates non-trivial $n$-Massey products from any non-zero cohomology classes supported on a full subcomplex of any simplicial complex $\mathcal{K}^i$. 
	Therefore there is no limit on $n$ or the dimension of the classes $\alpha_i$.
	Using this construction it is also possible to construct Massey products on torsion elements.

	\begin{exmp} \label{ex: joins torsion example}
		Let $\mathcal{K}^1$ be a triangulation of $\mathbb{R}P^2$ on $6$ vertices as in Figure~\ref{fig: RP2 triangulation}. 
		Let $\mathcal{K}^2$, $\mathcal{K}^3$ be copies of two disjoint vertices labelled $6, 7$ and $8, 9$, respectively.
		Let $\alpha_1\in \widetilde{H}^2(\mathcal{K}^1)$ be represented by $\chi_{012}$. For $i=2,3$, let $\alpha_i \in \widetilde{H}^0(\mathcal{K}^i)$ be represented by $a_2=\chi_6$ and $a_3=\chi_8$, respectively.
		By Construction~\ref{cons: simplicial complex for n-Massey}, $P_{a_2}=\{ \{7\} \}$ and $P_{a_3}=\{ \{9 \}\}$.
		Then let 
		\[	
		\mathcal{K} = \sd_{\{0127\}} \sd_{\{69\}} \mathcal{K}^1 * \mathcal{K}^2 * \mathcal{K}^3.
		\]
		By Theorem~\ref{thm: joins}, there is a non-trivial triple Massey product $\langle \alpha_1, \alpha_2, \alpha_3 \rangle \subset H^{14}(\mathcal{Z}_\mathcal{K})$. 
		This is the smallest example of a non-trivial triple Massey product on a torsion class since $\mathcal{K}^1$ is the triangulation of $\mathbb{R}P^2$ on the least number of vertices.
		
		Since $\alpha_1$ is the generator of $\widetilde{H}^2(\mathcal{K}^1)\cong \widetilde{H}^2(\mathbb{R}P^2)$, $\alpha_1$ is a torsion element. 
		The associated cocycle for the defining system constructed in \eqref{eq: a_ik joins}
		is $\omega=-\chi_{0126}-\chi_{0128}\in C^3(\mathcal{K})$.
		The corresponding class $[\omega]\in \langle \alpha_1, \alpha_2, \alpha_3 \rangle$ is not a torsion element in $H^{14}(\mathcal{Z}_\mathcal{K})$. 
		
		Also, there is a cochain $a_{1,2}'=\chi_{126}+\chi_{124}-\chi_{147}-\chi_{347}+\chi_{037}+\chi_{027}$ such that $d(a_{1,2}')=\chi_{0126}\in C^{3}(\mathcal{K}_{01234567})$, which is different to $a_{1,2}$ constructed in \eqref{eq: a_ik joins}.
		The associated cocycle to this defining system is $\omega'=-\chi_{0126}+\chi_{1268}+\chi_{1248}-\chi_{1478}-\chi_{3478}+\chi_{0378}+\chi_{0278}$ with $[\omega']\neq 0$ and $[\omega]\neq [\omega']$.
		Therefore $\langle \alpha_1, \alpha_2, \alpha_3 \rangle$ has non-trivial indeterminacy.
		In particular, the indeterminacy is given by $\alpha_1 \cdot \widetilde{H}^0(\mathcal{K}_{6789})+\alpha_3 \cdot \widetilde{H}^2(\mathcal{K}_{01234567})=\alpha_3 \cdot \widetilde{H}^2(\mathcal{K}_{01234567})$, where $\widetilde{H}^2(\mathcal{K}_{01234567})\cong~\mathbb{Z}$.
	\end{exmp}

	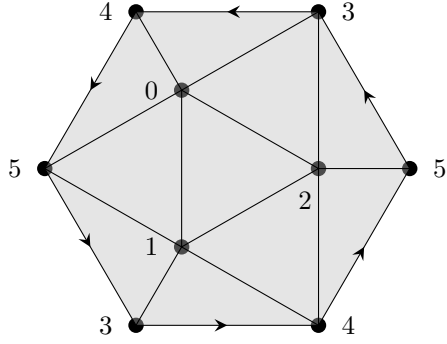
\begin{figure}[ht]
		\centering
		\begin{tikzpicture}	[scale=1.2, inner sep=3mm, decoration={
			markings,
			mark=at position 0.5 with {\arrow[scale=1.5]{stealth}}}
		]
		
		\foreach \i in {1, ..., 6}
		\fill (\i*360/6:2) coordinate (O\i) circle(2.4 pt);
		
		
		\foreach \i in {1,2,3}
		\fill (\i*360/3:1) coordinate (I\i) circle(2.4 pt);

		\fill[lightgray, fill opacity=0.4](O1) --(O2)--(O3) --(O4)--(O5)--(O6)--(O1);
		
		\draw (I1)--(I2)--(I3)--(I1);
		\draw[postaction={decorate}] (O1)--(O2);
		\draw[postaction={decorate}] (O2)--(O3);
		\draw[postaction={decorate}] (O3)--(O4);
		\draw[postaction={decorate}] (O4)--(O5);
		\draw[postaction={decorate}] (O5)--(O6);
		\draw[postaction={decorate}] (O6)--(O1);
		\draw (O3)--(I2)--(O5)--(I3)--(O1)--(I1)--(O3);
		\draw (O6)--(I3); \draw (I1)--(O2); \draw (I2)--(O4);
		
		\node[right] at (O1){$3$};
		\node[left] at (O2){$4$};
		\node[left] at (O3){$5$};
		\node[left] at (O4){$3$};
		\node[right] at (O5){$4$};
		\node[right] at (O6){$5$};
		\node[left] at (I1){$0$};
		\node[left] at (I2){$1$};
		\node[below] at (I3){$2\quad$};
		
		\end{tikzpicture}
		\caption{A $6$-vertex triangulation of $\mathbb{R}P^2$.}
		\label{fig: RP2 triangulation}
	\end{figure}
	
	We also extend Construction~\ref{cons: simplicial complex for n-Massey} by allowing more general star deletions in order to construct more non-trivial Massey products. These often only require a difference in the technical details of the proof of Theorem~\ref{thm: joins} and do not change the nature of the construction.
	For example, if $\mathcal{K}^i$ is the disjoint union of two vertices $\{i\}$ and $\{i'\}$, then let $\mathcal{K}$ be the simplicial complex that is obtained from $\mathcal{K}^1*\mathcal{K}^2*\mathcal{K}^3*\mathcal{K}^4$ by the sequence of star deletions
	\scriptsize
	\[
	     {\sd_{\{1,4\}}\sd_{\{1,4'\}}\sd_{\{1',4'\}}\sd_{\{2,4\}}\sd_{\{2',4'\}}\sd_{\{2,4'\}}\sd_{\{1,3\}}\sd_{\{1',3'\}}\sd_{\{1,3'\}}\sd_{\{3,4'\}}\sd_{\{2,3'\}}\sd_{\{1,2'\}}}.
	\] \normalsize
	This is a full subcomplex of the icosahedron $\mathcal{I}$ as shown in \cite[Theorem~4.6]{baralic-grbic-limonchenko-vucic}. Also, no obstruction graph from the classification \cite{DenhamSuciu, LowestDegreeClassification} is a full subcomplex of $\mathcal{I}$. In \cite{baralic-grbic-limonchenko-vucic}, this example is given in order to demonstrate a non-trivial $4$-Massey product of lowest-degree classes in $H^*(\mathcal{Z}_\mathcal{I})$ where there are no non-trivial $3$-Massey products of lowest-degree classes in $H^*(\mathcal{Z}_\mathcal{I})$. 
	
	Let us consider another example of more general star deletions. Suppose we have a simplicial complex  $\mathcal{K}$ with $n\geqslant 3$ disjoint subsets of its vertices $J_1, \ldots, J_n\subset V(\mathcal{K})$ such that there are $n$ non-trivial classes $\alpha_i\in \widetilde{H}^{p_i}(\mathcal{K}_{J_i})$. 
	Limonchenko \cite[Lemma~3.3]{Limonchenko_multiwedge} 
	showed that if a Massey product $\langle \alpha_1, \ldots, \alpha_n\rangle\subset H^{*}(\mathcal{Z}_\mathcal{K})$ is defined and $\widetilde{H}^{p_i+\cdots+p_{k}}(\mathcal{K}_{J_i \cup \cdots \cup J_{k}})=0$ for every $1\leqslant i < k \leqslant n$, $(i,k)\neq (1,n)$, then the Massey product has trivial indeterminacy. 
	The following example shows that this is not a necessary condition for trivial indeterminacy.
	\begin{exmp}
	    For $i=1,2,3,4$, let $\mathcal{K}^i$ be the disjoint union of two vertices $J_i=\{i, i'\}$. 
	    Let $\mathcal{K}$ be a simplicial complex obtained by Construction~\ref{cons: simplicial complex for n-Massey} with an additional star deletion at the edge $\{1',2'\}$, that is,
		\[
		\mathcal{K}=\sd_{\{2,4'\}}\sd_{\{1,3'\}}\sd_{\{3,4'\}}\sd_{\{2,3'\}}\sd_{\{1',2'\}}\sd_{\{1,2'\}} \mathcal{K}^1*\mathcal{K}^2*\mathcal{K}^3*\mathcal{K}^4.
		\]
		For each $1\leqslant i\leqslant 4$, let $a_i=\chi_i\in C^0(\mathcal{K}_{J_i})$  and set $\alpha_i=[a_i]\in \widetilde{H}^0(\mathcal{K}_{J_i})$.
		The star deletions at $\{1,2'\}$ and $\{1',2'\}$ imply that any cochain $a_{1,2}\in C^0(\mathcal{K}_{J_1\cup J_2})$ such that $d(a_{1,2})=\doubleoverline{a_1}a_2$ is of the form \[a_{1,2}=-\chi_1+c_1(\chi_1+\chi_{1'}+\chi_2)+c_1'\chi_{2'}\] for any $c_1, c_1'\in \mathbf{k}$. However, a cochain $a_{1,3}$ such that $d(a_{1,3})=\doubleoverline{a_1} a_{2,3}+\doubleoverline{a_{1,2}}a_3$ is only defined when $c_1'=c_1$. 
		Thus, any defining system for $\langle \alpha_1, \alpha_2, \alpha_3, \alpha_4 \rangle\subset H^*(\mathcal{Z}_\mathcal{K})$ is of the form
		\begin{align*}
			&a_{1,2}=-\chi_1+c_1(\chi_1+\chi_{1'}+\chi_2+\chi_{2'}) \\
			&a_{2,3}=-\chi_2+c_2(\chi_2+\chi_{2'}+\chi_3+\chi_{3'}) \\
			&a_{3,4}=-\chi_3+c_3(\chi_3+\chi_{3'}+\chi_4+\chi_{4'}) \\
			&a_{1,3}=-(c_2-1)\chi_1+c_1\chi_3 + c_4(\chi_1+\chi_{1'}+\chi_2+\chi_{2'}+\chi_3+\chi_{3'}) \\
			&a_{2,4}=-(c_3-1)\chi_2+c_2\chi_4 + c_5(\chi_2+\chi_{2'}+\chi_3+\chi_{3'}+\chi_4+\chi_{4'})
		\end{align*}
		for coefficients $c_1, \ldots, c_5\in \mathbf{k}$.
		These are the same defining systems we would get if we had not star deleted $\mathcal{K}^1*\mathcal{K}^2*\mathcal{K}^3*\mathcal{K}^4$ at the edge $\{1',2'\}$.
		The associated cocycle $\omega$ to any of these defining systems is 
		\begin{multline*}
		    \omega=-\chi_{14'}-d(\chi_1)+c_3d(\chi_1)-c_1d(\chi_3) +c_4d(\chi_4)-c_5 d(\chi_1)+ \\ +c_1c_3(-d(\chi_1)-d(\chi_{1'})-d(\chi_2)-d(\chi_{2'})).
		\end{multline*}
		Thus, $\langle \alpha_1, \alpha_2, \alpha_3, \alpha_4 \rangle=[\omega]=[-\chi_{14'}]$ and hence this Massey product is non-trivial and has no indeterminacy. However, the star deletions at $\{1,2'\}$ and $\{1',2'\}$ imply that $\widetilde{H}^{0}(\mathcal{K}_{J_1\cup J_2})=\mathbb{Z}\neq 0$. Therefore this is an example of a non-trivial Massey product with trivial indeterminacy that does not satisfy the conditions of \cite[Lemma~3.3]{Limonchenko_multiwedge}.
	\end{exmp}
	
	\subsection{Infinite families of Massey products with non-trivial indeterminacy}\label{sec: infinite families non-trivial indeterminacy}
	
	In the last example, we saw that doing Construction~\ref{cons: simplicial complex for n-Massey} followed by an extra star deletion at $\{1',2'\}$ produced more choices of cochains $a_{1,2}$ such that $d(a_{1,2})=\doubleoverline{a_1}a_2$.
	We extend this technique to create the first infinite families of moment-angle complexes with non-trivial Massey products that have non-trivial indeterminacy.
    These are the first known examples of non-trivial indeterminacy in $n$-Massey products in $H^*(\mathcal{Z}_\mathcal{K})$ for $n\geqslant 4$.
	
	The idea in Construction~\ref{cons: simplicial complex for n-Massey} was to create a non-trivial Massey product $\langle \alpha_1, \ldots, \alpha_n\rangle$ by defining two sets of simplices $S_{a_i}, P_{a_i}$ for each $1\leqslant i \leqslant n$ and star deleting the join of $n$ simplicial complexes at the simplices $\sigma_i\cup \sigma_k'$ for $\sigma_i\in S_{a_i}$, $\sigma_k'\in P_{a_k}$, $1\leqslant i <k \leqslant n$, $(i,k)\neq (1,n)$.
	A star deletion at $\sigma_i\cup \sigma_k'$ made the Massey product $\langle \alpha_i, \ldots, \alpha_k \rangle$ trivial by allowing us to define a cochain $a_{i,k}$ such that  $d(a_{i,k})$ represents a (trivial) class in the lower Massey product $\langle \alpha_i, \ldots, \alpha_k \rangle$. Supposing that $\langle \alpha_1, \ldots, \alpha_k \rangle$ has non-trivial indeterminacy, we construct indeterminacy in the higher Massey product $\langle \alpha_1, \ldots, \alpha_n \rangle$ by making more than one class in the lower product $\langle \alpha_1, \ldots, \alpha_k \rangle$ trivial. 
	In this version of the construction, we star delete at $\sigma_1\cup \sigma_k'$ for $\sigma_k'\in P_{a_k}$, $k\neq n$, and any $p_1$-simplex $\sigma_1\in \mathcal{K}^1$, rather than $\sigma_1\in S_{a_1}\subset \mathcal{K}^1$.
	These extra star deletions create choices for $a_{1,k}$ in the defining system for $\langle \alpha_1, \ldots, \alpha_n\rangle$, and do not affect the proof of Theorem~\ref{thm: joins}.
	We will show that these choices result in non-trivial indeterminacy in $\langle \alpha_1, \ldots, \alpha_n\rangle$ when $n>2$. 
	
	\begin{theorem}\label{thm: infinite Massey products with indeterminacy}
		Let $\mathcal{K}^i$ be a simplicial complex on the vertex set $[m_i]$ that is not an $(m_i-1)$-simplex, for $i\in\{1, \ldots, n\}$, $n>2$.
		Then there exists a simplicial complex $\mathcal{K}$ obtained by star deletions on $\mathcal{K}^1* \cdots * \mathcal{K}^n$ such that $H^*(\mathcal{Z}_\mathcal{K})$ has a non-trivial $n$-Massey product with non-trivial indeterminacy.
	\end{theorem}
	
	\begin{proof}
		Since $\mathcal{K}^i$ is not an $(m_i-1)$-simplex, there is a non-trivial class $\alpha_i\in \widetilde{H}^{p_i}(\mathcal{K}_{J_i}^i)$ for $J_i\subset [m_i]$.
		We will construct two different defining systems for a Massey product $\langle \alpha_1, \ldots, \alpha_n \rangle$ and show that the two associated cocycles are non-zero and not cohomologous. Therefore this concludes there is non-trivial indeterminacy in $\langle \alpha_1, \ldots, \alpha_n \rangle$.
		
		Let $a_i$ be a cocycle representative for $\alpha_i$.
		Recall that in Construction~\ref{cons: simplicial complex for n-Massey}, we had a set of $p_i$-simplices $S_{a_i}\subset \mathcal{K}^i$ for each $i$ such that $a_i=\sum_{\sigma_i\in S_{a_i}} c_{\sigma_i} \chi_{\sigma_i}$.
		For any $\sigma_i\in S_{a_i}$, the set
		$P_{\sigma_i}\subset \mathcal{K}^i$ contains all $p_i$-simplices $\sigma_i'\in \mathcal{K}^i$ such that there is a vertex $v_{\sigma_i'}$ and $\sigma_i\setminus v_{\sigma_i}= \sigma_i'\setminus v_{\sigma_i'}$, where $v_{\sigma_i}$ is a fixed choice of vertex in $\sigma_i$. 
		We will use these fixed choices of $v_{\sigma_i}\in\sigma_i \in S_{a_i}$ and $v_{\sigma_i'}\in \sigma_i'\in P_{\sigma_i}$ throughout this proof.
		Also recall the set
		\[P_{a_i}=P_{\sigma_i^{(1)}}\cup \cdots \cup P_{\sigma_i^{(l)}}\] for  $\sigma_i^{(1)}, \ldots, \sigma_i^{(l)} \subset S_{a_i}$.
		To define a simplicial complex $\mathcal{K}$ so that $\langle \alpha_1, \ldots,\alpha_n \rangle\subset H^*(\mathcal{Z}_\mathcal{K})$ has non-trivial indeterminacy, we star delete $\mathcal{K}^1* \cdots * \mathcal{K}^n$ at  $\sigma_1\cup \sigma_k'$ for every $p_1$-simplex $\sigma_1\in \mathcal{K}^1$ and  $\sigma_k'\in P_{a_k}$, $1<k<n$, as well as at each $\sigma_i \cup \sigma_k'$ for $\sigma_i\in S_{a_i}$ and $\sigma_k'\in P_{a_k}$, $1< i < k \leq n$.
		This is more star deletions than in Construction~\ref{cons: simplicial complex for n-Massey}, where we used $\sigma_1\in S_{a_1}$ instead of $\sigma_1\in \mathcal{K}^1$.
		Let $\widetilde{S}_{a_k}=S_{a_k}\setminus P_{a_k}$.
		If there are simplices $\sigma_k\in S_{a_k}\setminus \widetilde{S}_{a_k}$ for any $k$, then we also star delete at $\sigma_i'\cup \sigma_k$ for every $\sigma_i'\in P_{a_i}$, $i<k$. This is for technical purposes, to ensure that $\sigma_k\in S_{a_k}$ and $\sigma_i'\cup \sigma_k\in \mathcal{K}$ implies that $\sigma_k\in \widetilde{S}_{a_k}$.
		
		We construct two different defining systems for $\langle \alpha_1, \ldots,\alpha_n \rangle$.
		Recall from \eqref{eq: a_ik joins} in Proposition~\ref{prop: joins Massey defined} that $a_{i,k}\in C^{p_i + \cdots + p_k}(\mathcal{K}_{J_i \cup \cdots \cup J_k})$ for $1\leqslant i \leqslant k\leqslant n$, $(i,k)\neq (1,n)$ is the cochain 
		\[
		a_{i,k}=  \sum_{\sigma_i \in S_{a_i}} 
		\sum_{\sigma_{i+1} \in \widetilde{S}_{a_{i+1}}} \cdots \sum_{\sigma_k \in \widetilde{S}_{a_k}}
		c_{\sigma_i}  \ldots c_{\sigma_k} \; \theta_{i,k} \;
		\chi_{\sigma_i \cup \cdots \cup \sigma_k \setminus (v_{i+1} \cup \cdots \cup v_{k})}
		\]
		where $\widetilde{S}_{a_{i}}=S_{a_i}\setminus P_{a_i}$ and $\theta_{i,k}=1$ when $i=k$ or otherwise
		\begin{multline} 
			\theta_{i,k}=
			(-1)^{k-i +|J_i|(p_{i+1}+\cdots +p_k)+|J_{i+1}|(p_{i+2}+\cdots +p_k)+
				\cdots + |J_{k-1}|p_k} \\
			\cdot \varepsilon(v_{\sigma_{i+1}},\sigma_{i+1}) \ldots \varepsilon(v_{\sigma_k}, \sigma_k).
		\end{multline}
		The defining system $(a_{i,k})$ is a defining system for $\langle \alpha_1, \ldots, \alpha_n\rangle$ by the same proof as for Proposition~\ref{prop: joins Massey defined}, since neither the simplices 
		$\sigma_1\cup \sigma_k$ for $\sigma_1\notin S_{a_1}$ and  $\sigma_k\in P_{a_k}$, $1<k<n$, nor $\sigma_i'\cup \sigma_k$ for $\sigma_i'\in P_{a_i}$ and $\sigma_k\in (S_{a_k}\setminus \widetilde{S}_{a_k})$, $i<k$,
		play an active role in the proof.
		To construct a different defining system, for any $1<k\leqslant n$, let 
		$b_{1,k}\in C^{p_1 + \cdots + p_k}(\mathcal{K}_{J_1 \cup \cdots \cup J_k})$ be the cochain
		\begin{equation*}
			b_{1,k}=
			\sum_{\sigma_1 \in S_{a_1}} 
			\sum_{\sigma_{2} \in \widetilde{S}_{a_{2}}} \cdots \sum_{\sigma_k \in \widetilde{S}_{a_k}}
			\sum_{\sigma_i'\in P_{\sigma_2}\cup \cdots \cup P_{\sigma_k}}
			 \varrho_{1,k} \;
			\chi_{v_{\sigma_i'} \cup \sigma_1 \cup \cdots \cup \sigma_k \setminus (v_{\sigma_1} \cup \cdots \cup v_{\sigma_k})}
		\end{equation*}
		where $\varrho_{1,k}=c_{\sigma_1}  \ldots c_{\sigma_k}\varepsilon(v_{\sigma_i'},v_{\sigma_i'}\cup \sigma_1\cup \cdots \cup \sigma_k \setminus (v_{\sigma_1} \cup \cdots \cup v_{\sigma_k})) \theta_{1,k}$.
		Also let $b_{i,k}=0$ for $i\neq 1$ or $i=k=1$, so $a_{i,k}'=a_{i,k}+b_{i,k}$ for all $1 \leqslant i\leqslant  k\leqslant n$, $(i,k)\neq (1,n)$.
		We will show that $(a_{i,k}')$ is a defining system for $\langle \alpha_1, \ldots, \alpha_n\rangle$.
		
		First we check that $d(b_{1,k})=\sum_{r=1}^{k-1} \doubleoverline{b_{1,r}}a_{r+1,k}$, where
		\begin{multline*}
		d(b_{1,k})=\sum_{\sigma_1 \in S_{a_1}} 
		\sum_{\sigma_{2} \in \widetilde{S}_{a_{2}}} \cdots \sum_{\sigma_k \in \widetilde{S}_{a_k}} 
		\sum_{\sigma_i'\in P_{\sigma_2}\cup \cdots \cup P_{\sigma_k}}\;
		\sum_{j\in \mathcal{K}_{J_1\cup \cdots \cup J_n}}
		 \varrho_{1,k} \cdot\\
		\cdot \varepsilon(j, j\cup v_{\sigma_i'}  \cup \sigma_1 \cup \cdots \cup \sigma_k \setminus (v_{\sigma_1} \cup \cdots \cup v_{\sigma_k})) \;
		\chi_{j\cup v_{\sigma_i'}  \cup \sigma_1 \cup \cdots \cup \sigma_k \setminus (v_{\sigma_1} \cup \cdots \cup v_{\sigma_k})}.
		\end{multline*}
		Fix a simplex $\tau=v_{\sigma_i'} \cup \sigma_1 \cup \cdots \cup \sigma_k \setminus (v_{\sigma_1} \cup \cdots \cup v_{\sigma_k})\in S_{b_{1,k}}$. 
		For any $1\leqslant r \leqslant k$, recall from the definition of $P_{\sigma_r}$ that since $\sigma_r\in S_{a_r}$, if there is a vertex $v\in \mathcal{K}^r$ such that $v\cup (\sigma_r\setminus v_{\sigma_r})\in \mathcal{K}^r$, then $v\cup (\sigma_r\setminus v_{\sigma_r})\in P_{\sigma_r}$.
		Thus $\sigma_i'=v_{\sigma_i'} \cup (\sigma_i\setminus v_{\sigma_i})\in P_{\sigma_i}$.
		Consider the link of $\tau$ in $\mathcal{K}_{J_1 \cup \cdots \cup J_k}$.
		There is no vertex $v\in \mathcal{K}^1$ in this link since if $v\cup \sigma_1\setminus v_{\sigma_1}\in \mathcal{K}^1$, then $(v\cup \sigma_1\setminus v_{\sigma_1}) \cup \sigma_i'\notin \mathcal{K}$ because there was a star deletion at that simplex.
		Similarly, for any $r<i$, there is no vertex $v_{\sigma_r}$ in the link of $\tau$ because $\sigma_r\cup \sigma_i'\notin \mathcal{K}$.
		Therefore the only vertices in the link of $\tau$ are $v_{\sigma_r'}$ for $\sigma_{r'}\in P_{\sigma_r}$ and any $r$, and $v_{\sigma_r}$ for $\sigma_r\in S_{a_r}$ and $r>i$.
		
		Consider the summands of $d(b_{1,k})$ when $j=v_{\sigma_r'}$ for $\sigma_{r'}\in P_{\sigma_r}$ and any $r$. 
		If $v_{\sigma_i'} \cup v_{\sigma_{r}'} \cup \sigma_1 \cup \cdots \cup \sigma_k \setminus (v_{\sigma_1} \cup \cdots \cup v_{\sigma_k})\in \mathcal{K}$, then the coefficient of $\chi_{v_{\sigma_i'} \cup v_{\sigma_{r}'} \cup \sigma_1 \cup \cdots \cup \sigma_k \setminus (v_{\sigma_1} \cup \cdots \cup v_{\sigma_k})}$ is the product of $c_{\sigma_1}  \ldots c_{\sigma_k} \; \theta_{1,k}$ and 
		\small{\begin{multline}\label{eq: sign for joins indeterminacy}
			\varepsilon(v_{\sigma_i'},v_{\sigma_i'}\cup \sigma_1\cup \cdots \cup \sigma_k \setminus (v_{\sigma_1} \cup \cdots \cup v_{\sigma_k}))
			\varepsilon(v_{\sigma_r'}, v_{\sigma_i'}  \cup v_{\sigma_r'}\cup \sigma_1 \cup \cdots \cup \sigma_k \setminus (v_{\sigma_1} \cup \cdots \cup v_{\sigma_k})) 
			+\\
			\varepsilon(v_{\sigma_r'},v_{\sigma_r'}\cup \sigma_1\cup \cdots \cup \sigma_k \setminus (v_{\sigma_1} \cup \cdots \cup v_{\sigma_k}))
			\varepsilon(v_{\sigma_i'}, v_{\sigma_i'}  \cup v_{\sigma_r'}\cup \sigma_1 \cup \cdots \cup \sigma_k \setminus (v_{\sigma_1} \cup \cdots \cup v_{\sigma_k})).
		\end{multline}}
		First suppose that $\sigma_i', \sigma_r'\in \mathcal{K}^i$, so $\sigma_i'\cup v_{\sigma_r'}=\sigma_r'\cup v_{\sigma_i'}\in \mathcal{K}^i$.
		Also suppose, without loss of generality, that $v_{\sigma_i'}<v_{\sigma_r'}$ in the order of the vertex set of $\mathcal{K}$ and that $v_{\sigma_r'}$ is  the $l$th vertex in   $\sigma_r'$.
		Since $v_{\sigma_i'}<v_{\sigma_r'}$ and $\varepsilon(v_{\sigma_i'},v_{\sigma_i'}\cup \sigma_1\cup \cdots \cup \sigma_k \setminus (v_{\sigma_1} \cup \cdots \cup v_{\sigma_k}))=(-1)^{p_1+\cdots+p_{i-1}}\varepsilon(v_{\sigma_i'},\sigma_i')$ by the definition of $\varepsilon$ in \eqref{eq: varepsilon}, we rewrite \eqref{eq: sign for joins indeterminacy}  as
		\begin{multline*}
			\varepsilon(v_{\sigma_i'},\sigma_i') \varepsilon(v_{\sigma_r'},\sigma_i'\cup v_{\sigma_r'})
			+
			\varepsilon(v_{\sigma_r'},\sigma_r')
			\varepsilon(v_{\sigma_i'},\sigma_i'\cup v_{\sigma_r'})
			=\\ (-1)^l\varepsilon(v_{\sigma_i'},\sigma_i')-(-1)^{l-1}\varepsilon(v_{\sigma_i'},\sigma_i'\cup v_{\sigma_r'}).
		\end{multline*}
		Also  $\varepsilon(v_{\sigma_i'},\sigma_i')=\varepsilon(v_{\sigma_i'},\sigma_i'\cup v_{\sigma_r'})$ because $v_{\sigma_i'}<v_{\sigma_r'}$,  so \eqref{eq: sign for joins indeterminacy} is zero.
		In particular, when $k=i=2$, then $d(b_{1,2})=0$. So $d(a_{1,2}')=d(a_{1,2})=\doubleoverline{a_1}a_2$.
		
		Alternatively, suppose that $\sigma_i'\in \mathcal{K}^i$, $\sigma_r'\in \mathcal{K}^r$ and, without loss of generality, $i<r$. By using the definition of $\varepsilon$ in \eqref{eq: varepsilon}, then \eqref{eq: sign for joins indeterminacy} becomes
		\begin{multline*}
			(-1)^{p_1+\cdots+p_{i-1}} \varepsilon(v_{\sigma_i'},\sigma_i')(-1)^{p_1+\cdots+p_{r-1}+1} \varepsilon(v_{\sigma_r'},\sigma_r')
			+\\
			(-1)^{p_1+\cdots+p_{r-1}} \varepsilon(v_{\sigma_r'},\sigma_r')(-1)^{p_1+\cdots+p_{i-1}} \varepsilon(v_{\sigma_i'},\sigma_i')=0.
		\end{multline*}
		Since all of these summands cancel out, we conclude that $d(b_{1,k})$ only has non-zero summands when $j=v_{\sigma_r}$ for $\sigma_r\in S_{a_r}$ and $r>i$.  By rewriting $r$ as $r+1$ for $r\in \{1, \ldots, n\}$ and $i\in \{2, \ldots, r\}$,  $d(b_{1,k})$ is equal to
		\begin{multline}\label{eq: joins indeterminacy d(b1k)}
			\sum_{\sigma_1 \in S_{a_1}} 
			\sum_{\sigma_{2} \in \widetilde{S}_{a_{2}}} \cdots \sum_{\sigma_k \in \widetilde{S}_{a_k}} \sum_{r=1}^{k-1}
			\sum_{\sigma_i'\in P_{\sigma_2}\cup \cdots \cup P_{\sigma_r}}
			 \varrho_{i,k} \;\cdot  \\
			\cdot \varepsilon(v_{\sigma_{r+1}}, v_{\sigma_i'}  \cup v_{\sigma_{r+1}}\cup \sigma_1 \cup \cdots \cup \sigma_k \setminus (v_{\sigma_1} \cup \cdots \cup v_{\sigma_k})) \;
			\chi_{v_{\sigma_i'} \cup v_{\sigma_{r+1}} \cup \sigma_1 \cup \cdots \cup \sigma_k \setminus (v_{\sigma_1} \cup \cdots \cup v_{\sigma_k})}.
		\end{multline}
		Since the simplices $\sigma_i'\cup \sigma_r$ were star deleted for $\sigma_i'\in P_{a_i}$ and $\sigma_r\in S_{a_r}\setminus \widetilde{S}_{a_r}$, this sum is the same whether we use $\sum_{\sigma_r \in \widetilde{S}_{a_r}}$ or $\sum_{\sigma_r \in S_{a_r}}$. 	
		Therefore we split this sum into products 
		so that $d(b_{1,k})=\sum_{r=1}^{k-1} \doubleoverline{b_{1,r}}a_{r+1,k}$, by using the fact that  
		\begin{multline*}
		\varrho_{1,k}\varepsilon(v_{\sigma_{r+1}}, v_{\sigma_i'}  \cup v_{\sigma_{r+1}}\cup \sigma_1 \cup \cdots \cup \sigma_k \setminus (v_{\sigma_1} \cup \cdots \cup v_{\sigma_k})) 
		=\\
		(-1)^{p_1+\cdots+p_r +|J_1\cup \cdots \cup J_r|(p_{r+1}+\cdots+p_k)} \varrho_{1,r} \,
		c_{\sigma_{r+1}}\ldots c_{\sigma_{k}} \theta_{{r+1},k}.
		\end{multline*}
		Then  
		\begin{align*}
        d(a_{i,k}')&=d(a_{i,k})+d(b_{i,k})
        \\ &=
		\sum_{r=1}^{k-1} \doubleoverline{a_{1,r}}a_{r+1,k}+
		\sum_{r=1}^{k-1} \doubleoverline{b_{1,r}}a_{r+1,k}
        = \sum_{r=1}^{k-1} \doubleoverline{a_{1,r}'}a_{r+1,k}.
		\end{align*}
		Hence $(a_{i,k}')$ is a defining system for $\langle \alpha_1, \ldots, \alpha_n \rangle$.

		The associated cocycle for $(a_{i,k}')$ is \begin{equation*}
		    \omega'=\omega+
			\sum_{r=1}^{n-1} \doubleoverline{b_{1,r}}a_{r+1,n}
		\end{equation*}
		where $\omega$ is the associated cocycle for $(a_{i,k})$.
		We show that the difference $\omega'-\omega$ is not a coboundary by constructing a cycle $x'\in C_{p_1 + \ldots + p_n +1}(\mathcal{K}_{J_1 \cup \cdots \cup J_n})$ such that $(\omega'-\omega)(x)\neq 0$.
		We use a similar method to Construction~\ref{cons: joins x}. 
		Fix $\sigma_i\in \widetilde{S}_{a_i}$ for $i=1 ,3, \ldots, n-1$ and fix $\sigma_2'\in P_{a_2}$. 
		Since it was star deleted,  $\sigma_1\cup \sigma_2'\notin \mathcal{K}$ but the boundary complex $\partial(\sigma_1\cup \sigma_2')$ is contained in $\mathcal{K}$. 
		Also, since $\alpha_n\in \widetilde{H}^{p_n}(\mathcal{K}_{J_n})$ is non-zero, there is a cycle $x_n\in C_{p_n}(\mathcal{K}_{J_n})$ such that $a_n(x_n)\neq 0$. 
		Define $x'\in C_{p_1 + \ldots + p_n +1}(\mathcal{K}_{J_1 \cup \cdots \cup J_n})$ to be the chain 
		\begin{multline*}
		      x'=\sum_{w_2\in \sigma_1\cup \sigma_2'} 
		      \sum_{w_3\in \sigma_3} \cdots \sum_{w_{n-1}\in \sigma_{n-1}} \\
		      \sum_{\tilde{\sigma}_n\in S_{x_n}}
			 c_{w_2} \cdots c_{w_{n-1}}  c_{\tilde{\sigma}_n} 
			\Delta_{\sigma_1 \cup \sigma_2' \cup \sigma_3 \cdots \cup \sigma_{n-1} \cup \tilde{\sigma}_n\setminus (w_2 \cup \cdots \cup w_{n-1})}
		\end{multline*}
		where $c_{\tilde{\sigma}_n}$ are the non-zero coefficients from $x_n$, and  $c_{w_2}, \ldots, c_{w_{n-1}}$ are the coefficients of cycles in $C_{p_1+p_2}(\partial(\sigma_1\cup \sigma_2'))$, $C_{p_i}(\partial(\sigma_i))$ for $3\leqslant i \leqslant n-1$.
		Every simplex $\sigma$ in the support $S_x$ of $x$ is a simplex in $\mathcal{K}$ since none of them were star deleted.
		By an analogous proof to Lemma~\ref{lem: joins cycle x}, the chain $x'$ is a cycle.
		
		We want to compare the supports $S_{\omega'-\omega}$ of $\omega'-\omega$ and $S_{x'}$ of $x'$.
		The cochain $\omega'-\omega=\sum_{r=1}^{n-1} \doubleoverline{b_{1,r}}a_{r+1,n}$ is given in \eqref{eq: joins indeterminacy d(b1k)} when $k=n$.
		A simplex $\sigma$ is in $S_{\omega'-\omega}\cap S_{x'}$ precisely when $w_2=v_{\sigma_1}$, $w_j=v_{\sigma_j}$ for $3\leqslant j \leqslant n-1$,
		$r=n-1$ so that $\sigma_{r+1}=\sigma_n\in S_{x_n}$, and $i=2\in \{2, \ldots, r\}$ so that  $\sigma_i'=\sigma_2'\in P_{a_2}$.
		Hence $S_{\omega'-\omega}\cap S_{x'}$ contains only one simplex, $\sigma$.  
		Thus $(\omega'-\omega)(x)=\pm \chi_\sigma (\Delta_\sigma)\neq 0$.
		Therefore $\omega'-\omega$ is not a coboundary and so $[\omega']\neq[\omega]$.

		The proof that $\langle \alpha_1, \ldots, \alpha_n \rangle$ is non-trivial is the same as the proof of Proposition~\ref{prop: joins Massey non-trivial} since neither the simplices  $\sigma_1\cup \sigma_k'$ for $\sigma_1\notin S_{a_1}$ and $\sigma_k'\in P_{a_k}$, nor $\sigma_i'\cup \sigma_k$ for $\sigma_i\in P_{a_i}$ and $\sigma_k\in (S_{a_k}\setminus \widetilde{S}_{a_k})$, $i<k$, 
		play a role so the extra star deletions do not change the proof.
		Hence  $\langle \alpha_1, \ldots, \alpha_n \rangle$ is non-trivial with non-trivial indeterminacy.
	\end{proof}
	
	\begin{exmp}
	    For $i=1,2,3$, suppose $\mathcal{K}^i$ is a pair of disjoint vertices labelled $\sigma_i, \sigma_i'$.
	    Let $\alpha_i\in \widetilde{H}^0(\mathcal{K}^i)$ be represented by the cocycle $a_i=\chi_{\sigma_i}$.
	    Then $S_{a_i}=\{\sigma_i\}$ and $P_{a_i}=\{\sigma_i'\}$.
	    Following the construction in the proof of Theorem~\ref{thm: infinite Massey products with indeterminacy}, we define 
	    \[ \mathcal{K}= \sd_{\sigma_1 \cup \sigma_2'} \sd_{\sigma_1' \cup \sigma_2'} \sd_{\sigma_2 \cup \sigma_3'} \mathcal{K}^1 * \mathcal{K}^2 * \mathcal{K}^3. \]
	    This simplicial complex is shown in Figure~\ref{fig: smallest indeterminacy example}. The Massey product $\langle \alpha_1, \alpha_2, \alpha_3\rangle$ is one of the simplest examples of a Massey product in a moment-angle complex with non-trivial indeterminacy. It is one of the obstruction graphs in the classification of lowest degree non-trivial triple Massey products in \cite{LowestDegreeClassification}. Since it is a triple Massey product, its indeterminacy is given by $\alpha_1 \cdot \widetilde{H}^{0}(\mathcal{K}_{\sigma_2, \sigma_2', \sigma_3, \sigma_3'})+ \alpha_3 \cdot \widetilde{H}^{0}(\mathcal{K}_{\sigma_1, \sigma_1', \sigma_2, \sigma_2'})=\alpha_3 \cdot\widetilde{H}^{0}(\mathcal{K}_{\sigma_1, \sigma_1', \sigma_2, \sigma_2'})$.
	\end{exmp}
	
	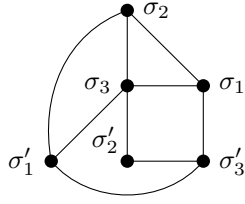
\begin{figure}[ht]
	    \centering
	    \begin{tikzpicture}[scale=1, inner sep=2mm]
		\coordinate (3) at (0,0);
		\coordinate (5) at (1,0);
		\coordinate (6) at (0,1);
		\coordinate (2) at (1,1);
		\coordinate (4) at (0,2);
		\coordinate (1) at (-1,0);
		\coordinate (a) at (0, -1.0);
		
		\draw (2)  -- (6)--(3) -- (5) -- (2) --(4) ;
		\draw (1)  to [out=100,in=-160] (4);
		\draw (1) to [out=-50,in=-130] (5);
		\draw (1) -- (6);
		\draw[] (4)--(6);
		
		\foreach \i in {1, ..., 6} {\fill (\i) circle (2.5pt);}
		
		\draw (1) node[left] {$\sigma_1'$};
		\draw (2) node[right] {$\sigma_1$};
		\draw (3) node[above left, inner sep=1mm] {$\sigma_2'$};
		\draw (4) node[right] {$\sigma_2$};
		\draw (5) node[right] {$\sigma_3'$};
		\draw (6) node[left] {$\sigma_3$};
		\end{tikzpicture} 
	    \caption{A simplicial complex $\mathcal{K}$ such that a triple Massey product in $H^*(\mathcal{Z}_\mathcal{K})$ has indeterminacy.}
	    \label{fig: smallest indeterminacy example}
	\end{figure}
	

	
	\section{Massey products constructed by edge contractions}
	
	A simplicial homotopy map $\varphi \colon \mathcal{K}\to \mathcal{\hat{K}}$ induces a map on the cohomology of moment-angle complexes $\varphi^*\colon H^*(\mathcal{Z}_{\mathcal{\hat{K}}}) \to  H^*(\mathcal{Z}_\mathcal{K})$.
	However, a property of Massey products \cite[Section 2]{Kraines} is that $\varphi^*\langle \hat{\alpha}_1,\ldots, \hat{\alpha}_n\rangle \subset \langle \varphi^* (\hat{\alpha}_1),\ldots, \varphi^*(\hat{\alpha}_n)\rangle$.
	Hence if $\langle \varphi^* (\hat{\alpha}_1),\ldots, \varphi^*(\hat{\alpha}_n)\rangle$ has non-trivial indeterminacy, it may be trivial even if $\langle \hat{\alpha}_1,\ldots, \hat{\alpha}_n\rangle$ is non-trivial.

	In this section we  use edge contractions $\varphi \colon \mathcal{K}\to \mathcal{\hat{K}}$ as a simplicial homotopy operation to construct non-trivial Massey products. 
	Given a non-trivial Massey product $\langle \hat{\alpha}_1,\ldots, \hat{\alpha}_n\rangle \subset H^*(\mathcal{Z}_{\mathcal{\hat{K}}})$ with $\hat{\alpha}_i\in \widetilde{H}^{p_i}(\mathcal{\hat{K}}_{\hat{J}_i})$, $\hat{J}_i\neq \hat{J}_j$ for $i\neq j$, 
	we explicitly construct a defining system to show that $\langle \alpha_1, \ldots, \alpha_n\rangle \subset H^*(\mathcal{Z}_\mathcal{K})$ 
	is defined
	where $\alpha_i$ is the pullback of $\hat{\alpha}_i$ along $\varphi$.
	Then we also show that it is non-trivial to conclude the main result of this section, Theorem~\ref{thm: edge contractions Massey}.


	\begin{defn}
		Let $\mathcal{K}$, $\mathcal{\hat{K}}$ be simplicial complexes with an edge $\{u,w\}\in \mathcal{K}$, and a vertex $z \in V(\mathcal{\hat{K}})$ such that $V(\mathcal{\hat{K}})\setminus \{z\}=V(\mathcal{K})\setminus \{\{u\},\{w\}\}$. 
		The simplicial complex $\mathcal{\hat{K}}$ is obtained from $\mathcal{K}$ by \textit{an edge contraction} of $\{u,w\}$ if there is a map $\varphi_V\colon V(\mathcal{K}) \to V(\mathcal{\hat{K}})$
		\begin{equation*}
			\varphi_V(v)= \begin{cases}
				z & \text{for } v\in \{u,w\}  \\
				v & \text{for } v\notin \{u,w\}
			\end{cases}
		\end{equation*}
		that extends to a surjective map $\varphi\colon \mathcal{K} \to \mathcal{\hat{K}}$, where  $\varphi(I)=\{\varphi_V(v_1), \ldots, \varphi_V(v_n)\}$ for $I=\{v_1, \ldots, v_n \}\in \mathcal{K}$. 
		The map $\varphi\colon \mathcal{K} \to \mathcal{\hat{K}}$ is called the \textit{edge contraction of} $\{u,w\}\in \mathcal{K}$.
	\end{defn}
	
	Edge contractions are simplicial maps, but they do not preserve the topology of $\mathcal{K}$ in general. Attali, Lieutier and Salinas \cite[Theorem 2]{ALS} showed that the homotopy type of a simplicial complex is preserved under edge contractions that satisfy the \textit{link condition}.
	
	\begin{theorem}[\cite{ALS}] \label{thm:ALS}
		For any simplicial complex $\mathcal{K}$, if an edge $\{u,w\}\in \mathcal{K}$ satisfies the \textit{link condition}, 
		\begin{equation}\label{eq:linkcondition}
			\link_\mathcal{K}(\{u\}) \cap \link_\mathcal{K}(\{w\})=\link_\mathcal{K} (\{u,w\}),
		\end{equation}
		then the edge contraction of $\{u,w\}$ preserves the homotopy type of $\mathcal{K}$.
	\end{theorem}

	\begin{exmp}
		The following is a series of edge contractions that satisfy the link condition.
		
		
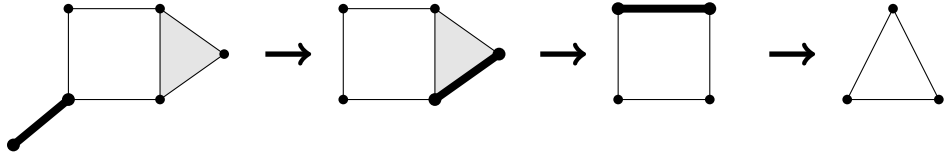
\begin{center}
			\begin{tikzpicture}[scale=\textwidth/10.5cm, inner sep=2mm]
			\coordinate (1) at (0,1);
			\coordinate (2) at (1, 1);
			\coordinate (3) at (1,0);
			\coordinate (4) at (0,0);
			\coordinate (5) at (-0.6,-0.5);
			\coordinate (6) at (1.7, 0.5);
			
			\fill[lightgray, fill opacity=0.4] (2)--(6)--(3)--(2); 
			
			\draw[] (5)--(4)--(1)--(2)--(3)--(4); 
			\draw[] (2)--(6)--(3); 
			
			\foreach \i in {1, ..., 6} {\fill (\i) circle (1.5pt);}
			
			\draw[line width=3] (5)--(4);
			\foreach \i in {4, 5} {\fill (\i) circle (2pt);}
			
			\draw[->, line width=2] (2.15, 0.5) -- (2.65, 0.5);
			
			\begin{scope}[xshift=3cm]
			\coordinate (1) at (0,1);
			\coordinate (2) at (1, 1);
			\coordinate (3) at (1,0);
			\coordinate (4) at (0,0);
			\coordinate (6) at (1.7, 0.5);
			
			\fill[lightgray, fill opacity=0.4] (2)--(6)--(3)--(2); 
			
			\draw[] (4)--(1)--(2)--(3)--(4); 
			\draw[] (2)--(6)--(3); 
			
			\foreach \i in {1, ..., 4, 6} {\fill (\i) circle (1.5pt);}
			
			\draw[line width=3] (3)--(6);
			\foreach \i in {3, 6} {\fill (\i) circle (2pt);}
			
			\draw[->, line width=2] (2.15, 0.5) -- (2.65, 0.5);
			
			\end{scope}
			
			\begin{scope}[xshift=6cm]
			\coordinate (1) at (0,1);
			\coordinate (2) at (1, 1);
			\coordinate (3) at (1,0);
			\coordinate (4) at (0,0);
			
			\draw[] (4)--(1)--(2)--(3)--(4); 
			
			\foreach \i in {1, ..., 4} {\fill (\i) circle (1.5pt);}
			
			\draw[line width=3] (1)--(2);
			\foreach \i in {1, 2} {\fill (\i) circle (2pt);}
			
			\draw[->, line width=2] (1.65, 0.5) -- (2.15, 0.5);
			
			\end{scope}
			
			\begin{scope}[xshift=8.5cm]
			\coordinate (1) at (0.5,1);
			\coordinate (3) at (1,0);
			\coordinate (4) at (0,0);
			
			\draw[] (4)--(1)--(3)--(4); 
			
			\foreach \i in {1, 3, 4} {\fill (\i) circle (1.5pt);}
			
			\end{scope}
			
			\end{tikzpicture} 			
		\end{center}
		
	\end{exmp}

	\begin{exmp}
		Without the link condition, the homotopy type of a simplicial complex under edge contractions can change, such as in the following example.
		
		
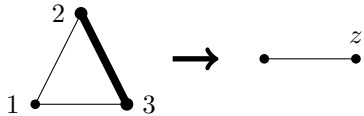
\begin{center}
			\begin{tikzpicture}[scale=\textwidth/10.5cm, inner sep=2mm]
			\begin{scope}
			\coordinate (1) at (0.5,1);
			\coordinate (2) at (1,0);
			\coordinate (3) at (0,0);
			
			\draw[] (3) node[left] {1}--(1)node[left] {2}--(2) node[right] {3}--(3); 
			
			\foreach \i in {1, 2, 3} {\fill (\i) circle (1.5pt);}
			
			\draw[line width=3] (1)--(2);
			\foreach \i in {1, 2} {\fill (\i) circle (2pt);}
			
			\draw[->, line width=2] (1.5, 0.5) -- (2.0, 0.5);
			
			\end{scope}			
			\begin{scope}[xshift=2.5cm]
			\coordinate (1) at (0,0.5);
			\coordinate (2) at (1,0.5);
			
			\draw[] (1)--(2) node[above]{$z$}; 
			
			\foreach \i in {1, 2} {\fill (\i) circle (1.5pt);}

			\end{scope}
			\end{tikzpicture}
		\end{center}
		
		The links of the vertices $\{2\}$ and $\{3\}$ both contain the vertex $\{1\}$, but $\link_{\mathcal{K}} (\{2,3\})$ is empty, so the link condition is not satisfied.
		
	\end{exmp}

	\begin{exmp}
		An edge contraction that does not satisfy the link condition may create a non-trivial cycle. For example, suppose that $\mathcal{\hat{K}}$ is a triangulation of $S^2$ on four vertices, and let $\mathcal{K}$ be a $2$-dimensional simplicial complex on $5$ vertices with facets $\{1,2,3\}$, $\{1,2,4\}$, $\{1,3,4\}$, $\{3,4,5\}$, $\{2,5\}$ as shown in Figure~\ref{fig: contraction creates cycle}. 
		There is no non-trivial $2$-cycle in $\mathcal{K}$ so $H^2(\mathcal{K})=0$, but the  contraction $\{2,5\}\mapsto z$ results in a $2$-cycle and $H^2(\mathcal{\hat{K}})\neq 0$.
		In this case the link condition is not satisfied  because $\link_{\mathcal{K}}\{2,5\}= \myempty$ but $\link_{\mathcal{K}} \{2\}\cap \link_{\mathcal{K}}\{5\}=\{\{3\},\{4\}\}$. 
		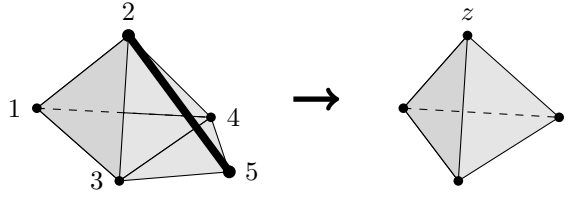
\begin{figure}
			\centering
			\begin{tikzpicture}[scale=\textwidth/10.5cm, inner sep=2mm]
			\begin{scope}
			\coordinate (1) at (0,0.2);
			\coordinate (2) at (1,1);
			\coordinate (3) at (0.9,-0.6);
			\coordinate (4) at (1.9, 0.1);
			\coordinate (5) at (2.1, -0.5);
			\coordinate (a) at (0.95, 0.15);
			
			\fill[lightgray, fill opacity=0.4] (1)--(2)--(4)--(1);
			\draw (1) node[left]{1}--(2)node[above]{2}--(4) node[right]{4};
			
			\fill[lightgray, fill opacity=0.4] (1)--(3)--(4)--(1);
			\draw (1)--(3)--(4)--(0.95, 0.15);
			
			\fill[lightgray, fill opacity=0.4] (5)--(3)--(4)--(5);
			\draw (5)node[right]{5}--(3)node[left]{3}--(4)--(5);
			
			\fill[lightgray, fill opacity=0.4] (1)--(2)--(3)--(1);
			\draw (1)--(2)--(3)--(1);
			
			\draw[dashed] (1)--(4);
			
			\foreach \i in {1, 2, 3, 4, 5} {\fill (\i) circle (1.5pt);}
			
			\draw[line width=3] (2) --(5);
			\foreach \i in {5, 2} {\fill (\i) circle (2pt);}
			
			\draw[->, line width=2] (2.8, 0.3) -- (3.3, 0.3);
			
			\end{scope}			
			\begin{scope}[xshift=3.8cm]
			\coordinate (1) at (0.2,0.2);
			\coordinate (2) at (0.9,1);
			\coordinate (3) at (0.8,-0.6);
			\coordinate (4) at (1.9, 0.1);
			
			\fill[lightgray, fill opacity=0.4] (1)--(2)--(4)--(1);
			\draw (1)--(2)--(4);
			
			\fill[lightgray, fill opacity=0.4] (1)--(3)--(4)--(1);
			\draw (1)--(3)--(4);
			
			\fill[lightgray, fill opacity=0.4] (1)--(2)--(3)--(1);
			\draw (1)--(2) node[above] {$z$}--(3)--(1); 
			
			\draw[dashed] (1)--(4);
			
			\foreach \i in {1, 2, 3, 4} {\fill (\i) circle (1.5pt);}
			
			\end{scope}
			\end{tikzpicture}
			\caption{An edge contraction without the link condition can create a non-trivial cycle}
			\label{fig: contraction creates cycle}
		\end{figure}
	\end{exmp}

	
	
	We construct cohomological classes in $H^*(\mathcal{Z}_\mathcal{K})$ on which a new pulled-back Massey product will be defined.
	
	\begin{cons} \label{cons: edge contraction}
		Let $\mathcal{\hat{K}}$ be a simplicial complex with a non-trivial $n$-Massey product $\langle \hat{\alpha}_1, \ldots, \hat{\alpha}_n\rangle \subset H^*(\mathcal{Z}_\mathcal{\hat{K}})$.	
		By Hochster's theorem, every class $\hat{\alpha}_i \in H^*(\mathcal{Z}_\mathcal{\hat{K}})$ has a corresponding class
		\[
		\hat{\alpha}_i \in \widetilde{H}^{p_i}(\mathcal{\hat{K}}_{\hat{J}_i})
		\]
		for a set of vertices $\hat{J}_i \subset V(\mathcal{\hat{K}})$.
		When $\langle \hat{\alpha}_1, \ldots, \hat{\alpha}_n \rangle$ is non-trivial, the sets of vertices $\hat{J}_i$, $\hat{J}_j$ are disjoint for any $i\neq j$. 
		
		Suppose that there is a simplicial complex $\mathcal{K}$ and a series of edge contractions $\varphi\colon \mathcal{K} \to \mathcal{\hat{K}}$  satisfying the link condition. 
		Let the vertices in $V(\mathcal{\hat{K}})$ be ordered and suppose that all of the vertices in $\hat{J}_i$ come before those of $\hat{J}_{i+1}$.
		For a set of $p$-simplices $P \subset \mathcal{\hat{K}}$, let 
		\[
		\varphi^{-1}_p(P)=\{\sigma\in \mathcal{K}  \suchthat |\sigma|=p+1 \text{ and } \varphi(\sigma)=\hat{\sigma}\text{ for } \hat{\sigma}\in P \}.
		\]
		Suppose that the vertices $V(\mathcal{K})$ are ordered in such a way that for any vertex $\hat{v}$ that comes before $\hat{w}$ in $\mathcal{\hat{K}}$, each vertex $v\in \varphi^{-1}_0(\hat{v})$ comes before every $w \in \varphi^{-1}_0(\hat{w})$. 
		Let $J_i=\varphi^{-1}_0(\hat{J}_i)\subset V(\mathcal{K})$.
		Then by the order on $V(\mathcal{K})$, all vertices in $J_i$ come before those in $J_{i+1}$. 
		Also $J_i \cap J_j=\myempty$ for any $i\neq j$ since  $\hat{J}_i \cap \hat{J}_j = \myempty$ and $\varphi^{-1}_0(\hat{v}) \cap \varphi^{-1}_0(\hat{w})= \myempty$ for any vertices $\hat{v}, \hat{w} \in \mathcal{\hat{K}}$, $\hat{v}\neq \hat{w}$. 
	\end{cons}

	Let $\hat{a}_i$ be a cocycle representing $\hat{\alpha}_i \in \widetilde{H}^{p_i}(\mathcal{\hat{K}}_{\hat{J}_i})$. Let $S_{\hat{a}_i}$ be the support of $\hat{a}_i$, that is, the set of $p_i$-simplices $\hat{\sigma}_i \in \mathcal{\hat{K}}_{\hat{J}_i}$ such that 
	\[
	\hat{a}_i=\sum_{\hat{\sigma}\in S_{\hat{a}_i}} c_{\hat{\sigma}} \chi_{\hat{\sigma}} \in C^{p_i} (\mathcal{\hat{K}}_{\hat{J}_i})
	\]
	for non-zero coefficients $c_{\hat{\sigma}_i} \in \mathbf{k}$.
	Define $a_i\in C^{p_i}(\mathcal{K}_{J_i})$ to be the cochain
	\begin{equation}\label{eq:cochain a}
		a_i=\sum_{\hat{\sigma} \in S_{\hat{a}_i}} c_{\hat{\sigma}} \sum_{\sigma \in \varphi^{-1}_{p_i}(\hat{\sigma})} \chi_{\sigma}.
	\end{equation}
	Since $a_i$ is a pullback of $\hat{a}_i$ along $\varphi$, $a_i$ is a cocycle and $\alpha_i=[a_i] \in \widetilde{H}^{p_i}(\mathcal{K}_{J_i})$ is non-zero.
	
	\begin{exmp}
		Let $\mathcal{K}_{J_i}$, $\mathcal{\hat{K}}_{\hat{J}_i}$ be the simplicial complexes as shown below, where $\mathcal{\hat{K}}_{\hat{J}_i}$  is obtained from $\mathcal{K}_{J_i}$ by contracting the edges $e_2=\{2,3\}\mapsto \{\hat{2}\}$ and $e_5=\{4, 5\}\mapsto \{\hat{3}\}$. The cohomology class $\hat{\alpha}_i\in \widetilde{H}^1(\mathcal{\hat{K}}_{\hat{J}_i})$ is represented by the cocycle $\chi_{\hat{e}}$, so let $S_{\hat{a}_i}=\{\hat{e}\}$.
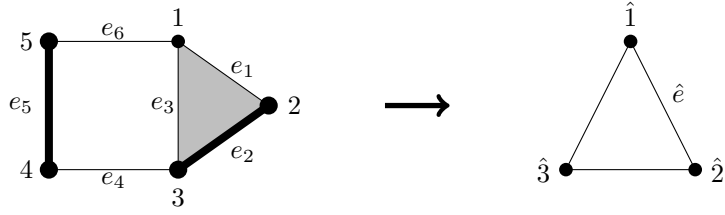
		\begin{center}
			\begin{tikzpicture}[scale=1.7, inner sep=2mm]
			\begin{scope}
			\coordinate (5) at (0,1);
			\coordinate (1) at (1, 1);
			\coordinate (2) at (1.7, 0.5);
			\coordinate (3) at (1,0);
			\coordinate (4) at (0,0);
			\coordinate (e1) at (0.5,1);
			\coordinate (e2) at (1,0.5);
			\coordinate (e3) at (0.5,0.5);
			\coordinate (e4) at (0.5,0);
			\coordinate (e5) at (0,0.5);
			
			\fill[lightgray] (2)--(1)--(3)--(2);
			
			\draw[] (4) node[left]{$4$} --(5) node[left]{$5$}--(1) node[above]{$1$}--(2)--(3)--(4); 
			\draw[] (1)--(3); 
			
			\node[above]  at (1.5,0.6) {$e_1$};
			\node[below]  at (1.5, 0.3){$e_2$};
			\node[left] at (1.1, 0.5){$e_3$};
			\node[below] at (0.5,0.1) {$e_4$};
			\node[left] at (0, 0.5) {$e_5$};
			\node[above] at (0.5, 0.9) {$e_6$};
			
			\foreach \i in {1, ..., 5} {\fill (\i) circle (1.5pt);}
			
			\draw[line width=3] (3)node[below] {$3$}--(2)node[right] {$2$};
			\draw[line width=3] (4)--(5);
			\foreach \i in {2, 3,4,5} {\fill (\i) circle (2pt);}
			
			\draw[->, line width=2] (2.6, 0.5) -- (3.1, 0.5);
			
			\end{scope}			
			\begin{scope}[xshift=4cm]
			\coordinate (1) at (0.5,1);
			\coordinate (3) at (1,0);
			\coordinate (4) at (0,0);
			\coordinate (e) at (0.7,0.6);
			
			\draw[] (4)--(1) node[above] {$\hat{1}$}--(3) node[right] {$\hat{2}$}--(4) node[left] {$\hat{3}$}; 
			\draw (e) node[right] {$\hat{e}$};
			
			\foreach \i in {1, 3, 4} {\fill (\i) circle (1.5pt);}
			\end{scope}
			\end{tikzpicture}
		\end{center}
		The contraction of $e_2$  satisfies the link condition, since $\link_{\mathcal{K}}(e_2)= \link_{\mathcal{K}}\{2\}\cap \link_{\mathcal{K}}\{3\}=\{1\}$. 
		Under the map $\varphi\colon \mathcal{K} \to \mathcal{\hat{K}}$, $\varphi_{1}^{-1}(\hat{e})=\{e_1, e_3\}$. So by \eqref{eq:cochain a}, $a_i$ is the cochain
		\[
		a_i=\chi_{e_1}+\chi_{e_3} \in C^1(\mathcal{K}_{J_i}).
		\]
		This is a cocycle since $d(a_i)=\chi_{\{1,2,3\}} - \chi_{\{1,2,3\}}=0$.
	\end{exmp}

	For the Massey product $\langle \hat{\alpha}_1, \ldots, \hat{\alpha}_n \rangle\subset H^{(p_1+\cdots +p_n) +|\hat{J}_1\cup \cdots \cup \hat{J}_n|+2}(\mathcal{Z}_\mathcal{\hat{K}})$, there is a defining system $(\hat{a}_{i,k})$ for cochains $\hat{a}_{i,k}\in C^{p_i+\cdots +p_k}(\mathcal{\hat{K}}_{\hat{J}_i \cup \ldots \cup \hat{J}_k})$, $1\leqslant i \leqslant k \leqslant n$ and $(i,k)\neq (1,n)$. 
	Suppose that
	\begin{equation} \label{eq: hat{a_ik}}
		\hat{a}_{i,k}= 
		\sum_{\hat{\tau}\in S_{\hat{a}_{i,k}}} c_{\hat{\tau}} \chi_{\hat{\tau}}
	\end{equation}
	for simplices $\hat{\tau}\in S_{\hat{a}_{i,k}}\subset \mathcal{\hat{K}}_{\hat{J}_i \cup \ldots \cup \hat{J}_k}$, non-zero coefficients $c_{\hat{\tau}}\in \mathbf{k}$.
	Then 
	\begin{align*}
		d(\hat{a}_{i,k})=
		\sum_{\hat{\tau}\in S_{\hat{a}_{i,k}}} c_{\hat{\tau}} \left(
		\sum_{\hat{j}\in \hat{J}_i \cup \cdots \cup \hat{J}_k\setminus V(\hat{\tau})} \varepsilon(\hat{j}, \hat{j}\cup \hat{\tau}) \chi_{\hat{j}\cup \hat{\tau}} \right)
	\end{align*}
	is equal to 
	\begin{equation}\label{eq: hat(a_i,r) hat(a_r+1,k) for edge contraction}
		\begin{aligned}
			\sum_{r=i}^{k-1} (-1)^{1+\mydeg(\hat{a}_{i,r})}\hat{a}_{i,r}\hat{a}_{r,k} 
			=
			\sum_{r=i}^{k-1}
			(-1)^{1+\mydeg(\hat{a}_{i,r})}
			 c
			\left(\sum_{\hat{\nu}\in S_{\hat{a}_{i,r}}} \ \sum_{\hat{\eta}\in S_{\hat{a}_{r+1,k}}} 
			c_{\hat{\nu}} c_{\hat{\eta}}  
			\chi_{\hat{\nu}\cup \hat{\eta}}\right)
		\end{aligned}
	\end{equation}
	where $c=(-1)^{|\hat{J}_i \cup \cdots \cup\hat{J}_r| (p_{r+1}+\cdots +p_k+1)}$ comes from the product of $\hat{a}_{i,r}$ and $\hat{a}_{r,k}$, as in Lemma~\ref{lem: multiplication of general cochains}, and $(-1)^{1+\mydeg(\hat{a}_{i,r})}=(-1)^{(p_i+\ldots +p_r)+|\hat{J}_i \cup \cdots \cup\hat{J}_r|}$.
	We use this defining system to construct a defining system for $\langle \alpha_1, \ldots, \alpha_n \rangle$.

	%

	\begin{prop} \label{prop: Massey defined edge contractions}\label{lem: d(a_ik)=a_i,r a_r,k contracting}
		Let $\mathcal{K}$ be a simplicial complex that maps to $\mathcal{\hat{K}}$ by edge contractions satisfying the link condition. Then there is an $n$-Massey product $\langle \alpha_1, \ldots, \alpha_n \rangle$ defined on $H^{*}(\mathcal{Z}_\mathcal{K})$.
	\end{prop}
	
	\begin{proof}
		For every $i\in \{1, \ldots, n\}$, let $\alpha_i=[a_i]$ for $a_i$ as in \eqref{eq:cochain a}. We start by constructing a defining system $(a_{i,k})$ for $\langle \alpha_1, \ldots, \alpha_n \rangle\subset H^*(\mathcal{Z}_\mathcal{K})$, where $a_{i,k} \in C^{p_i+\cdots +p_k}(\mathcal{K}_{J_i \cup \cdots \cup J_k})$.
		Define
		\begin{equation}\label{eq: a_ik}
			a_{i,k}= \theta_{i,k} \;\hat{\theta}_{i,k} \sum_{\hat{\tau}\in S_{\hat{a}_{i,k}}} c_{\hat{\tau}} \left(\sum_{\tau \in \varphi^{-1}_{p_i + \cdots +p_k}(\hat{\tau})} \chi_\tau\right)
		\end{equation}
		for $S_{\hat{a}_{i,k}}$ and $c_{\hat{\tau}}\in \mathbf{k}$ from \eqref{eq: hat{a_ik}}, $\theta_{i,i}=1=\hat{\theta}_{i,i}$, and 
		\begin{equation}\label{eq: theta_ik stretching}
			\begin{aligned}
				&\theta_{i,k}=(-1)^{|J_i|(p_{i+1}+\cdots+p_k)}(-1)^{|J_{i+1}|(p_{i+2}+\cdots+p_k)} \cdots (-1)^{|J_{k-1}|p_k}\\
				&\hat{\theta}_{i,k}=(-1)^{|\hat{J}_i|(p_{i+1}+\cdots+p_k)}(-1)^{|\hat{J}_{i+1}|(p_{i+2}+\cdots+p_k)} \cdots (-1)^{|\hat{J}_{k-1}|p_k}.
			\end{aligned}
		\end{equation}
		Since $\theta_{i,i}=1=\hat{\theta}_{i,i}$, $a_{i,i}=a_i$ as in \eqref{eq:cochain a}.
		We show that $d(a_{i,k})=\sum_{r=i}^{k-1} \doubleoverline{a_{i,r}}a_{r,k}$, where $\doubleoverline{a_{i,r}}=(-1)^{1+\mydeg a_{i,r}}a_{i,r}$ as in Definition~\ref{def: degree of elements in C(KJ)}.

		Applying the coboundary map to $a_{i,k}$, $d(a_{i,k})$ is 
		\begin{align}
			\theta_{i,k} &\,\hat{\theta}_{i,k} \sum_{\hat{\tau}\in S_{\hat{a}_{i,k}}} c_{\hat{\tau}} \left(\sum_{\tau \in \varphi^{-1}_{p_i + \cdots +p_k}(\hat{\tau})} \ \sum_{j\in J_i \cup \cdots \cup J_k\setminus V(\tau)} \varepsilon(j, j\cup \tau) \chi_{j\cup\tau}\right) \nonumber\\
			&=\theta_{i,k} \, \hat{\theta}_{i,k} \sum_{\hat{\tau}\in S_{\hat{a}_{i,k}}} c_{\hat{\tau}} 
			\left(\sum_{\tau \in \varphi^{-1}_{p_i + \cdots +p_k}(\hat{\tau})} \ 
			\sum_{j\in J_i \cup \cdots \cup J_k \setminus \varphi^{-1}_0(V(\hat{\tau}))} \varepsilon(j, j\cup \tau) \chi_{j\cup\tau}\right)+\label{eq: d(a_ij) first part edge contractions}\\
			& \qquad
			+\theta_{i,k} \,\hat{\theta}_{i,k} \sum_{\hat{\tau}\in S_{\hat{a}_{i,k}}} c_{\hat{\tau}} 
			\left(\sum_{\tau \in \varphi^{-1}_{p_i + \cdots +p_k}(\hat{\tau})} \ 
			\sum_{j\in \varphi^{-1}_0(V(\hat{\tau})) \setminus V(\tau)} \varepsilon(j, j\cup \tau) \chi_{j\cup\tau}\right). \label{eq: d(a_ij) second part edge contractions}
		\end{align}
		For any $(p_i+\cdots +p_k)$-simplex $\hat{\tau} \in S_{\hat{a}_{i,k}}$ and any $\tau\in \varphi^{-1}_{p_i + \cdots +p_k}(\hat{\tau})$, first suppose that there is a vertex $j\in \varphi^{-1}_0(V(\hat{\tau})) \setminus V(\tau)$ such that $j\cup \tau \in \mathcal{K}$. 
		Then $j\cup \tau=\bar{\tau}\in \varphi^{-1}_{p_i + \cdots +p_k+1}(\hat{\tau})$ and there is a vertex $i\in V(\tau)$ such that $\varphi(i)=\varphi(j)$. 
		Thus $j \cup \tau\setminus i \in \varphi^{-1}_{p_i + \cdots +p_k}(\hat{\tau})$. 
		Moreover, $i, j$ are consecutive vertices in $V(\bar{\tau})$ by the order of vertices  in $\mathcal{K}$ defined in Construction~\ref{cons: edge contraction}, so $\varepsilon(j, \bar{\tau})=-\varepsilon(i, \bar{\tau})$. 
		Therefore \eqref{eq: d(a_ij) second part edge contractions} is zero since all summands cancel out in pairs, that is, for any $\hat{\tau}\in S_{\hat{a}_{i,k}}$,
		\begin{multline*}
		\sum_{\tau \in \varphi^{-1}_{p_i + \cdots +p_k}(\hat{\tau})} \ 
		\sum_{j\in \varphi^{-1}_0(V(\hat{\tau})) \setminus V(\tau)} \varepsilon(j, j\cup \tau) \chi_{j\cup\tau}= \\
		= \sum_{\substack{\bar{\tau}\in \varphi^{-1}_{p_i + \cdots +p_k+1}(\hat{\tau}), \\
				i,j\in \bar{\tau} \suchthat \varphi(i)=\varphi(j) }} 
		\varepsilon(j, \bar{\tau}) \chi_{\bar{\tau}} + \varepsilon(i, \bar{\tau}) \chi_{\bar{\tau}} 
		=0.
		\end{multline*}
		
		Consider summands in \eqref{eq: d(a_ij) first part edge contractions}. For any $j\in J_i \cup \cdots \cup J_k \setminus \varphi^{-1}_0(V(\hat{\tau}))$, $\varphi(j)\notin V(\hat{j})$. So for any simplex $j\cup \tau\in \mathcal{K}$ with $j\in J_i \cup \cdots \cup J_k \setminus \varphi^{-1}_0(V(\hat{\tau}))$, there is a simplex $\varphi(j)\cup \hat{\tau}\in \mathcal{\hat{K}}$. Therefore any summand in \eqref{eq: d(a_ij) first part edge contractions} has a corresponding summand in the expression for $d(\hat{a}_{i,k})$. Hence we rewrite \eqref{eq: d(a_ij) first part edge contractions} as
		\begin{equation}\label{eq: d(a_ik) related to d(hat(a)_ik) edge contraction}
			d(a_{i,k})= \theta_{i,k} \, 
			\hat{\theta}_{i,k} \sum_{\hat{\tau}\in S_{\hat{a}_{i,k}}} c_{\hat{\tau}} 
			\left(\sum_{\hat{j}\in \hat{J}_i \cup \cdots \cup \hat{J}_k\setminus V(\hat{\tau})} \;
			\sum_{j\cup\tau \in \varphi^{-1}_{p_i + \cdots +p_k+1}(\hat{j}\cup\hat{\tau})}
			\varepsilon(j, j\cup\tau) \chi_{j\cup\tau}\right)
		\end{equation}
		where, by  the order of  vertices in $\mathcal{K}$,  $\varepsilon(j, j\cup\tau)=\varepsilon (\hat{j}, \hat{j}\cup \hat{\tau})$.
		Since $d(\hat{a}_{i,k})=\sum_{r=i}^{k-1} \doubleoverline{\hat{a}_{i,r}}\hat{a}_{r,k}$, the expression in \eqref{eq: d(a_ik) related to d(hat(a)_ik) edge contraction} can be written in terms of the expression in \eqref{eq: hat(a_i,r) hat(a_r+1,k) for edge contraction}. Thus $d(a_{i,k})$ is equal to
		\begin{equation}\label{eq: simplified d(a_ik) for edge contraction}
			\begin{aligned}
				\theta_{i,k}\, \hat{\theta}_{i,k} 
				\sum_{r=i}^{k-1}\ 
				(-1)^{1+\mydeg(\hat{a}_{i,r})}
				c
				\left(\sum_{\hat{\nu}\in S_{\hat{a}_{i,r}}} \ \sum_{\hat{\eta}\in S_{\hat{a}_{r+1,k}}} 
				c_{\hat{\nu}} c_{\hat{\eta}} 
				\left(\sum_{\zeta\in \varphi^{-1}_{p_i + \cdots + p_k+1}(\hat{\nu}\cup\hat{\eta})} \ \chi_\zeta\right)\right)
			\end{aligned}
		\end{equation}
		where $c=(-1)^{|\hat{J}_i \cup \cdots \cup\hat{J}_r| (p_{r+1}+\cdots +p_k+1)}$ comes from the product of $\hat{a}_{i,r}$ and $\hat{a}_{r,k}$, as in Lemma~\ref{lem: multiplication of general cochains}, and $(-1)^{1+\mydeg(\hat{a}_{i,r})}=(-1)^{(p_i+\cdots +p_r)+|\hat{J}_i \cup \cdots \cup\hat{J}_r|}$.
		
		Any simplex $\zeta\in \varphi^{-1}_{p_i + \cdots + p_k+1}(\hat{\nu}\cup\hat{\eta})$ is on $p_i+\cdots+p_k+2$ vertices and so can be written as $\nu \cup \eta$ for $\nu$ the restriction of $\zeta$ to its first $p_i + \cdots + p_r +1$ vertices, and $\eta$ the restriction of $\zeta$ to its last $p_{r+1} + \cdots + p_k+1$ vertices. 
		Then $\nu \in \varphi^{-1}_{p_i+\cdots+p_r}(\hat{\nu})$ and $\eta \in \varphi^{-1}_{p_{r+1}+\cdots +p_k}(\hat{\eta})$.
		Furthermore, 
		$
		\hat{\theta}_{i,k} \;(-1)^{1+\mydeg(\hat{a}_{i,r})} \; c=
		(-1)^{(p_i+\cdots +p_r)} \; \hat{\theta}_{i,r}\; \hat{\theta}_{r+1,k}.
		$
		So \eqref{eq: simplified d(a_ik) for edge contraction} may be rewritten as 
		\begin{multline}\label{eq: d(a_ik) stretching}
			d(a_{i,k})=
			\sum_{r=i}^{k-1}\ 
			(-1)^{(p_i+\cdots +p_r)}
			\; \theta_{i,k}\; \hat{\theta}_{i,r}\; \hat{\theta}_{r+1,k}\cdot \\
			\cdot \left( \sum_{\hat{\nu}\in S_{\hat{a}_{i,r}}} \ \sum_{\hat{\eta}\in S_{\hat{a}_{r+1,k}}} 
			c_{\hat{\nu}} \;c_{\hat{\eta}} \;
			\left(\sum_{\nu\in \varphi^{-1}_{p_i + \cdots + p_r}(\hat{\nu})} 
			\ \sum_{\eta\in \varphi^{-1}_{p_{r+1} + \cdots + p_k}(\hat{\eta})} \ \chi_{\nu\cup\eta}\right)\right)_.
		\end{multline}
		
		Comparatively, the product $\sum_{r=i}^{k-1} (-1)^{1+\mydeg(a_{i,r})}a_{i,r}a_{r,k}$ is
		\begin{multline}\label{eq: a_ir a_r+1,k stretching}
			\sum_{r=i}^{k-1}\  
			(-1)^{1+\mydeg(a_{i,r})}
			(-1)^{|J_i \cup \cdots \cup J_r| (p_{r+1}+\cdots +p_k+1)}
			\theta_{i,r}\;\theta_{r+1,k} \;\hat{\theta}_{i,r} \;\hat{\theta}_{r+1,k} \cdot \\
			\cdot \left(\sum_{\hat{\nu}\in S_{\hat{a}_{i,r}}} \ \sum_{\hat{\eta}\in S_{\hat{a}_{r+1,k}}} 
			c_{\hat{\nu}} \ c_{\hat{\eta}}
			\left(\sum_{\nu\in \varphi^{-1}_{p_i + \cdots + p_r}(\hat{\nu})} \ \sum_{\eta\in \varphi^{-1}_{p_{r+1} + \cdots + p_k}(\hat{\eta})}\ \chi_{\nu\cup\eta}\right)\right)
		\end{multline}
		where the sign $(-1)^{|J_i \cup \cdots \cup J_r| (p_{r+1}+\cdots +p_k+1)}$ comes from the product of $a_{i,r}$ and $a_{r+1,k}$ as in Lemma~\ref{lem: multiplication of general cochains}, and $(-1)^{1+\mydeg(a_{i,r})}=(-1)^{(p_i+\cdots+p_r)+|J_i\cup \cdots \cup J_r|}$.
		Using the expression for $\theta_{i,k}$  in \eqref{eq: theta_ik stretching}, 
		\[
		(-1)^{1+\mydeg(a_{i,r})} 
		(-1)^{|J_i \cup \cdots \cup J_r| (p_{r+1}+\cdots +p_k+1)} \; \theta_{i,r}\; \theta_{r+1,k} \;
		= (-1)^{(p_i+\cdots +p_r)}\; \theta_{i,k}.
		\]
		Therefore the expressions in \eqref{eq: d(a_ik) stretching} and \eqref{eq: a_ir a_r+1,k stretching} are equal.
		
		Hence $d(a_{i,k})=\sum_{r=i}^{k-1} \doubleoverline{a_{i,r}}a_{r,k}$, and so $(a_{i,k})$ is a defining system for the Massey product $\langle \alpha_1, \ldots, \alpha_n \rangle$. 
	\end{proof}
	
	\begin{exmp}
		Let $J_1=\{1, 2, 3\}$, $\hat{J}_1=\{\hat{1}, \hat{2} \}$, $J_2=\{4,5\}$ and $\hat{J}_2=\{\hat{4}, \hat{5}\}$. Suppose that $\mathcal{K}_{J_1\cup J_2}$ and $\mathcal{\hat{K}}_{\hat{J}_1\cup \hat{J}_2}$ are the simplicial complexes shown below, where $\mathcal{K}_{J_1\cup J_2}$ maps onto $\mathcal{\hat{K}}_{\hat{J}_1\cup \hat{J}_2}$ by the edge contraction $\{2,3\}\mapsto \{\hat{2}\}$. 
		
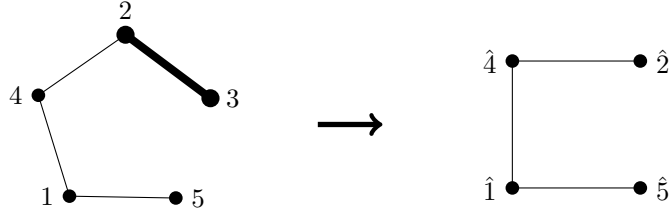
		\begin{center}
			\begin{tikzpicture}[scale=1.7, inner sep=2mm]	
			\begin{scope}
			\foreach \i in {1, ..., 5}
			\fill (\i*360/5-55:0.7) coordinate (\i) circle(1.5 pt);
			
			\draw (1)--(2)--(3)--(4)--(5);
			
			\node[right] at (1){$3$};
			\node[above] at (2){$2$};
			\node[left] at (3){$4$};
			\node[left] at (4){$1$};
			\node[right] at (5){$5$};
			
			\draw[line width=3] (1)--(2);
			\foreach \i in {1,2} {\fill (\i) circle (2pt);}
			
			\draw[->, line width=2] (1.5, 0) -- (2, 0);
			
			\end{scope}			
			\begin{scope}[xshift=3.5cm]
			\foreach \i in {1, ..., 4}
			\fill (\i*360/4-45:0.7) coordinate (\i) circle(1.5 pt);
			
			\draw (1)--(2)--(3)--(4);
			
			\node[right] at (1){$\hat{2}$};
			\node[left] at (2){$\hat{4}$};
			\node[left] at (3){$\hat{1}$};
			\node[right] at (4){$\hat{5}$};
			
			\end{scope}
			\end{tikzpicture}
		\end{center}
		Suppose that $\hat{a}_{1}=\chi_{\hat{2}}\in C^0(\mathcal{\hat{K}}_{\hat{J}_1})$, $\hat{a}_2=\chi_{\hat{4}} \in C^0(\mathcal{\hat{K}}_{\hat{J}_2})$, and $\hat{a}_{1,2}=-\chi_{\hat{2}}\in C^0(\mathcal{\hat{K}}_{\hat{J}_1\cup \hat{J}_2})$. Then $d(\hat{a}_{1,2})=\chi_{\hat{2},\hat{4}}=(-1)^{1+\mydeg \hat{a}_1}\hat{a}_1 \hat{a}_2$. 
		By \eqref{eq:cochain a}, $a_1=\chi_2+\chi_3\in C^0(\mathcal{K}_{J_1})$ and $a_2=\chi_4 \in C^0(\mathcal{K}_{J_2})$. 
		By \eqref{eq: a_ik}, $a_{1,2}=-\chi_2-\chi_3\in C^0(\mathcal{K}_{J_1\cup J_2})$, since $\theta_{1,2}=1$.
		We check that $d(a_{1,2})=(\chi_{2,4} +\chi_{2,3})-\chi_{2,3}=\chi_{2,4} = (-1)^{1+\mydeg a_1} a_1 a_2$. Hence $d(a_{1,2})=\doubleoverline{a_1} a_2$.
	\end{exmp}

	\begin{exmp}[a]\label{ex: Massey defined example - edge contractions}
		Let $\mathcal{\hat{K}}_1$ be a triangulation of $S^1$ on three vertices, $\{\hat{1}, \hat{2}, \hat{3}\}$. Let $\mathcal{\hat{K}}_2=\{\{\hat{5}\}, \{\hat{6}\}\}$, and let  $\mathcal{\hat{K}}_3=\{\{\hat{7}\}, \{\hat{8}\}\}$. Let $\hat{\alpha}_1=[\chi_{\hat{1}\hat{3}}]\in \widetilde{H}^1(\mathcal{\hat{K}}_1)$, $\hat{\alpha}_2=[\chi_{\hat{5}}]\in \widetilde{H}^0(\mathcal{\hat{K}}_2)$ and $\hat{\alpha}_3=[\chi_{\hat{7}}]\in \widetilde{H}^0(\mathcal{\hat{K}}_3)$. 
		Let $\mathcal{\hat{K}}=\sd_{\{\hat{5},\hat{8}\}} \sd_{\{\hat{1}, \hat{3},\hat{6}\}} \mathcal{\hat{K}}_1 * \mathcal{\hat{K}}_2 *\mathcal{\hat{K}}_3$ be a simplicial complex on the vertices $\{\hat{1}, \hat{2}, \hat{3}, \hat{5}, \hat{6}, \hat{7}, \hat{8}\}$.  The simplicial complex $\mathcal{\hat{K}}_{\hat{1}, \hat{2}, \hat{3}, \hat{5}, \hat{6}}$ is shown in Figure~\ref{fig: options for omega', edge-contracted K}.
		By Theorem~\ref{thm: joins}, there is a non-trivial triple Massey product $\langle \hat{\alpha}_1, \hat{\alpha}_2, \hat{\alpha}_3 \rangle \subset H^*(\mathcal{Z}_\mathcal{\hat{K}})$.
		\begin{figure}[ht]
			\centering
			\begin{minipage}{0.4\textwidth}
				\centering
				\begin{tikzpicture} [scale=1.5, inner sep=2mm]
				\coordinate (4) at (-0.15,0);
				\coordinate (1) at (1,0);
				\coordinate (2) at (0.6,0.5);
				\coordinate (3) at (1.7,0.5);
				\coordinate (5) at (0.8,1.5);
				\coordinate (6) at (0.8,-1.2);
				
				\fill[lightgray, fill opacity=0.8] (1)--(3)--(5)--(4)--(6)--(1);
				\fill[lightgray, fill opacity=0.4] (3)--(5)--(4)--(6)--(3);
				
				\draw (4) -- (1) node[inner sep=1mm, below right] {1} -- (3) node[right] {3} -- (5) node[above] {5} -- (4)  node[left] {4} -- (6) node[below] {6} -- (3); \draw (5) -- (1) -- (6);
				\draw[dashed] (5) -- (2) node[below] {2} -- (6); \draw[dashed] (4) -- (2) -- (3);
				\draw[line width=2.7] (1)--(4);
				
				\foreach \i in {1,...,6} {\fill (\i) circle (1.8pt);}; 
				\end{tikzpicture} 
				\subcaption{The simplicial complex $\mathcal{K}_{1,2,3,4,5,6}$, which is missing the simplex $\{1,3,6\}$}
				\label{fig: options for omega', non-edge-contracted K}
			\end{minipage}\qquad
			\begin{minipage}{0.4\textwidth}
				\centering
				\begin{tikzpicture} [scale=1.5, inner sep=2mm]
				\coordinate (2) at (0,0.4);
				\coordinate (1) at (0.9,0);
				\coordinate (3) at (1.7,0.6);
				\coordinate (5) at (0.8,1.5);
				\coordinate (6) at (0.8,-1.2);
				
				\fill[lightgray, fill opacity=0.8] (1)--(3)--(5)--(2)--(6)--(1);
				\fill[lightgray, fill opacity=0.4] (3)--(5)--(2)--(6)--(3);
				
				\draw (1) node[inner sep=1mm, below right] {$\hat{1}$} -- (3) node[right] {$\hat{3}$} -- (5) node[above] {$\hat{5}$} --(2) node[left] {$\hat{2}$}-- (6) node[below] {$\hat{6}$} -- (3); \draw (5) -- (1) -- (6); \draw (1)--(2);
				\draw[dashed] (2) -- (3);
				
				\foreach \i in {1, 2,3,5,6} {\fill (\i) circle (1.8pt);}; 
				\end{tikzpicture} 
				\subcaption{The simplicial complex $\mathcal{\hat{K}}_{\hat{1}, \hat{2}, \hat{3}, \hat{5}, \hat{6}}$, which is missing the simplex $\{\hat{1}, \hat{3}, \hat{6}\}$.}
				\label{fig: options for omega', edge-contracted K}
			\end{minipage}
			\caption{The simplicial complex $\mathcal{K}_{1,2,3,4,5,6}$ maps to $\mathcal{\hat{K}}_{\hat{1}, \hat{2}, \hat{3}, \hat{5}, \hat{6}}$	by contracting the edge $\{1,4\}\mapsto\{\hat{1}\}$.}		
		\end{figure}
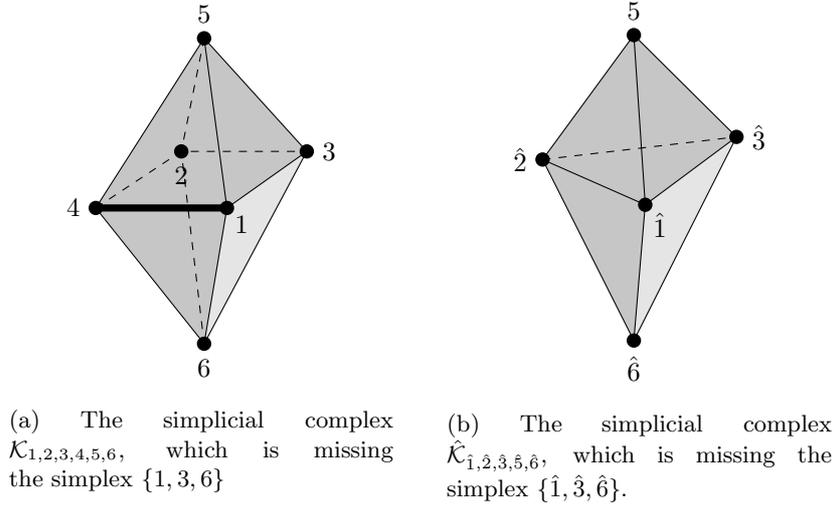
		
		Let $\mathcal{K}$ be the simplicial complex on vertices $\{1, \ldots, 8\}$ that edge contracts to $\mathcal{\hat{K}}$ by contracting the edge $\{1,4\}\mapsto\{\hat{1}\}$, which satisfies the link condition.
		The contraction of the full subcomplex $\mathcal{K}_{J_1\cup J_2}$ is shown in Figure~\ref{fig: options for omega', non-edge-contracted K}. 
		By Construction~\ref{cons: edge contraction}, there are cocycles $a_1=\chi_{13}\in C^1(\mathcal{K}_{J_1})$, $a_2=\chi_5\in C^0(\mathcal{K}_{J_2})$, $a_3=\chi_7\in C^0(\mathcal{K}_{J_3})$.
		The product $a_1a_2$ is $\chi_{13}\chi_5=(-1)^{4}\chi_{135}=\chi_{135}$. 
		If $\hat{a}_{1,2}=\chi_{\hat{1}\hat{3}}$, then using \eqref{eq: a_ik} we  construct $a_{1,2}=\theta_{1,2} \hat{\theta}_{1,2}\chi_{13}=-\chi_{13}$.
		Alternatively, if $\hat{a}_{1,2}=-\chi_{\hat{1}\hat{6}}-\chi_{\hat{1}\hat{2}}-\chi_{\hat{1}\hat{5}}$, then $S_{\hat{a}_{1,2}}=\{\{\hat{1},\hat{6}\}, \{ \hat{1},\hat{2}\}, \{\hat{1},\hat{5}\} \}$. 
		So 
		$	\varphi_{1}^{-1}(\{\hat{1},\hat{6}\})=\{ \{1,6\}, \{4,6\} \}$, 
		$ \varphi_{1}^{-1}(\{\hat{1},\hat{2}\})=\{ \{2,4\} \},
		$
		and 
		$\varphi_{1}^{-1}(\{\hat{1},\hat{5}\})=\{ \{1,5\}, \{4,5\}\}$.
		By \eqref{eq: a_ik}, $a_{1,2}=-\theta_{1,2} \hat{\theta}_{1,2}(\chi_{16}+\chi_{46}+\chi_{24}+\chi_{45}+\chi_{15})=\chi_{16}+\chi_{46}+\chi_{24}+\chi_{45}+\chi_{15}$.
		
		(b).
		In the proof of Proposition~\ref{prop: Massey defined edge contractions}, we showed that the pullback of a defining system $\langle \hat{\alpha}_1,\ldots, \hat{\alpha}_n\rangle$ is a defining system for $\langle \alpha_1, \ldots, \alpha_n \rangle$.
		However there are defining systems for $\langle \alpha_1, \ldots, \alpha_n \rangle$
		that are not pullbacks of defining systems for $\langle \hat{\alpha}_1,\ldots, \hat{\alpha}_n\rangle$.
		For example, let $a_1, a_2, a_3, \hat{a}_1, \hat{a}_2, \hat{a}_3$ be as in Part (a).
		Let $a_{1,2}=-\chi_{16}-\chi_{14}-\chi_{15}$. For the edge $\{1,4\}\in \mathcal{K}$, $\{1,4\}\notin \varphi_{1}^{-1}(\hat{e})$ for any edge $\hat{e}\in \mathcal{\hat{K}}$, so $a_{1,2}$ is not a pullback of any $\hat{a}_{1,2}$.
		However for $\chi_1\in C^*(\mathcal{K}_{1,2,3,4,5,6})$, 
		\begin{align*}
			a_{1,2}-d(\chi_1)&=-\chi_{16}-\chi_{14}-\chi_{15}-(\chi_{16}+\chi_{14}+\chi_{15}+\chi_{13}) \\
			&=-\chi_{13}=\theta_{1,2} \hat{\theta}_{1,2}\sum_{\tau \in \varphi^{-1}_{p_1 + p_2}(\hat{1}\hat{3})} \chi_{\tau}.
		\end{align*}
		Therefore $a_{1,2}$ differs from the pullback of $\hat{a}_{1,2}=\chi_{\hat{1}\hat{3}}$ by a coboundary.
	\end{exmp}
	
	In order to prove that $\langle \alpha_1, \ldots, \alpha_n \rangle$ is non-trivial, we show that for every defining system for $\langle \alpha_1, \ldots, \alpha_n \rangle$, its associated cocycle is homologous to the pullback of an associated cocycle for a defining system for the non-trivial Massey product $\langle \hat{\alpha}_1,\ldots, \hat{\alpha}_n\rangle$.
	
	\begin{prop} \label{prop: Massey nontrivial edge contractions}
		The $n$-Massey product $\langle \alpha_1, \ldots, \alpha_n \rangle$ is non-trivial.
	\end{prop}
	\begin{proof}
		Suppose that $\varphi\colon \mathcal{K}\to \mathcal{\hat{K}}$ is the contraction of just one edge $\{u,v\}\in \mathcal{K}$. By Construction~\ref{cons: edge contraction}, $\{u,v\}\subset J_i$ for  $i\in \{1, \ldots, n\}$.
		
		For $a_{i,i}=a_i$ the representative cocycle for $\alpha_i$ as defined in \eqref{eq:cochain a}, let $(a_{i,k})$ be a defining system for $\langle \alpha_1, \ldots, \alpha_n \rangle$,
		\[
		a_{i,k}=\sum_{\sigma\in S_{a_{i,k}}} c_{\sigma} \chi_{\sigma} \in C^{p_i+\cdots + p_k}(\mathcal{K}_{J_i\cup \cdots \cup J_k}).
		\]
		We show that any defining system $(a_{i,k})$ corresponds to a defining system $(\hat{a}_{i,k})$ for $\langle \hat{\alpha}_1, \ldots, \hat{\alpha}_n \rangle$ in $H^*(\mathcal{Z}_\mathcal{\hat{K}})$. 
		There are two main stages to this proof. Firstly, for a defining system $(a_{i,k})$ such that for any $\{i,k\}$, $\{u,v\}\notin \sigma$ for any $\sigma\in S_{a_{i,k}}$, we construct a corresponding defining system $(\varphi^*(a_{i,k}))$ for $\langle \hat{\alpha}_1, \ldots, \hat{\alpha}_n \rangle$. 
		Secondly, for any other defining system $(a_{i,k})$, we change $a_{i,k}$  to create a different defining system $(\widetilde{a}_{i,k})$ for $\langle \alpha_1, \ldots, \alpha_n \rangle$ such that the associated cocycles are homologous and for any $\{i,k\}$, $\{u,v\}\notin \sigma$ for any $\sigma\in S_{\widetilde{a}_{i,k}}$. 
		Applying the first step to $(\widetilde{a}_{i,k})$, we have a defining system $(\varphi^*(\widetilde{a}_{i,k}))$ that corresponds to $(a_{i,k})$.

		For this first step, suppose that  for any $\{i,k\}$, $\{u,v\}\notin \sigma$ for any $\sigma\in S_{a_{i,k}}$.
		We define a tool $\varphi^*$, which will only be well-defined for certain specified cochains such as
		$a_{i,k}\in C^{p}(\mathcal{K}_{J_i\cup \cdots \cup J_k})$ or $a_{i,r}a_{r+1,k}\in C^{p}(\mathcal{K}_{J_i\cup \cdots \cup J_k})$
		where $p=p_i+\cdots + p_k$ or $p=p_i+\cdots + p_k+1$ respectively.
		We check three properties of $\varphi^*$ in order to construct a defining system $(\varphi^*(a_{i,k}))$ for $\langle \hat{\alpha}_1, \ldots, \hat{\alpha}_n \rangle$.
		Let $a\in C^{p}(\mathcal{K}_{J_i\cup \cdots \cup J_k})$ be a general cochain such that $\{u,v\}\notin \sigma$ for any $\sigma\in S_{a}$, where either $p=p_i+\cdots + p_k$ or $p=p_i+\cdots + p_k+1$.
		For $J\subset[m]$, let $\hat{J}=\varphi(J)$.
		Define	
		\begin{equation}\label{eq: varphi (edge contraction)} 
			\varphi^*(a) = c_{i,k} \sum_{\hat{\sigma}\in \varphi(S_{a})} c_{\hat{\sigma}} \chi_{\hat{\sigma}} 
			\; \in \; C^{p}(\mathcal{\hat{K}}_{\hat{J}_i\cup \cdots \cup \hat{J}_{k}})
		\end{equation}
		where $c_{\hat{\sigma}}=c_\sigma$ for any $\sigma\in S_a$ such that $\varphi(\sigma)=\hat{\sigma}$,  $c_{i,i}=1$ and 
		\begin{equation*}
			c_{i,k} = (-1)^{(|J_i|-|\hat{J}_i|)p_{i+1} +(|J_i\cup J_{i+1}|-|\hat{J}_i\cup \hat{J}_{i+1}|)p_{i+2} +
				\cdots +(|J_i\cup \cdots \cup J_{k-1}|-|\hat{J}_i\cup \cdots \cup \hat{J}_{k-1}|)p_{k} }.
		\end{equation*}
		
		(\romannumeral 1) 
		First note that for any constant $c'\in \mathbf{k}$ and for 
		$\chi_{\sigma}, \chi_{\tau}$ in $C^p(\mathcal{K}_{J_i\cup \cdots \cup J_k})$ where $p$ is either $p_i+\cdots + p_k$ or $p_i+\cdots + p_k+1$ and $\{u,v\}\notin \sigma, \tau$,
		\begin{equation}\label{eq: varphi is scalar & additive (edge contractions)}
			\begin{gathered}
				\varphi^*(c' c_{\sigma} \chi_{\sigma})=c_{i,k}\;c' c_\sigma \chi_{\varphi(\sigma)}=c'\varphi^*(c_{\sigma} \chi_{\sigma}) \text{ and }  \\
				\varphi^*(c_{\sigma} \chi_{\sigma}+c_{\tau} \chi_{\tau})=c_{i,k}\; (c_\sigma \chi_{\varphi(\sigma)}+ c_\tau \chi_{\varphi(\tau)})=\varphi^*(c_{\sigma} \chi_{\sigma})+\varphi^*(c_{\tau} \chi_{\tau}).
			\end{gathered}
		\end{equation}
		
		(\romannumeral 2) 
		Next we show that $\varphi^*(d(a_{i,k}))=d(\varphi^*(a_{i,k}))$.
		Suppose that for a simplex $\sigma\in S_{a_{i,k}}$, there is a simplex $j\cup \sigma \in \mathcal{K}_{J_i\cup \cdots \cup J_k}$ for $j\in J_i\cup \cdots \cup J_k\setminus \sigma$ that is contracted. That is, $\{u, v\}\in j\cup \sigma$.
		By the definition of a defining system, $d(a_{i,k})=\sum_{r=i}^{k-1} \doubleoverline{a_{i,r}}a_{r,k}$. Therefore either $c_\sigma \; \varepsilon(j, j\cup \sigma) \; \chi_{j\cup \sigma}$ is cancelled by other terms in $d(a_{i,k})$, or there exists $i\leqslant r < k$ and simplices $\tau \in S_{a_{i,r}}$, $\eta \in S_{a_{r+1, k}}$ such that $\tau\cup \eta = j \cup \sigma$.
		In the latter case, if $\{u,v\} \in j\cup \sigma$, then $\{u,v\} \in \tau\cup \eta$. 
		This implies that either $\{u,v\}\in \tau$ or $\{u,v\}\in \eta$, since  by construction
		$\{u,v\}\subset J_i$ for an $1 \leqslant i\leqslant n$ and $\tau\in S_{a_{i,r}}\subset J_i\cup \cdots \cup J_r$, $\eta \in S_{a_{r+1,k}}\subset J_{r+1}\cup \cdots \cup J_k$.
		This then contradicts the assumption that $\{u,v\}\notin \sigma$ for any $\sigma\in S_{a_{i,k}}$ and any $\{i,k\}$. 
		Hence a summand of the form $c_\sigma \; \varepsilon(j, j\cup \sigma) \; \chi_{j\cup \sigma}$, where $\{u,v\}\in j\cup\sigma$, is cancelled out by other summands.
		
		Let  $a=\sum_{\sigma \in S_a} c_{\sigma} \chi_{\sigma}\in C^{p_i+\cdots+p_k}(\mathcal{K}_{J_i\cup \cdots \cup J_k})$ be a cochain such that for any simplex $j\cup \sigma\in \mathcal{K}_{J_i\cup \cdots \cup J_k}$ for $\sigma\in S_{a}$ and  $j\in J_i\cup \cdots \cup J_k\setminus \sigma$, either $c_\sigma \; \varepsilon(j, j\cup \sigma) \; \chi_{j\cup \sigma}$ is cancelled by other terms in $d(a_{})$ or $j\cup \sigma$ does not contract.
		Applying $\varphi^*$ to
		\[
		d(a)=\sum_{\sigma\in S_{a}}
		\sum\limits_{ \substack{j\in J_i\cup \cdots \cup J_k\setminus \sigma, \\j\cup \sigma \in \mathcal{K}_{J_i\cup \cdots \cup J_k}}}
		c_\sigma \; \varepsilon(j, j\cup \sigma)\; \chi_{j\cup \sigma},
		\]
		we write
		\[
		\varphi^*(d(a))=
		c_{i,k}\sum_{\hat{\sigma}\in \varphi(S_{a})}
		\; c_{\hat{\sigma}}
		\mathopen{\raisebox{-1.8ex}{$\biggl($}}\, 
		\;
		\sum\limits_{ \substack{\hat{j}\in \hat{J}_i\cup \cdots \cup \hat{J}_k\setminus \hat{\sigma}, \\\hat{j}\cup \hat{\sigma} \in \mathcal{\hat{K}}_{\hat{J}_i\cup \cdots \cup \hat{J}_k}}}
		\; \varepsilon(\hat{j}, \hat{j}\cup \hat{\sigma})\; \chi_{\hat{j}\cup \hat{\sigma}}  
		\mathclose{\raisebox{-1.8ex}{$\biggr)$}} 
		\]
		where $\varepsilon(j, j\cup \sigma)=\varepsilon(\hat{j}, \hat{j}\cup \hat{\sigma})$ due to the order on vertices in $\mathcal{K}$ and since $j\cup \sigma$ does not contract.
		Let $\hat{S}=\{ \varphi(\sigma) \; | \; \sigma\in S_a, |\varphi(\sigma)|=p_i+\cdots + p_k+1 \}$ and 
		let $b=\sum_{\hat{\sigma} \in \hat{S}} c_{\hat{\sigma}} \chi_{\hat{\sigma}} \in C^{p_i+\cdots + p_k}(\mathcal{\hat{K}}_{\hat{J}_i\cup \cdots \cup \hat{J}_k})$.
		Then 
		\begin{equation}\label{eq: varphi respects differential (edge contractions)} 
			\varphi^*(d(a))=d(b).
		\end{equation}
		In particular, $\varphi^*(d(a_{i,k}))=d(\varphi^*(a_{i,k}))$.
		
		(\romannumeral 3) 
		We also show that 
		\begin{equation} \label{eq: varphi of defining system - edge contractions}
			\sum_{r=i}^{k-1}  \;\doubleoverline{\varphi^*(a_{i,r})} \varphi^*(a_{r+1,k}) 
			= \varphi^*\left(\sum_{r=i}^{k-1}\;\doubleoverline{a_{i,r}} a_{r+1,k}\right).
		\end{equation}
		Let $a_{i,r}\in C^{p_i+\cdots + p_r}(\mathcal{K}_{J_i\cup \cdots \cup J_r})$ and $a_{r+1,k}\in C^{p_{r+1}+\cdots + p_k}(\mathcal{K}_{J_{r+1}\cup \cdots \cup J_k})$ be represented by $\sum_{\tau\in S_{a_{i,r}}} c_{\tau} \chi_{\tau}$ and $\sum_{\eta\in S_{a_{r+1,k}}} c_{\eta} \chi_{\eta}$ respectively. 
		The left hand side of \eqref{eq: varphi of defining system - edge contractions} is
		\begin{align*}
			\sum_{r=i}^{k-1} & \;\doubleoverline{\varphi^*(a_{i,r})} \varphi^*(a_{r+1,k}) \\
			&=\sum_{r=i}^{k-1} (-1)^{1+\mydeg \varphi^*(a_{i,r})} \left( c_{i,r} \sum_{\hat{\tau}\in \varphi(S_{a_{i,r}})} c_{\tau} \chi_{\hat{\tau}} \right)\cdot 
			\left( c_{r+1,k} \sum_{\hat{\eta}\in \varphi(S_{a_{r+1,k}})} c_{\eta} \chi_{\hat{\eta}} \right)\\
			&= \sum_{r=i}^{k-1} \;C
			\left(\sum_{\hat{\tau}\in \varphi(S_{a_{i,r}})} \sum_{\hat{\eta}\in \varphi(S_{a_{r+1,k}})} c_{\tau} c_{\eta} \chi_{\hat{\tau}\cup \hat{\eta}}\right)
		\end{align*}
		where 
		\begin{equation*}
			C=(-1)^{1+\mydeg \varphi^*(a_{i,r})+|\hat{J}_i \cup \cdots \cup \hat{J}_r|(p_{r+1}+\cdots +p_k+1)} c_{i,r} c_{r+1,k}.
		\end{equation*}
		Using the expressions for $c_{i,r}$ and $c_{r+1,k}$, and using $\mydeg \varphi^*(a_{i,r})=1+p_i+\cdots +p_r+|\hat{J}_i\cup \cdots \cup \hat{J}_r|$,  
		\begin{align*}
			C= &(-1)^{p_i+\cdots +p_r+|\hat{J}_i\cup \cdots \cup \hat{J}_r|+|\hat{J}_i \cup \cdots \cup \hat{J}_r|(p_{r+1}+\cdots +p_k+1)} \\
			&\quad \cdot (-1)^{(|J_i|-|\hat{J}_i|)p_{i+1} +
				\cdots +(|J_i\cup \cdots \cup J_{r-1}|-|\hat{J}_i\cup \cdots \cup \hat{J}_{r-1}|)p_{r}} \\
			&\quad \cdot (-1)^{(|J_{r+1}|-|\hat{J}_{r+1}|)p_{r+2} +
				\cdots +(|J_{r+1}\cup \cdots \cup J_{k-1}|-|\hat{J}_{r+1}\cup \cdots \cup \hat{J}_{k-1}|)p_{k}}\\
			=& (-1)^{p_i+\cdots +p_r+(|J_i|-|\hat{J}_i|)p_{i+1} +
				\cdots +(|J_i\cup \cdots \cup J_{r-1}|-|\hat{J}_i\cup \cdots \cup \hat{J}_{r-1}|)p_{r}+|\hat{J}_i\cup \cdots \cup \hat{J}_r|p_{r+1}}\\
			&\quad \cdot (-1)^{(|J_{r+1}|-|\hat{J}_i\cup \cdots \cup \hat{J}_{r+1}|)p_{r+2} +
				\cdots +(|J_{r+1}\cup \cdots \cup J_{k-1}|-|\hat{J}_i\cup \cdots \cup \hat{J}_{k-1}|)p_{k}} \\
			=& (-1)^{1+\mydeg a_{i,r}} (-1)^{|J_i\cup \cdots \cup J_r|(p_{r+1}+\cdot + p_k+1)} c_{i,k}.
		\end{align*}
		By assumption, $\{u,v\}\notin \sigma$ for any $\sigma\in S_{a_{i,k}}$ and any $\{i,k\}$.
		Thus $\{u,v\}\notin \tau$ and $\{u,v\}\notin \eta$ for any $i\leqslant r < k$ and any simplices $\tau \in S_{a_{i,r}}$, $\eta \in S_{a_{r+1, k}}$.
		Also, $\{u,v\}\subset J_i$ for an index $1 \leqslant i\leqslant n$,  so $\{u,v\}\notin \tau\cup \eta$.
		Hence $\varphi(\tau\cup \eta)=\varphi(\tau)\cup \varphi(\eta)$ is a $(p_i+\cdots +p_k+1)$-simplex.
		Therefore using the definition of $\varphi^*$, the property (\romannumeral 1), 
		and the fact that $\varphi(\tau\cup \eta)=\varphi(\tau)\cup \varphi(\eta)=\hat{\tau}\cup \hat{\eta}$,
		\begin{multline*}
			\sum_{r=i}^{k-1}  \;\doubleoverline{\varphi^*(a_{i,r})} \varphi^*(a_{r+1,k}) =\\
			= \sum_{r=i}^{k-1} 
			C
			\left(\sum_{\hat{\tau}\in \varphi(S_{a_{i,r}})} \sum_{\hat{\eta}\in \varphi(S_{a_{r+1,k}})} c_{\tau} c_{\eta} \chi_{\hat{\tau}\cup \hat{\eta}}\right)
			= \varphi^*\left(\sum_{r=i}^{k-1}\;\doubleoverline{a_{i,r}} a_{r+1,k}\right).
		\end{multline*}
		
		Using properties (\romannumeral 1), (\romannumeral 2) and (\romannumeral 3), we prove that a defining system $(a_{i,k})$ for $\langle \alpha_1, \ldots, \alpha_n \rangle$ and its associated cocycle $\omega$ are mapped by $\varphi^*$ onto a defining system for $\langle \hat{\alpha}_1, \ldots, \hat{\alpha}_n \rangle$ and its associated cocycle is $\varphi^*(\omega)$.
		By the definition of $a_i=a_{i,i}$ in \eqref{eq:cochain a}, $\varphi^*(a_{i,i})=\hat{a}_{i,i}=\hat{a}_i$.
		By properties  (\romannumeral 2) and (\romannumeral 3), 
		we see that
		\begin{equation*}
			d(\varphi^*(a_{i,k}))=\varphi^*(d(a_{i,k}))=\varphi^*\left(\sum_{r=i}^{k-1}\;\doubleoverline{a_{i,r}} a_{r+1,k}\right)
			=\sum_{r=i}^{k-1} \;\doubleoverline{\varphi^*(a_{i,r})} \varphi^*(a_{r+1,k}).
		\end{equation*}
		Hence $(\varphi^*(a_{i,k}))$ is a defining system for $\langle \hat{\alpha}_1, \ldots, \hat{\alpha}_n \rangle$ if $(a_{i,k})$ is a defining system such that $\{u,v\}\notin \sigma$ for any $\sigma\in S_{a_{i,k}}$ and any pair $\{i,k\}$. 
		Also, for the associated cocycle $\omega$ for $(a_{i,k})$,  
		\[
		\varphi^*(\omega)=\varphi^*\left(\sum_{r=1}^{n-1} \;\doubleoverline{a_{1,r}} a_{r+1,n} \right)=\sum_{r=1}^{n-1} \;\doubleoverline{\varphi^*(a_{1,r})} \varphi^*(a_{r+1,n})
		\]
		so $\varphi^*(\omega)$ is the associated cocycle for $(\varphi^*(a_{i,k}))$.

		Lastly we prove that $[\omega]\neq0$.
		If $[\omega]=0$, then there is a cochain $a\in C^{p_1+\cdots+p_n}(\mathcal{K}_{J_1\cup \cdots \cup J_n})$ such that $\omega=d(a)$. 
		Since $\{u,v\}\in J_j$ for some $j\in \{1,\ldots, n\}$ and $\{u,v\}\notin \sigma$ for any $\sigma\in S_{a_{i,k}}$ and any $\{i,k\}$, no simplices in $S_\omega$ contract.
		Thus no simplices in $S_{d(a)}$.
		So by applying $\varphi^*$ and 
		\eqref{eq: varphi respects differential (edge contractions)} from property (\romannumeral 2),  
		$\varphi^*(\omega)=\varphi^*(d(a))=d(b)$ for a cochain $b\in C^{p_i+\cdots + p_k}(\mathcal{\hat{K}}_{\hat{J}_i\cup \cdots \cup \hat{J}_k})$. 
		So $[\varphi^*(\omega)]=0$, which contradicts the non-triviality of $\langle \hat{\alpha}_1, \ldots, \hat{\alpha}_n \rangle$. Therefore $[\omega]\neq0$.
		
		For the second stage of this proof, suppose that $(a_{i,k})$ is a defining system for $\langle \alpha_1, \ldots, \alpha_n \rangle$ such that there is a pair of indices $\{i,k\}$ with $\{u,v\}\in \sigma$ for some $\sigma\in S_{a_{i,k}}$. 
		We will define a new defining system  $(\widetilde{a}_{i,k})$  such that $\{u,v\}\notin \sigma$ for any $\sigma\in S_{\widetilde{a}_{i,k}}$ and such that $[\omega]=[\widetilde{\omega}]$ where $\omega$ and $ \widetilde{\omega}$ are the associated cocycles for $(a_{i,k})$ and $(\widetilde{a}_{i,k})$, respectively.
		
		The cocycle $a_i=a_{i,i}$ as defined in \eqref{eq:cochain a} is such that $\{u,v\}\notin \sigma$ for every $\sigma\in S_{a_{i}}$. Therefore, let $\{i,k\}$ be a pair of indices such that there is a simplex $\sigma\in S_{a_{i,k}}$ with $\{u,v\}\in \sigma$, and for every $i<i''<k''<k$, $\{u,v\}\notin \tau$ for any $\tau\in S_{a_{i'',k''}}$.
		Let $\sigma\in S_{a_{i,k}}$ be a simplex such that $\{u,v\}\in \sigma$, and let $c_\sigma$ be the non-zero coefficient of $\chi_\sigma$ in $a_{i,k}$.
		Then for every pair $\{i',k'\}\subset [n]$, let $c=(-1)^{\mydeg a_{i,k}} c_\sigma \; \varepsilon(u, \sigma)$ and define
		\begin{equation}\label{eq: corresponding aik that doesn't contract}
			\widetilde{a}_{i',k'} = 
			\begin{cases}
				a_{i,k}-c_\sigma \; \varepsilon(u, \sigma)\; d(\chi_{\sigma\setminus u}) &\text{if } i'=i<k=k',  \\
				a_{i',k}+c_\sigma\; \varepsilon(u,\sigma)\;a_{i', i-1}\chi_{\sigma\setminus u} &\text{if } i'<i<k=k',  \\
				a_{i,k'}+c\; \chi_{\sigma\setminus u}a_{k+1,k'} &\text{if }  i'=i<k<k', \\
				a_{i',k'} &\text{if }  i'<i<k<k' \text{ or } i<i'<k'<k
			\end{cases}
		\end{equation}
		where $\chi_{\sigma\setminus u}\in C^{p_i+\cdots p_k-1}(\mathcal{K}_{J_i \cup \cdots \cup J_k})$.
		We show that $(\widetilde{a}_{i',k'})$ is a defining system for $\langle \alpha_1, \ldots, \alpha_n \rangle$.
		Firstly since $k-i>1$, $\widetilde{a}_{i',i'}=a_{i',i'}$ for every $i'\in [n]$. We also need to show that $d(\widetilde{a}_{i',k'})=\sum_{r=i'}^{k'-1} \doubleoverline{\widetilde{a}_{i',r}}\widetilde{a}_{r+1, k'}$ for every $\{i',k'\}$. 
		
		(\romannumeral 1)
		For $i<i'<k'<k$, we have $\widetilde{a}_{i',k'}=a_{i',k'}$ so
		\[
		d(\widetilde{a}_{i',k'})=d(a_{i',k'})=\sum_{r=i'}^{k'-1} \doubleoverline{a_{i',r}}a_{r+1, k'}=\sum_{r=i'}^{k'-1} \doubleoverline{\widetilde{a}_{i',r}}\widetilde{a}_{r+1, k'}.
		\]
		(\romannumeral 2) For $i'=i<k=k'$, 
		\[
		d(\widetilde{a}_{i,k})=d(a_{i,k}-c_\sigma \; \varepsilon(u,\sigma) \; d(\chi_{\sigma\setminus u}))=d(a_{i,k}).
		\]
		Also $d(\chi_{\sigma\setminus u})\in C^{p_i+\cdots p_k}(\mathcal{K}_{J_i \cup \cdots \cup J_k})$
		since $\chi_{\sigma\setminus u}\in C^{p_i+\cdots p_k-1}(\mathcal{K}_{J_i \cup \cdots \cup J_k})$. 
		Hence $\widetilde{a}_{i,k}\in C^{p_i+\cdots p_k}(\mathcal{K}_{J_i \cup \cdots \cup J_k})$ and $\mydeg \widetilde{a}_{i,k}= \mydeg a_{i,k}$. Additionally, 
		\[
		d(\chi_{\sigma\setminus u})=\sum\limits_{\substack{j\in J_i\cup \cdots \cup J_k \setminus (\sigma\setminus u),\\ j\cup \sigma\setminus u \,\in \mathcal{K}_J}} \varepsilon(j, j\cup \sigma\setminus u) \chi_{j\cup \sigma\setminus u}.
		\]
		So $\chi_\sigma$ is the only summand of $d(\chi_{\sigma\setminus u})$ such that $\{u,v\}\in \sigma$. Thus $a_{i,k}-c_\sigma \; \varepsilon(u,\sigma) \; d(\chi_{\sigma\setminus u})$ no longer contains the summand $\chi_\sigma$ and also 
		\[
		|\{\tau \in S_{\widetilde{a}_{i,k}}\suchthat \{u,v\}\in \tau \}|< |\{\tau \in S_{a_{i,k}}\suchthat \{u,v\}\in \tau \}|.
		\]
		
		(\romannumeral 3)
		Next, for $i'<i<k=k'$, we have $a_{i',i-1}\in C^{p_{i'}+\cdots +p_{i-1}}(\mathcal{K}_{J_{i'}\cup \cdots \cup J_{i-1}})$. So $a_{i', i-1}\chi_{\sigma\setminus u}\in C^{p_{i'}+\cdots +p_{k}}(\mathcal{K}_{J_{i'}\cup \cdots \cup J_{k}})$. Hence $\widetilde{a}_{i',k}\in C^{p_{i'}+\cdots +p_{k}}(\mathcal{K}_{J_{i'}\cup \cdots \cup J_{k}})$. Also,
		\begin{align*}
			d(\widetilde{a}_{i',k})&=d(a_{i',k}+c_\sigma\; \varepsilon(u,\sigma)\;a_{i', i-1}\chi_{\sigma\setminus u})\\
			&\begin{multlined}[c][.9\displaywidth]
			=\sum_{r=i'}^{k-1} \doubleoverline{a_{i',r}}a_{r+1, k} + 
			\\ + c_\sigma\, \varepsilon(u,\sigma) 
			\left(\sum_{r=i'}^{i-2} \doubleoverline{a_{i',r}}a_{r+1, i-1}\right) \chi_{\sigma\setminus u} 
			-c_\sigma\, \varepsilon(u,\sigma)\, \doubleoverline{a_{i', i-1}}d(\chi_{\sigma\setminus u})
			\end{multlined}\\
			&
			=\sum_{r=i'}^{i-2} \doubleoverline{a_{i',r}}(a_{r+1, k}+ c_\sigma\; \varepsilon(u,\sigma)\;a_{r+1, i-1} \chi_{\sigma\setminus u} ) +
			 \\ &\qquad \quad \ 
			 +\doubleoverline{a_{i',i-1}}(a_{i, k}- c_\sigma\; \varepsilon(u,\sigma)\;d(\chi_{\sigma\setminus u} )) +\sum_{r=i}^{k-1} \doubleoverline{a_{i',r}}a_{r+1, k} \\
			&= \sum_{r=i'}^{k-1} \doubleoverline{\widetilde{a}_{i',r}}\widetilde{a}_{r+1, k}.
		\end{align*}
		
		(\romannumeral 4)
		For $i'=i<k<k'$,  we have $\widetilde{a}_{i,k'}\in C^{p_{i}+\cdots +p_{k'}}(\mathcal{K}_{J_{i}\cup \cdots \cup J_{k'}})$ since $\chi_{\sigma\setminus u}a_{k+1,k'}\in C^{p_{i}+\cdots +p_{k'}}(\mathcal{K}_{J_{i}\cup \cdots \cup J_{k'}})$. Furthermore, $d(\widetilde{a}_{i,k'})$ is
		\begin{align*}
			&d(a_{i,k'}+(-1)^{\mydeg a_{i,k}} c_\sigma \; \varepsilon(u, \sigma)\; \chi_{\sigma\setminus u}a_{k+1,k'})\\
			&\ \begin{multlined}[t]
			    = \sum_{r=i}^{k'-1} \doubleoverline{a_{i,r}}a_{r+1, k'} 
			+(-1)^{\mydeg a_{i,k}} c_\sigma \; \varepsilon(u, \sigma)\; d(\chi_{\sigma\setminus u})a_{k+1,k'} \cdot \\ \ 
			 \cdot (-1)^{\mydeg a_{i,k}} c_\sigma \; \varepsilon(u, \sigma)\;  (-1)^{\mydeg \chi_{\sigma\setminus u}} \chi_{\sigma\setminus u} 
			\left( \sum_{r=k+1}^{k'-1} \doubleoverline{a_{k+1,r}}a_{r+1, k'} \right)
			\end{multlined}\\
			&\ \begin{multlined}[t]
			  = \sum_{r=i}^{k-1} \doubleoverline{a_{i,r}}a_{r+1, k} 
			- (-1)^{\mydeg a_{i,k}}(a_{i,k}-c_\sigma\;\varepsilon(u,\sigma)\; d(\chi_{\sigma\setminus u}) ) a_{k+1,k'} +\\
			+\sum_{r=k+1}^{k'-1} \left((-1)^{\mydeg a_{i,k}} c_\sigma \; \varepsilon(u, \sigma)\;  (-1)^{\mydeg \chi_{\sigma\setminus u}}\chi_{\sigma\setminus u} \doubleoverline{a_{k+1,r}}+ \doubleoverline{a_{i,r}}\right) a_{r+1, k'}.
			\end{multlined}
		\end{align*}		
		More specifically, let $c=(-1)^{\mydeg a_{i,k}} c_\sigma \; \varepsilon(u, \sigma)$. Then in the last sum,
		\begin{align*}
			c\;  &(-1)^{\mydeg \chi_{\sigma\setminus u}} \chi_{\sigma\setminus u} \doubleoverline{a_{k+1,r}} \\
			&\qquad= (-1)^{p_i+\cdots +p_k +|J_i\cup \cdots \cup J_k|+p_{k+1}+\cdots +p_r +|J_{k+1}\cup \cdots \cup J_r|} \;c\; \chi_{\sigma\setminus u} a_{k+1,r} \\
			&\qquad= (-1)^{1+\mydeg a_{i,r}} \;c \;\chi_{\sigma\setminus u} a_{k+1,r}.
		\end{align*}
		Therefore 			
		\begin{align*}
			d(\widetilde{a}_{i,k'})&= \sum_{r=i}^{k-1} \doubleoverline{a_{i,r}}a_{r+1, k} 
			+ (-1)^{1+\mydeg a_{i,k}}(a_{i,k}-c_\sigma\;\varepsilon(u,\sigma)\; d(\chi_{\sigma\setminus u}) ) a_{k+1,k'} \\
			&\qquad \qquad  +\sum_{r=k+1}^{k'-1} (-1)^{1+\mydeg a_{i,r}}( \;c \;\chi_{\sigma\setminus u} a_{k+1,r}+ a_{i,r})a_{r+1, k'}\\
			&= \sum_{r=i'}^{k-1} \doubleoverline{\widetilde{a}_{i',r}}\widetilde{a}_{r+1, k}.
		\end{align*}
		
		(\romannumeral 5)
		Lastly when $i'<i<k<k'$, $\widetilde{a}_{i',k'}=a_{i',k'}$ and we want to show that $d(\widetilde{a}_{i',k'})=\sum_{r=i'}^{k'-1}\doubleoverline{\widetilde{a}_{i',r}}\widetilde{a}_{r+1, k'}$.
		The right hand side is
		\begin{align*}
			\sum_{r=i'}^{k'-1} &\doubleoverline{\widetilde{a}_{i',r}}\widetilde{a}_{r+1, k'} 
			=\doubleoverline{a_{i', i-1}} \widetilde{a}_{i,k'} + \doubleoverline{\widetilde{a}_{i', k}} a_{k+1,k'} 
			+ \sum_{r\in\{i', \ldots, \widehat{i-1}, \ldots, \widehat{k}, \ldots k'-1 \}} \doubleoverline{a_{i',r}}a_{r+1, k'}
		\end{align*}
		where $\ \widehat{\ }\ $ denotes omission. 
		By expanding $\widetilde{a}_{i,k'}$, $\widetilde{a}_{i', k}$ and the signs in this expression, $\sum_{r=i'}^{k'-1} \doubleoverline{\widetilde{a}_{i',r}}\widetilde{a}_{r+1, k'}$ is
		\begin{align*}
			&\begin{multlined}[t][0.9\textwidth]
			(-1)^{1+\mydeg a_{i',i-1}}a_{i', i-1}  \left(a_{i,k'}+(-1)^{\mydeg a_{i,k}} c_\sigma \, \varepsilon(u, \sigma)\, \chi_{\sigma\setminus u} a_{k+1,k'}\right) + \\
			 + (-1)^{1+\mydeg a_{i',k}} \left( a_{i',k}+c_\sigma\, \varepsilon(u,\sigma)\,a_{i', i-1}\chi_{\sigma\setminus u} \right) a_{k+1,k'} +\\
			+ \sum_{r\in\{i', \ldots, \widehat{i-1}, \ldots, \widehat{k}, \ldots k'-1 \}} \doubleoverline{a_{i',r}}a_{r+1, k'}
			\end{multlined}
			\\
			& = \begin{multlined}[t][0.9\textwidth]
			\sum_{r=i'}^{k'-1} \doubleoverline{a_{i',r}}a_{r+1, k'} 
			+ \left(  (-1)^{1+\mydeg a_{i',i-1}+\mydeg a_{i,k}} +(-1)^{1+\mydeg a_{i',k}} \right) \cdot\\ \cdot c_\sigma \; \varepsilon(u, \sigma)\;  a_{i', i-1} \chi_{\sigma\setminus u} a_{k+1,k'} 
			\end{multlined}\\
			&=  d(a_{i',k'}) 
			+ \left(  (-1)^{\mydeg a_{i',k}} +(-1)^{1+\mydeg a_{i',k}} \right) c_\sigma \; \varepsilon(u, \sigma)\;  a_{i', i-1} \chi_{\sigma\setminus u} a_{k+1,k'} \\
			&= d(a_{i',k'})=d(\widetilde{a}_{i',k'})
		\end{align*}
		since $\mydeg a_{i',k}=|J_{i'}\cup \cdots \cup J_k| + p_{i'}+\cdots p_{k}+1=\mydeg \widetilde{a}_{i',k}$.
		
		Therefore for all $\{i',k'\}$, $\widetilde{a}_{i',k'}\in C^{p_{i'}+\cdots +p_{k'}} (\mathcal{K}_{J_{i'}\cup \cdots \cup J_{k'}})$
		and  $d(\widetilde{a}_{i',k'})=\sum_{r=i'}^{k'-1} \doubleoverline{\widetilde{a}_{i',r}}\widetilde{a}_{r+1, k'}$. 
		So $(\widetilde{a}_{i',k'})$ is a defining system for $\langle \alpha_1, \ldots, \alpha_n \rangle$. Also $\sigma\notin \tau$ for any $\tau\in S_{\widetilde{a}_{i',k'}}$ and any $\{i',k'\}$.
		The associated cocycle $\widetilde{\omega}$ for this defining system is given by 
		$\sum_{r=1}^{n-1} \doubleoverline{\widetilde{a}_{1,r}}\widetilde{a}_{r+1, n}$. By calculating $\sum_{r=1}^{n-1} \doubleoverline{\widetilde{a}_{1,r}}\widetilde{a}_{r+1, n}$ in a similar manner as in the above calculations,  
		\begin{equation}\label{eq: corresponding omega to tilde aik}
			\widetilde{\omega} = 
			\begin{cases}
				\omega &\text{if } i\neq 1, k\neq n,  \\
				\omega + c_\sigma \; \varepsilon(u, \sigma) \; d(a_{i', i-1}\chi_{\sigma\setminus u}) &\text{if } 1=i<k=n,\\
				\omega - (-1)^{1+\mydeg a_{i,k}} c_\sigma \; \varepsilon(u, \sigma) d(\chi_{\sigma\setminus u}a_{k+1,k'})&\text{if }  1=i<k<n
			\end{cases}
		\end{equation}
		where $\omega$ is the associated cocycle for $(a_{i',k'})$. 
		So $[\widetilde{\omega}]=[\omega]$. Therefore $[\widetilde{\omega}]=0$ if and only if $[\omega]=0$.
		
		If there is cochain $\widetilde{a}_{i',k'}$ in the defining system $(\widetilde{a}_{i,k})$ such that there is a a simplex $\sigma \in S_{\widetilde{a}_{i',k'}}$ with $\{u,v\}\in \sigma$, then we repeat the above procedure to construct  $(\widetilde{\widetilde{a}}_{i',k'})$, etc. 
		After a finite number of iterations, we obtain a defining system $(\widetilde{a}_{i',k'})$ such that for any $\{i',k'\}$ and any simplex $\sigma\in S_{\widetilde{a}_{i',k'}}$, the edge $\{u,v\}$ is not contained in $\sigma$. 
		Then we can construct a defining system $(\varphi^*(\widetilde{a}_{i',k'}))$ for $\langle \hat{\alpha}_1, \ldots, \hat{\alpha}_n \rangle$. 
		Let $\omega$ and $\widetilde{\omega}$ be the associated cocycles for $(a_{i,k})$ and $(\widetilde{a}_{i,k})$, respectively.
		If $[\omega]=[\widetilde{\omega}]=0$, then $[\varphi^*(\widetilde{\omega})]=0$, which contradicts the assumption that $\langle \hat{\alpha}_1, \ldots, \hat{\alpha}_n \rangle$ is non-trivial.
		Hence if $\langle \hat{\alpha}_1, \ldots, \hat{\alpha}_n \rangle$ is non-trivial, then $\langle \alpha_1, \ldots, \alpha_n \rangle$ is non-trivial.
		
		If $\mathcal{K}\to \mathcal{\hat{K}}$ by a series of more than one edge contractions, we repeat the steps in this proof for each edge contraction in turn.
	\end{proof}
	
	Putting together Proposition~\ref{lem: d(a_ik)=a_i,r a_r,k contracting} 
	and Proposition~\ref{prop: Massey nontrivial edge contractions}, 
	we have proved the following statement.
	
	\begin{theorem}\label{thm: edge contractions Massey}
		Let $\mathcal{\hat{K}}$ be a simplicial complex with a non-trivial $n$-Massey product in $H^*(\mathcal{Z}_\mathcal{\hat{K}})$. 
		Let $\mathcal{K}$ be a simplicial complex that maps onto $\mathcal{\hat{K}}$ by a series of edge contractions $\varphi\colon \mathcal{K} \to \mathcal{\hat{K}}$ that satisfy the link condition.
		Then there is a non-trivial $n$-Massey product in $H^*(\mathcal{Z}_\mathcal{K})$. \qed
	\end{theorem} 
	

	By construction,  $\alpha_i \in H^{|J_i|+p_i+1}(\mathcal{Z}_\mathcal{K})$ and $\hat{\alpha}_i \in H^{|\hat{J}_i|+p_i+1}(\mathcal{Z}_\mathcal{\hat{K}})$ with $|J_i|\geqslant |\hat{J}_i|$ for each $i$.
	Hence the degree of $\langle \alpha_1, \ldots, \alpha_n \rangle \subset H^{|J_1 \cup \cdots \cup J_n|+(p_1+\cdots +p_n)+2}(\mathcal{Z}_\mathcal{K})$ is greater than the degree of
	$\langle \hat{\alpha}_1, \ldots,  \hat{\alpha}_n \rangle \subset H^{|\hat{J}_1 \cup \cdots \cup \hat{J}_n|+(p_1+\cdots+p_n+1)+1}(\mathcal{Z}_\mathcal{\hat{K}})$.
	Also, if $\langle \hat{\alpha}_1, \ldots,  \hat{\alpha}_n \rangle$ has non-trivial indeterminacy, then $\langle \alpha_1, \ldots, \alpha_n \rangle$ also has non-trivial indeterminacy.
	As noted earlier, the converse does not necessarily hold: the pullback Massey product in $H^*(\mathcal{Z}_\mathcal{K})$ might have non-trivial indeterminacy even if it is a pullback of a uniquely defined Massey product in $H^*(\mathcal{Z}_\mathcal{\hat{K}})$.
	
	\begin{figure}[ht]
		\centering
		\begin{minipage}[c]{0.45\textwidth}
			\centering
			\begin{tikzpicture}	[scale=0.60, inner sep=2mm]
			\coordinate (c) at (0.3,-0.3);
			\coordinate (i1) at (-0.7, 0.7);
			\coordinate (i2) at (1.3, 0.7);
			\coordinate (i3) at (1.3, -1.3);
			\coordinate (i4) at (-0.7, -1.3);
			\coordinate (m1) at (0.3, 1.7);
			\coordinate (m2) at (2, -0.3);
			\coordinate (m3) at (0.3, -2);
			\coordinate (m4) at (-1.7, -0.3);
			\coordinate (o1) at (-2.7, 2.7);
			\coordinate (o2) at (2.7, 2.7);
			\coordinate (o3) at (2.7, -2.7);
			\coordinate (o4) at (-2.7, -2.7);
			
			\fill[lightgray, fill opacity=0.6](o4)--(o1) --(o2)--(i2) --(i1)--(m4)--(o4);
			\fill[lightgray, fill opacity=0.4](o4)--(o1) --(o2)--(o3) --(o4);
			
			\draw[dashed] (i1) --(i2)--(i3) --(i4)--(i1);
			\draw[dashed] (o1) --(o2)--(o3) --(o4)--(o1);
			\draw[dashed] (i1)--(m1) --(i2)--(m2)--(i3)--(m3) --(i4)-- (m4)--(i1);
			\draw[dashed] (o1)--(m1) --(o2)--(m2)--(o3)--(m3) --(o4) --(m4)--(o1);
			\draw[dashed] (o1)--(i1)--(c)--(i3)--(o3);
			\draw[dashed] (o2)--(i2)--(c)--(i4)--(o4);
			
			\draw[line width=1pt] (o4)--(m4)--(i1)--(i2)--(i3)--(o3)--(o2)--(i2)--(m1)--(o2)--(o1)--(m1)--(i1)--(o1)--(m4)--(o4)--(o1);
			\draw[line width=1pt] (o4)--(o3);
			\draw[line width=5.5pt] (o1)--(m1);
			\draw[line width=5.5pt] (i1)--(m4);
			\draw[line width=5.5pt] (i1)--(i2);
			\draw[line width=4pt, colour1] (o1)--(m1);
			\draw[line width=4pt, colour2] (i1)--(m4);
			\draw[line width=4pt, colour2] (i1)--(i2);
			
			\foreach \i in {i1, i2, i3, m1, m4, o1, o2, o3, o4} {\fill (\i) circle (6.5pt);}
			\foreach \i in {i3, m1, o1} {\fill[colour1] (\i) circle (5pt);}
			\foreach \i in {i1, i2, m4, o3} {\fill[colour2] (\i) circle (5pt);}
			\foreach \i in {o2,  o4} {\fill[colour3] (\i) circle (5pt);}
			\foreach \i in {c,  i4, m2, m3} {\fill (\i) circle (3pt);}
			
			\node[left] at (o4) {1};
			\node[right] at (o2) {2};
			\node[below]  at (i3) {3};
			\node[left] at  (o1) {4};
			\node[above] at  (m1) {5};
			\node[right] at  (o3) {6};
			\node[right] at (m4) {7};
			\node[left] at  (i1) {8};
			\node[right] at (i2) {9};
			\end{tikzpicture}
			\subcaption{A full subcomplex $\mathcal{K}\subset \mathcal{K}_P$, when $P$ is a truncated octahedron} 
			\label{fig: (square) truncated octahedron Massey}
		\end{minipage}\quad
		\begin{minipage}[c]{0.45\textwidth}
			\centering
			\begin{tikzpicture} [scale=1.2, inner sep=2mm]
			\coordinate (a) at (0,0); 
			\coordinate (b) at (1,0); 
			\coordinate (c) at (0,1); 
			\coordinate (d) at (1,1); 
			\coordinate (e) at (0,2); 
			\coordinate (f) at (-1,0); 
			
			
			\fill[lightgray, fill opacity=0.4](f)to [out=90,in=-135] (e) --(d)--(c)--(f);
			
			\draw[line width=1pt] (a) -- (b) -- (d) --(e) -- (c) -- (a); 
			\draw[line width=1pt] (f) to [out=90,in=-135] (e); 
			\draw[line width=1pt] (f) to [out=-40,in=-140] (b); 
			\draw[line width=1pt] (f) -- (c) -- (d); 
			
			\foreach \i in {a, b, c, d, e, f} {\fill (\i) circle (2.5pt);} 
			\foreach \i in {a, e} {\fill[colour1] (\i) circle (1.9pt);} 
			\foreach \i in {b, c} {\fill[colour2] (\i) circle (1.9pt);} 
			\foreach \i in {d, f} {\fill[colour3] (\i) circle (1.9pt);} 
			
			\node[left] at (a) {$\hat{3}$};
			\draw (b) node[right] {$\hat{6}$};
			\draw (c) node[left] {$\hat{5}$};
			\draw (d) node[right] {$\hat{2}$};
			\draw (e) node[left] {$\hat{4}$};
			\draw (f) node[left] {$\hat{1}$};
			\end{tikzpicture} 
			\subcaption{A simplicial complex $\mathcal{\hat{K}}$ such that $\mathcal{Z}_\mathcal{\hat{K}}$ has non-trivial triple Massey product with indeterminacy.} 
			\label{fig: simplest non-trivial indeterminacy example}
		\end{minipage}
		\caption{Edge contraction example}
	\end{figure}
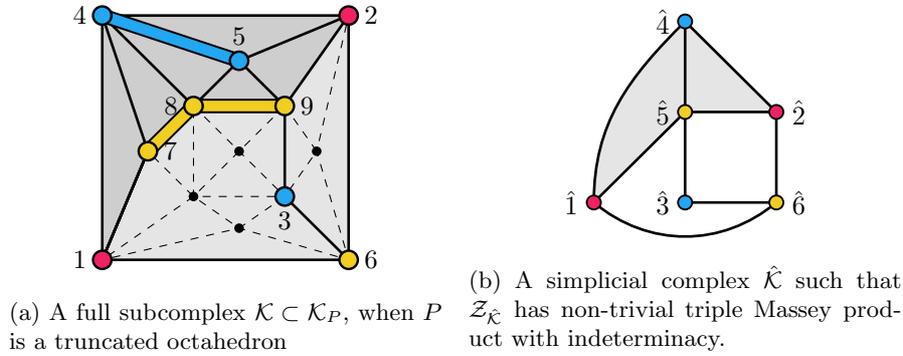
	
	\begin{exmp} \label{ex: truncated octahedron with indeterminacy}
		Let $\mathcal{\hat{K}}$ be the simplicial complex in Figure~\ref{fig: simplest non-trivial indeterminacy example}.
		Since the $1$-skeleton of $\mathcal{\hat{K}}$ is one of the obstruction graphs in the classification of lowest degree non-trivial triple Massey products \cite{LowestDegreeClassification}, there is a non-trivial triple Massey product $\langle \hat{\alpha}_1, \hat{\alpha}_2, \hat{\alpha}_3 \rangle \subset H^8(\mathcal{Z}_{\mathcal{\hat{K}}})$ where $\hat{\alpha}_1\in \widetilde{H}^0(\mathcal{\hat{K}}_{\hat{1} \hat{2}})$, $\hat{\alpha}_2\in \widetilde{H}^0(\mathcal{\hat{K}}_{\hat{3} \hat{4}})$ and $\hat{\alpha}_3\in \widetilde{H}^0(\mathcal{\hat{K}}_{\hat{5} \hat{6}})$.
		This Massey product has non-trivial indeterminacy, since the indeterminacy of this triple Massey product is given by $\hat{\alpha}_1 \cdot \widetilde{H}^0(\mathcal{\hat{K}}_{\hat{3} \hat{4} \hat{5} \hat{6}})+ \hat{\alpha}_3 \cdot \widetilde{H}^0(\mathcal{\hat{K}}_{\hat{1} \hat{2} \hat{3} \hat{4}})=\hat{\alpha}_3 \cdot \widetilde{H}^0(\mathcal{\hat{K}}_{\hat{1} \hat{2} \hat{3} \hat{4}})$.
		
		Let $\mathcal{K}$ be the simplicial complex on $9$ vertices in Figure~\ref{fig: (square) truncated octahedron Massey}. Let $\varphi\colon \mathcal{K}\to \mathcal{\hat{K}}$ be the simplicial map that takes $i\mapsto \hat{i}$ for $i=1,2,3,6$ and contracts the bold coloured edges $\{4,5\}\mapsto \hat{4}$, $\{7,8\}, \{8,9\}\mapsto \hat{5}$.
		By Theorem~\ref{thm: edge contractions Massey} and Construction~\ref{cons: edge contraction}, there is a non-trivial Massey product $\langle \alpha_1, \alpha_2, \alpha_3 \rangle\subset H^{11}(\mathcal{Z}_\mathcal{K})$ where $\alpha_1 \in \widetilde{H}^0(\mathcal{K}_{12})$, $\alpha_2 \in \widetilde{H}^0(\mathcal{K}_{345})$ and $\alpha_3 \in \widetilde{H}^0(\mathcal{K}_{6789})$.
		Also the indeterminacy of this Massey product is non-trivial since it is given by $\alpha_1\cdot \widetilde{H}^0(\mathcal{K}_{3456789})+\alpha_3 \cdot \widetilde{H}^0(\mathcal{K}_{12345})=\alpha_3 \cdot \widetilde{H}^0(\mathcal{K}_{12345})$.
		
		For any simple polytope $P$, define $\mathcal{K}_P=\partial(P^*)$ to be the boundary of the dual polytope.
		This is a simplicial complex and the moment-angle complex $\mathcal{Z}_P=\mathcal{Z}_{\mathcal{K}_P}$ is a moment-angle manifold.
		The simplicial complex $\mathcal{K}$ in Figure~\ref{fig: (square) truncated octahedron Massey} is a full-subcomplex of $\mathcal{K}_P$ when $P$ is a truncated octahedron, otherwise known as the $3$-dimensional permutahedron. 
		A truncated octahedron is a $3$-dimensional simple polytope whose facets are $6$ squares and $8$ hexagons, so there are $6$ vertices of $\mathcal{K}_P$ with valency $4$ and $8$ with valency $6$.
		Since $\mathcal{K}\subset \mathcal{K}_P$, the non-trivial Massey product in $H^*(\mathcal{Z}_\mathcal{K})$ lifts to a non-trivial Massey product in $H^*(\mathcal{Z}_P)$ with non-trivial indeterminacy.
		Hence we found a non-trivial Massey product in $H^*(\mathcal{Z}_P)$ using only Theorem~\ref{thm: edge contractions Massey} and the classification of lowest-degree non-trivial triple Massey products in \cite{DenhamSuciu, LowestDegreeClassification}.
		This technique also recovers the first example of a triple Massey product in $H^*(\mathcal{Z}_P)$ that was given in \cite[Lemma~4.9(2)]{Limonchenko_multiwedge}, where the constructed full subcomplex edge contracts to one of the obstruction graphs that give trivial indeterminacy.
	\end{exmp}

	\begin{exmp}
		A Pogorelov polytope is a $3$-dimensional polytope that can be realised in hyperbolic (Lobachevsky) space as a bounded right-angled polytope. The Pogorelov class is large and includes all fullerenes, whose facets are pentagons and hexagons.
		Zhuravleva \cite[Theorem~3.2]{Liz} showed that for any Pogorelov polytope $P$, $\mathcal{K}_P=\partial(P^*)$ has a full subcomplex $\mathcal{K}$ as shown in Figure~\ref{fig:Liz2}.
		This full subcomplex was used to explicitly construct a non-trivial Massey product $\langle \alpha_1, \alpha_2, \alpha_3 \rangle \subset H^*(\mathcal{Z}_P)$ where $\alpha_1\in \widetilde{H}^0(\mathcal{K}_{567})$, $\alpha_2\in \widetilde{H}^0(\mathcal{K}_{2 b_0\ldots b_n})$ and $\alpha_3\in \widetilde{H}^0(\mathcal{K}_{34})$.
		moment-angle manifolds $\mathcal{Z}_P$ have a non-trivial triple Massey product using the full subcomplex in Figure~\ref{fig:Liz2}.
		
		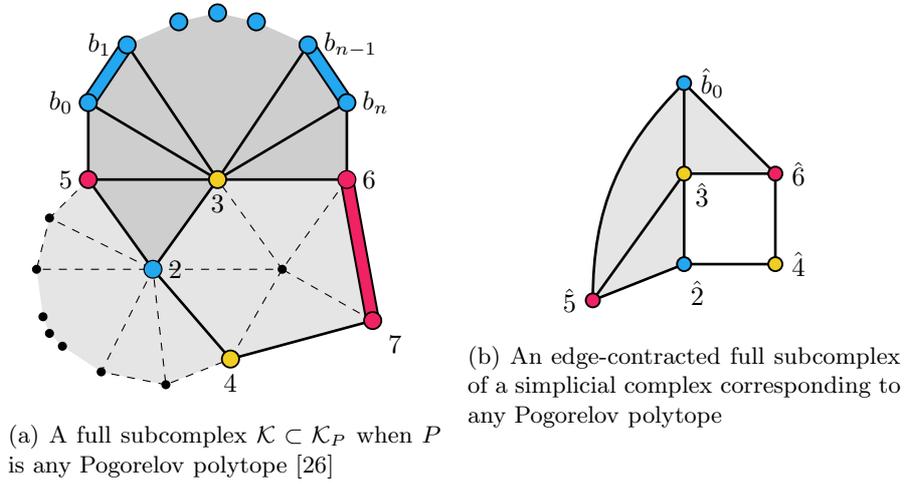
\begin{figure}[ht]
			\centering
			\begin{minipage}[c]{0.45\textwidth}
				\centering
				\begin{tikzpicture}[scale=1.7, inner sep=2mm]
				\coordinate (5) at (0,0);
				\coordinate (3) at (1,0); 
				\coordinate (6) at (2, 0);
				\coordinate (2) at (0.5, -0.7);
				\coordinate (1) at (1.5, -0.7);
				\coordinate (7) at (2.2, -1.1);
				\coordinate (4) at (1.1, -1.4);
				
				\coordinate (b1) at (0, 0.6);
				\coordinate (b2) at (0.3, 1.05);
				\coordinate (b3) at (0.7, 1.23);
				\coordinate (b4) at (1, 1.3);
				\coordinate (b5) at (1.3, 1.23);
				\coordinate (b6) at (1.7, 1.05);
				\coordinate (b7) at (2, 0.6);
				
				\coordinate (01) at (0.6, -1.6);
				\coordinate (02) at (0.1, -1.5);
				\coordinate (03) at (-0.2, -1.3);
				\coordinate	(04) at (-0.3, -1.2);
				\coordinate (05) at (-0.35, -1.07);
				\coordinate (06) at (-0.4, -0.7);
				\coordinate (07) at (-0.3, -0.3);

				\fill[lightgray, fill opacity=0.4] (5)  -- (b1)-- (b2) --(b3)--(b4)--(b5) --(b6)--(b7)--(6)--(7)-- (4)--(01)--(02)--(03)--(04)--(05)--(06)--(07)--(5);
				\fill[lightgray, fill opacity=0.6] (2)--(5)--(b1)-- (b2) --(b3)--(b4)--(b5) --(b6)--(b7)--(6)--(3)--(2);
				
				\draw[line width=1pt] (5) -- (3) -- (6); 
				\draw[line width=1pt] (5) -- (2) -- (3); 
				\draw[line width=1pt] (2) -- (4) -- (7);
				\draw[line width=1pt] (5) -- (b1) -- (3)--(b2); 
				\draw[line width=1pt] (6) -- (b7)--(3)--(b6);
				
				\draw[dashed] (2) -- (1) -- (3); 
				\draw[dashed] (6) -- (1) -- (7); 
				\draw[dashed] (1) -- (4)--(01)--(2)--(02)--(01); \draw[dashed] (5) -- (07) -- (2) -- (06) -- (07);

				\draw[line width=5.5pt] (b1)--(b2);
				\draw[line width=4pt, colour1] (b1)--(b2);
				\draw[line width=5.5pt] (b6)--(b7);
				\draw[line width=4pt, colour1] (b6)--(b7);
				\draw[line width=5.5pt] (6)--(7);
				\draw[line width=4pt, colour3] (6)--(7);
				
				\foreach \i in {2, ..., 7} {\fill (\i) circle (2.1pt);}
				\foreach \i in {1, ..., 7} {\fill (b\i) circle (2.1pt);}
				\foreach \i in {1, ..., 7} {\fill[colour1] (b\i) circle (1.7pt);} 
				\filldraw[colour1](2) circle (1.7pt);
				\fill[colour2] (3) circle (1.7pt); 
				\fill[colour2] (4) circle (1.7pt);
				\fill[colour3] (5) circle (1.7pt); 
				\filldraw[colour3] (6) circle (1.7pt); 
				\fill[colour3] (7) circle (1.7pt);
				
				\fill (1) circle (1pt);
				\foreach \i in {1, ..., 7} {\fill (0\i) circle (1pt);}
				
				\draw (2)node[right] {$2$};\draw (3)node[below] {$3$};\draw (4)node[below] {$4$};\draw (5)node[left] {$5$};\draw (6)node[right] {$6$};\draw (7)node[below right] {$7$};
				\draw (b1) node[left] {$b_0$};\draw (b2) node[left] {$b_1$};\draw (b6) node[right] {$b_{n-1}$};\draw (b7) node[right] {$b_n$};
				
				\end{tikzpicture} 
				\subcaption{A full subcomplex $\mathcal{K}\subset \mathcal{K}_P$ when $P$ is any Pogorelov polytope \cite{Liz}} \label{fig:Liz2}
			\end{minipage}\quad
			\begin{minipage}[c]{0.45\textwidth}
				\centering
				\begin{tikzpicture}[scale=1.2, inner sep=2mm]
				\coordinate (b1) at (0,0); 
				\coordinate (6) at (1,0); 
				\coordinate (5) at (0,1); 
				\coordinate (4) at (1,1); 
				\coordinate (2) at (0,2); 
				\coordinate (3) at (-1,-0.4); 
				
				
				
				\fill[lightgray, fill opacity=0.4] (3)  to [out=90,in=-135] (2) --(5) --(b1); 
				\fill[lightgray, fill opacity=0.4] (2)--(5)--(4);
				
				\draw[line width=1pt] (3)node[left] {$\hat{5}$} -- (b1) node[below] {\ \ \ $\hat{2}$} -- (6)node[right] {$\hat{4}$} -- (4) node[right] {$\hat{6}$} --(2) node[right] {$\hat{b}_0$} -- (5) node[inner sep=1mm, below] {\ \ \ \ $\hat{3}$}-- (b1);
				\draw[line width=1pt] (3)  to [out=90,in=-135] (2);
				\draw[line width=1pt] (3) -- (5)--(4);
				
				\foreach \i in {2,3, 4, 5, 6} {\fill (\i) circle (2.5pt);}
				\fill (b1) circle (2.5pt);
				\fill[colour1] (b1) circle (1.9pt);
				\fill[colour2] (6) circle (1.9pt);
				\fill[colour2] (5) circle (1.9pt);
				\fill[colour3] (4) circle (1.9pt);
				\fill[colour1] (2) circle (1.9pt);
				\fill[colour3] (3) circle (1.9pt);
				\end{tikzpicture} 
				\subcaption{An edge-contracted full subcomplex of a simplicial complex corresponding to any Pogorelov polytope}
				\label{fig: edge contracted Liz's}
			\end{minipage}
			\caption{Massey products in Pogorelov polytopes}
		\end{figure}
		
		Edge contracting the coloured edges of $\mathcal{K}$, $\{b_i,b_{i+1}\}\mapsto \hat{b}_0$, $\{6,7\}\mapsto \hat{6}$,  we obtain the simplicial complex in Figure~\ref{fig: edge contracted Liz's}. 
		This simplicial complex has a non-trivial triple Massey product, since its $1$-skeleton is one of the obstruction graphs from the classification in \cite{DenhamSuciu, LowestDegreeClassification}.
		Since the edge contractions satisfy the link condition, Theorem~\ref{thm: edge contractions Massey} gives an alternative proof of non-trivial triple Massey products in Zhuravleva's work.
	\end{exmp}
	
	
	\subsection{Massey products constructed by edge stretching}

	For an edge contraction $\mathcal{K}\mapsto \mathcal{\hat{K}}$ that satisfies the link condition, we call the inverse $\mathcal{\hat{K}}\mapsto \mathcal{K}$ \textit{edge stretching}.
	\begin{cor}
		Let $\mathcal{\hat{K}}$ be a simplicial complex with a non-trivial $n$-Massey product $\langle \hat{\alpha}_1, \ldots, \hat{\alpha}_n \rangle \subset H^*(\mathcal{Z}_\mathcal{\hat{K}})$.
		Suppose that $\psi\colon \mathcal{\hat{K}}\to \mathcal{K}$ is a series of edge stretchings. 
		Then there is a non-trivial $n$-Massey product in $H^*(\mathcal{Z}_\mathcal{K})$. 
	\end{cor}
	\begin{proof}
		Since $\psi\colon \mathcal{\hat{K}}\to \mathcal{K}$ is a series of edge stretchings, there is a series of edge contractions $\varphi\colon \mathcal{K} \to \mathcal{\hat{K}}$. 
		Given $\langle \hat{\alpha}_1, \ldots, \hat{\alpha}_n \rangle$ in $H^*(\mathcal{Z}_\mathcal{\hat{K}})$, there is a non-trivial $n$-Massey product $\langle \alpha_1, \ldots, \alpha_n \rangle \subset H^*(\mathcal{Z}_\mathcal{K})$  by Theorem~\ref{thm: edge contractions Massey}. 
	\end{proof}

	We may use edge stretchings to build infinite families of examples of Massey products in moment-angle complexes given any known Massey product in a moment-angle complex. 
	For example we can start with one of the obstruction graphs for lowest-degree triple Massey products \cite{DenhamSuciu, LowestDegreeClassification} and produce infinite families of simplicial complexes that contain non-trivial triple Massey products of classes on different degrees. 
	This illustrates that non-trivial Massey products are very common in moment-angle complexes, contrary to previous belief.
	

	\section{Non-trivial Massey products in nestohedra}
	Theorems~\ref{thm: joins} and \ref{thm: edge contractions Massey} can be applied together to construct non-trivial higher Massey products of classes in various degrees in the cohomology of moment-angle complexes.
	Recall that for any simple polytope $P$, there is a simplicial complex $\mathcal{K}_P=\partial(P^*)$ and $\mathcal{Z}_P=\mathcal{Z}_{\mathcal{K}_P}$ is a moment-angle manifold.
	In this section we show that there are families of polytopes $P$ for which $H^*(\mathcal{Z}_P)$ has non-trivial higher Massey products. 
	
	Nestohedra are a large family of simple polytopes built out of Minkowski sums of simplices, introduced by Feichtner and Sturmfels \cite{nestohedra}.
	They include all simplices, permutahedra, Stasheff's associahedra and more generally Carr and Devadoss' graph associahedra \cite{graph_associahedra}.
	Alternatively nestohedra are interpreted as hypergraph polytopes~\cite{hypergraph_polytopes}.
	The first examples of Massey products in moment-angle manifolds associated to nestohedra were in \cite[Proposition~4.1]{LimonchenkoFlagNestohedra} and \cite[Lemma~4.9]{Limonchenko_multiwedge} and were triple Massey products constructed either by explicit calculation or using the classification of lowest degree Massey products~\cite{DenhamSuciu, LowestDegreeClassification}.
	We will use Theorems~\ref{thm: joins} and \ref{thm: edge contractions Massey} to construct families of new non-trivial higher Massey products in moment-angle manifolds associated to certain nestohedra. We use a construction of nestohedra due to Postnikov~\cite[Theorem~7.4]{postnikovnestohedra}.

	\begin{defn}
		A \textit{building set} $B$ is a collection of non-empty subsets of $[n+1]$ such that 
		\begin{enumerate}
			\item $\{i\}\in B$ for every $i\in [n+1]$,
			\item $S_1\cup S_2\in B$ for any $S_1, S_2\in B$ with $S_1\cap S_2\neq \myempty$.
		\end{enumerate}
	\end{defn}
	
	A convex polytope is the convex hull of a finite number of points in $\mathbb{R}^n$.
	If $M_1$ and $M_2$ are convex polytopes in $\mathbb{R}^n$, then the Minkowski sum
	\[
	M_1+M_2=\{x\in \mathbb{R}^n \suchthat x=x_1+x_2, x_1\in M_1, x_2\in M_2\}.
	\]
	is also a convex polytope.
	
	
	
	\begin{defn}
		For a building set $B\subset [n+1]$, a \textit{nestohedron} $P_B$ is the polytope $\sum_{S\in B} \Delta^S$, where $\Delta^S=\text{conv}\{e_i, i\in S\}$ is the convex hull of the basis elements $e_i\in \mathbb{R}^{n+1}$.
	\end{defn}
	
	For example, the $n$-simplex is a nestohedron with building set $\{ \{1\}, \ldots, \{n+1\}, [n+1] \}$.
	Other key examples of nestohedra are \textit{graph associahedra} $P_{B_\Gamma}$, which are associated to a graph $\Gamma$ on the vertex set $[n+1]$. The \textit{graphical building set} $B_\Gamma$ consists of subsets $S\subset [n+1]$ such that the restriction of $\Gamma$ to the vertices in $S$ is a connected graph.
	
	Since every nestohedron $P_B$ is simple~\cite{nestohedra, postnikovnestohedra}, we will consider the corresponding simplicial complex $\mathcal{K}_{P_B}=\partial (P_B^*)$, which is the boundary of the dual polytope.
	Let $B_{\max}$ be the set of maximal sets in $B$ with respect to inclusion.
	\begin{prop}[\cite{postnikovnestohedra}] \label{prop: nested set complex} 
		The simplicial complex $\mathcal{K}_{P_B}$ is isomorphic to the \textit{nested set complex} $\mathcal{N}(B)$, which contains a simplex $\{S_1, \ldots, S_k\} \subset B\setminus B_{\max}$ if 
		\begin{enumerate}
			\item for any $S_i, S_j\in \{S_1, \ldots, S_k\}$, either $S_i\subset S_j$, $S_j\subset S_i$ or $S_i\cap S_j=\myempty$,
			\item for any $S_{i_1}, \ldots, S_{i_p}\in \{S_1, \ldots, S_k\}$ with $S_{i_j}\cap S_{i_l}=\myempty$, $S_{i_1}\sqcup \cdots \sqcup S_{i_p} \notin B$.
		\end{enumerate} 
		\qed
	\end{prop}
	
	For example if $P_B$ is the polytopal $n$-simplex, then $\mathcal{K}_{P_B}$ is the boundary of an $n$-simplex.
	Another example is shown in Figure~\ref{fig: K_P when P is 3-permutahedron}.
	We denote the moment-angle complex $\mathcal{Z}_{\mathcal{K}_{P_B}}$ by $\mathcal{Z}_{P_B}$.
	
	

	\subsection{Permutahedra}
	A permutahedron is an example of a graph associahedron, when the associated graph is a complete graph on $n+1$ vertices.
	Limonchenko \cite[Theorem~3]{Limonchenko} showed that the $3$-dimensional permutahedron $P$ has no non-trivial triple Massey product $\langle \alpha_1, \alpha_2, \alpha_3 \rangle$ for three-dimensional classes $\alpha_i\in H^3(\mathcal{Z}_{P})$, using the classification by \cite[Theorem 6.1.1]{DenhamSuciu} and \cite{LowestDegreeClassification}.
	However, there are other non-trivial triple Massey products in $H^*(\mathcal{Z}_{P})$, as illustrated in Example~\ref{ex: truncated octahedron with indeterminacy}.
	Via an explicit example, it was also shown in \cite[Proposition~4.1]{LimonchenkoFlagNestohedra} and \cite[Lemma~4.9]{Limonchenko_multiwedge} that there are triple Massey products of three-dimensional classes in $H^*(\mathcal{Z}_P)$ for $n$ dimensional permutahedra $P$ with $n>3$. 
	Here we will generalise this and show that $\mathcal{Z}_P$, for the $n$-dimensional permutahedron $P$, has a non-trivial $k$-Massey product for $k\leqslant n$.
	
	\begin{figure}[ht]
	\centering
		\begin{tikzpicture}	[scale=0.60, inner sep=1.5mm]
		\coordinate (c) at (0,-0);
		\coordinate (i1) at (-1, 1);
		\coordinate (i2) at (1, 1);
		\coordinate (i3) at (1, -1);
		\coordinate (i4) at (-1, -1);
		\coordinate (m1) at (0, 2);
		\coordinate (m2) at (2, -0);
		\coordinate (m3) at (0, -2);
		\coordinate (m4) at (-2, -0);
		\coordinate (o1) at (-3, 3);
		\coordinate (o2) at (3, 3);
		\coordinate (o3) at (3, -3);
		\coordinate (o4) at (-3, -3);
		
		\fill[lightgray, fill opacity=0.6](o4)--(o1) --(o2)--(o3) --(o4);
		
		\draw[] (i1) --(i2)--(i3) --(i4)--(i1);
		\draw[] (o1) --(o2)--(o3) --(o4)--(o1);
		\draw[] (i1)--(m1) --(i2)--(m2)--(i3)--(m3) --(i4)-- (m4)--(i1);
		\draw[] (o1)--(m1) --(o2)--(m2)--(o3)--(m3) --(o4) --(m4)--(o1);
		\draw[] (o1)--(i1)--(c)--(i3)--(o3);
		\draw[] (o2)--(i2)--(c)--(i4)--(o4);
		
		\draw[] (o4)--(m4)--(i1)--(i2)--(i3)--(o3)--(o2)--(i2)--(m1)--(o2)--(o1)--(m1)--(i1)--(o1)--(m4)--(o4)--(o1);
		\draw[] (o4)--(o3);
		\draw[] (o1)--(m1);
		\draw[] (i1)--(m4);
		\draw[] (i1)--(i2);
		
		\foreach \i in {c, i1, i2, i3, i4, m1, m2, m3, m4, o1, o2, o3, o4} {\fill (\i) circle (4pt);}
		
		\node[left] at (o4) {$v_{234}$};
		\node[right] at (o2) {$v_{124}$};
		\node[below]  at (i3) {$v_{134}$};
		\node[left] at  (o1) {$v_{2}$};
		\node[above] at  (m1) {$v_{12}$};
		\node[right] at  (o3) {$v_{4}$};
		\node[left] at (m4) {$v_{23}$};
		\node[left] at  (i1) {$v_{123}$};
		\node[right] at (i2) {$v_{1}$};
		\node[right] at (c)  {$v_{13}$};
		\node[left] at (i4) {$v_{3}$};
		\node[right] at (m2) {$v_{14}$};
		\node[below] at (m3) {$v_{34}$};
		\end{tikzpicture}
		\caption{The simplicial complex $\mathcal{K}_P$, without the vertex $v_{24}$, when $P$ is the $3$-dimensional permutahedron}
		\label{fig: K_P when P is 3-permutahedron}
	\end{figure}
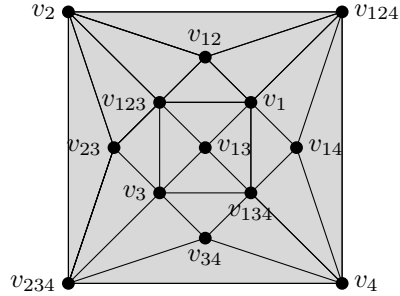

	\begin{prop} \label{prop: permutahedra}
		When $P$ is the $n$-dimensional permutahedron, $H^*(\mathcal{Z}_P)$ has a non-trivial $k$-Massey product for every $k\leqslant n$.
	\end{prop}
	\begin{proof}
		The building set $B$ of the $n$-dimensional permutahedron $P$ contains all possible subsets of $[n+1]$.
		Let $v_{S}$ be the vertex in $\mathcal{K}_{P_B}$ corresponding to a set $S\in B\setminus [n+1]$. 
		By Proposition~\ref{prop: nested set complex}, $\{v_{S_1}, \ldots, v_{S_k}\}$ is a simplex in  $\mathcal{K}_{P_B}$ if for any $S_i, S_j \in \{S_1, \ldots, S_k\}$, either $S_i\subset S_j$, $S_j\subset S_i$. 
		From now on we denote $\mathcal{K}_{P_B}$ by $\mathcal{K}$. 
		We construct a $k$-Massey product $\langle \alpha_1, \ldots, \alpha_k\rangle \subset  H^*(\mathcal{Z}_\mathcal{K})$ by explicitly defining $J_i$ and $\alpha_i\in  \widetilde{H}^0(\mathcal{K}_{J_i})$. 
		Then we edge contract $\mathcal{K}_{J_1\cup \cdots \cup J_k}$ to a simplicial complex that by Construction~\ref{cons: simplicial complex for n-Massey} has a non-trivial Massey product. 
		For $k<n$, let 
		\begin{align*}
			&\alpha_1\in \widetilde{H}^0(\mathcal{K}_{v_{\{1\}}, v_{\{2\}}}) \\
			&\alpha_i\in  \widetilde{H}^0(\mathcal{K}_{v_{\{1, \ldots, i, k+1\}}, v_{\{2, \ldots, i+1\}} }) \text{ for } 1<i<k\\
			&\alpha_k\in  \widetilde{H}^0(\mathcal{K}_{v_{\{1, \ldots, k+1\}}, v_{\{1, \ldots, k, k+2\}} })
		\end{align*}
		so $\alpha_i$ corresponds to a class $\alpha_i\in H^3(\mathcal{Z}_\mathcal{K})$.
		In this case $|J_i|=2$, so there are no edges to contract. 
		Let $\mathcal{\hat{K}}=\mathcal{K}_{J_1\cup \cdots \cup J_k}$. 
		There is no edge $\{v_{\{1\}}, v_{\{2, \ldots, i+1\}}\}$ in $\mathcal{\hat{K}}$ for $v_{\{1\}}\in J_1$ and $v_{\{2, \ldots, i+1\}}\in J_i$ since $\{1\}\not\subset \{2, \ldots, i+1\}$. 
		Also there is no edge $\{v_{\{1, \ldots, i, k+1\}}, v_{\{2,\ldots, j+1\}} \}$ nor $\{v_{\{1, \ldots, i, k+1\}}, v_{\{1,\ldots, k, k+2\}} \}$ for $v_{\{1, \ldots, i, k+1\}}\in J_i$, $v_{\{2,\ldots, j+1\}}\in J_j$ with $1<i<j<k$ and  $v_{\{1,\ldots, k, k+2\}}\in J_k$. 
		All the other edges are in $\mathcal{\hat{K}}$. 
		That is,  $\{v_{\{1\}}, v_{\{1,\ldots,i, k+1\}} \}\in \mathcal{\hat{K}}$ and $\{v_{\{1\}}, v_{\{1,\ldots,k, k+2\}} \}\in \mathcal{\hat{K}}$ for $v_{\{1\}}\in J_1$, $v_{\{1,\ldots,i, k+1\}}\in J_i$ for any $i\leqslant k$ and $v_{\{1,\ldots,k, k+2\}}\in J_k$. 
		Similarly, $\{v_{\{1,\ldots,i, k+1\}} , v_{\{1,\ldots,j, k+1\}} \}\in \mathcal{\hat{K}}$ for $v_{\{1,\ldots,i, k+1\}}\in J_i$ and $v_{\{1,\ldots,j, k+1\}}\in J_j$ with $1<i<j\leqslant k$. 
		Also $\{v_{\{2,\ldots,i+1\}} , v_{S_j} \}\in \mathcal{\hat{K}}$ for $v_{\{2,\ldots,i+1\}}\in J_i$ and any $v_{S_j}\in J_j$ with $1\leqslant i<j\leqslant k$. 
		Therefore $\mathcal{\hat{K}}$ is obtained from the join $\mathcal{K}_{J_1}*\cdots *\mathcal{K}_{J_k}$ by star deleting at the edges
		$\{v_{\{1\}}, v_{\{2, \ldots, i+1\}}\}$, $\{v_{\{1, \ldots, i, k+1\}}, v_{\{2,\ldots, j+1\}} \}$ and $\{v_{\{1, \ldots, i, k+1\}}, v_{\{1,\ldots, k, k+2\}} \}$ 
		for $v_{\{1\}}\in J_1$, $v_{\{2, \ldots, i+1\}}\in J_i$, $v_{\{1, \ldots, i, k+1\}}\in J_i$, $v_{\{2,\ldots, j+1\}}\in J_j$ with $1<i<j<k$ and  $v_{\{1,\ldots, k, k+2\}}\in J_k$. 
		Hence by Theorem~\ref{thm: joins}, the  Massey product $\langle \alpha_1, \ldots, \alpha_k \rangle \subset H^{2k+2}(\mathcal{Z}_{\mathcal{K}})$ is non-trivial.

		For $k=n$,  
		let 
		\begin{align*}
			&\alpha_1\in \widetilde{H}^0(\mathcal{K}_{v_{\{1\}}, v_{\{2, \ldots, n+1\}}, v_{\{3, \ldots, n+1\}} }) \\
			&\alpha_i\in  \widetilde{H}^0(\mathcal{K}_{v_{\{1, \ldots, i\}}, v_{\{2, \ldots, i\}}, v_{\{3, \ldots, i+1\}} }) \text{ for } 1<i<n \\
			&\alpha_n\in  \widetilde{H}^0(\mathcal{K}_{v_{\{1, \ldots, n\}}, v_{\{2, \ldots, n\}}, v_{\{1, 3, \ldots, n+1\}} }).
		\end{align*}
		Since $|J_i|=3$ for every $i\in \{1,\ldots, n\}$, we will perform $n$ edge contractions in order to obtain a simplicial complex $\mathcal{\hat{K}}$ on $2n$ vertices. 
		There is an edge $\{v_{\{2, \ldots, n+1\}}, v_{\{3, \ldots, n+1\}}\}\in \mathcal{K}_{J_1}$ since $\{3, \ldots, n+1\}\subset \{2, \ldots, n+1\}$.
		Also there are edges $\{v_{\{1, \ldots, i\}}, v_{\{2, \ldots, i\}}\}\in \mathcal{K}_{J_i}$ for $1<i\leqslant n$. 
		Since $P$ is a simple polytope, $\mathcal{K}$ is a triangulation of a sphere so the contraction of these edges satisfy the link condition. 
		Let $\mathcal{\hat{K}}$ be obtained from 
		$\mathcal{K}_{J_1\cup \cdots \cup J_n}$
		by contracting these $n$ edges.
		Then as in the case when $k<n$, $\mathcal{\hat{K}}$ is a simplicial complex obtained from the join of $n$ pairs of disjoint vertices by star deletions as described by Construction~\ref{thm: joins}.
		Hence by Theorem~\ref{thm: joins}, there is a non-trivial $k$-Massey product in $H^*(\mathcal{Z}_{\mathcal{\hat{K}}})$. 
		By Theorem~\ref{thm: edge contractions Massey}, the Massey product $\langle \alpha_1, \ldots, \alpha_k \rangle \subset H^*(\mathcal{Z}_\mathcal{K})$ is also non-trivial.
	\end{proof}

	A similar technique to that used in Proposition~\ref{prop: permutahedra} can be applied to other simple polytopes. 
	An example is the family of stellohedra: graph associahedra corresponding to star graphs, which are graphs with a central vertex  and edges attaching every other vertex to the central one.
	It was shown in \cite[Theorem~3]{Limonchenko} that there are $3$-Massey products on $3$-dimensional classes in $H^*(\mathcal{Z}_P)$ when $P$ is a $3$-dimensional stellohedron, using the classification in \cite{DenhamSuciu, LowestDegreeClassification}.
	By applying Theorems~\ref{thm: joins} and \ref{thm: edge contractions Massey}, we generalise that result by constructing  non-trivial $n$-Massey products in moment-angle manifolds over $n$-dimensional stellohedron.
	
	\begin{prop}\label{prop: stellohedra}
		When $P$ is the $n$-dimensional stellohedron, $H^*(\mathcal{Z}_P)$ has a non-trivial $n$-Massey product.
	\end{prop}
	\begin{proof}
		As in Proposition~\ref{prop: permutahedra}, we construct $\alpha_i\in \widetilde{H}^0(\mathcal{K}_{J_i})$ where $\mathcal{K}=\mathcal{K}_{P}$.
		Let the star graph associated to $P$ be labelled so that the central vertex is $1$ and the other vertices are $2, \ldots, n+1$. The building set for $P$ is
		\small{\[
		\{ \{1\}, \ldots, \{n+1\}, \{1,2\}, \ldots, \{1,n+1\}, \ldots, \{1,\ldots, n\}, \{1, \ldots, n-1, n+1\}, [n+1] \}.
		\]}
		Let 
		\begin{align*}
			&\alpha_1\in \widetilde{H}^0(\mathcal{K}_{v_{\{2\}}, v_{\{1\}} }) \\
			&\alpha_i\in  \widetilde{H}^0(\mathcal{K}_{ v_{\{1,\ldots, i\}}, v_{\{1,3,\ldots, i+2\}}, v_{\{1,4, \ldots, i+2\}} }) \text{ for } 1<i<n \\
			&\alpha_n\in  \widetilde{H}^0(\mathcal{K}_{v_{\{1,3\}}, v_{\{3\}}, v_{\{1, 2, 4, \ldots, n+1\}} }).
		\end{align*}
		By contracting the edges $\{v_{\{1,3,\ldots, i+2\}}, v_{\{1,4, \ldots, i+2\}} \}\in \mathcal{K}_{J_i}$ for $1<i<n$ and the edge $\{v_{\{1,3\}}, v_{\{3\}}\} \in \mathcal{K}_{J_n}$, we obtain a simplicial complex $\mathcal{\hat{K}}$ that is constructed from the join of $n$ disjoint points by star deletions as in Construction~\ref{cons: simplicial complex for n-Massey}.
	\end{proof}
	
	Propositions~\ref{prop: permutahedra} and \ref{prop: stellohedra} reiterate that the moment-angle manifolds associated to permutahedra and stellohedra are non-formal \cite{Limonchenko_multiwedge}.
	Also, the families of permutahedra and stellohedra are examples of geometric direct families of polytopes, whose moment-angle manifolds are studied in \cite{directfamilies_BL}. Hence, Propositions~\ref{prop: permutahedra} and \ref{prop: stellohedra} answer Problems 5.32, 5.34 and 5.35 in \cite{directfamilies_BL}, which ask if there are geometric direct families of polytopes with non-trivial higher Massey products.

	\subsection{Non-trivial indeterminacy and permutahedra}
	
	Massey products with non-trivial indeterminacy can be found in moment-angle manifolds. 
	We illustrate this in moment-angle manifolds associated with permutahedra.
	We first construct an example of a $4$-Massey product with non-trivial indeterminacy in a moment-angle complex using Theorem~\ref{thm: infinite Massey products with indeterminacy}, then find a full-subcomplex of a permutahedron that edge contracts to this example and apply Theorem~\ref{thm: edge contractions Massey}.
	
	\begin{exmp}\label{ex: non-trivial indeterminacy in 4-Massey}
		Let $\mathcal{K}^i$ be a pair of disjoint points $J_i=\{i, i'\}$ for $i=1,\ldots, 4$ and define
		\[
		\mathcal{K}=\sd_{\{1,2'\}} \sd_{\{1,3'\}} \sd_{\{2,3'\}} \sd_{\{2,4'\}} \sd_{\{3,4'\}} \sd_{\{1',2'\}} \sd_{\{1',3'\}} \mathcal{K}^1 * \mathcal{K}^2*\mathcal{K}^3*\mathcal{K}^4.
		\]
		Let $\alpha_i=[a_i]$ and $a_i=\chi_i\in C^0(\mathcal{K}_{J_i})$.
		By Theorem~\ref{thm: infinite Massey products with indeterminacy}, $\langle \alpha_1, \alpha_2, \alpha_3, \alpha_4 \rangle \subset H^*(\mathcal{Z}_\mathcal{K})$ is non-trivial with non-trivial indeterminacy.
	\end{exmp}
	
	\begin{prop}
		There are non-trivial Massey products with non-trivial indeterminacy in moment-angle manifolds corresponding to permutahedra.
	\end{prop}
	\begin{proof}
		Let $P$ be the $5$-dimensional permutahedron. 
		Denote $\mathcal{K_P}$ by $\mathcal{K}$.
		Recall that by Proposition~\ref{prop: nested set complex}, $\{v_{S_1}, \ldots, v_{S_k}\}$ is a simplex in  $\mathcal{K}$ if for any $S_i, S_j \in \{S_1, \ldots, S_k\}$, either $S_i\subset S_j$ or $S_j\subset S_i$. 
		Let
		\begin{align*}
			&J_1=\{v_{\{1\}}, v_{\{2\}}, v_{\{2,5\}}, v_{\{5\}}\} \\
			&J_2=\{v_{\{1,2\}}, v_{\{3\}}\} \\
			&J_3=\{v_{\{1,2,3\}}, v_{\{2,3\}}, v_{\{3,4\}} \} \\
			&J_4=\{v_{\{1,2,3,4\}}, v_{\{2,3,4\}}, v_{\{1,3,4,5\}}\}
		\end{align*}
		and let $\alpha_i\in \widetilde{H}^0(\mathcal{K}_{J_i})$.
		Let $\mathcal{\hat{K}}$ be the simplicial complex in Example~\ref{ex: non-trivial indeterminacy in 4-Massey}, so there is a non-trivial $4$-Massey product in $H^*(\mathcal{Z}_{\mathcal{\hat{K}}})$.
		Consider the map $\varphi\colon \mathcal{K}\to \mathcal{\hat{K}}$ that takes $J_i\mapsto \{i,i'\}$ by contracting the edges
		\begin{align*}
			\{v_{\{2\}},v_{\{2,5\}} \}, \{v_{\{2,5\}},v_{\{5\}}\} &\mapsto 1' \\
			\{v_{\{1,2,3\}}, v_{\{2,3\}}\} &\mapsto 3\\
			\{ v_{\{1,2,3,4\}}, v_{\{2,3,4\}}\} &\mapsto 4.
		\end{align*}
		Since $\mathcal{K}$ is a triangulation of a sphere, these edge contractions satisfy the link condition.
		Therefore by Theorem~\ref{thm: edge contractions Massey}, there is a non-trivial $4$-Massey product $\langle \alpha_1, \alpha_2, \alpha_3, \alpha_4 \rangle \subset H^*(\mathcal{Z}_\mathcal{K})$ for $\alpha_i\in H^0(\mathcal{K}_i)$, and this $4$-Massey product has non-trivial indeterminacy.
	\end{proof}
	
	This example of a non-trivial $n$-Massey product with non-trivial indeterminacy can be reproduced in any $(n+1)$-dimensional permutahedron.

	\bibliographystyle{abbrv} 
	\bibliography{paperbibliography} 

\end{document}